\documentclass[12pt,reqno,a4paper]{amsart}
\usepackage{fullpage}
\usepackage[utf8]{inputenc}
\usepackage[T1]{fontenc}
\usepackage{hyperref} 
\usepackage{enumitem}
\usepackage{amsmath,amssymb,amsthm,color}
\usepackage{todonotes,framed}
\usepackage{nameref}
\usepackage{enumitem}
\usepackage{boxedminipage}

\def\OO{\mathcal{O}}
\def\oo{\mathbf{o}}

\usepackage{mdframed}

\def\revision#1{{\color{red}#1}}

\def\abs#1{|#1|}
\def\norm#1#2{\|#1\|_{#2}}
\def\set#1#2{\big\{#1:#2\big\}}
\def\prod#1#2{{\langle #1,#2\rangle}}

\def\heff{\boldsymbol{h}_{\operatorname{eff}}}
\def\heffPi{\boldsymbol{h}_{\operatorname{eff},\boldsymbol{\Pi}}}

\def\diam#1{\operatorname{diam} (#1)}

\def\diver{\operatorname{div}}

\def\ppi{\boldsymbol{\pi}}

\def\mid#1#2{#1^{#2 + 1/2}}
\def\midh#1#2{#1_h^{#2 + 1/2}}
\def\gentime#1#2{#1^{#2,\boldsymbol{\Theta}}}
\def\gentimeh#1#2{#1_h^{#2,\boldsymbol{\Theta}}}
\def\midh#1#2{#1_h^{#2 + 1/2}}

\def\dt#1#2{{\mathop{\mathrm{d}_t#1}}^{#2+1}}
\def\dth#1#2{ \mathop{\mathrm{d}_t} { #1_h^{#2+1} } }
\def\dtshort{\mathop{\mathrm{d}_t}}

\def\R{\mathbb{R}}

\def\0{\boldsymbol{0}}

\def\TT{\mathcal T}
\def\NN{\mathcal N}

\def\H{\boldsymbol{H}}

\def\P{\mathbb{P}}

\def\Cex{\ell_{\rm ex}^2}
\def\Cmesh{\kappa}

\def\hh{\boldsymbol{h}}
\def\mm{\boldsymbol{m}}
\def\nn{\boldsymbol{n}}

\def\uu{\boldsymbol{u}}
\def\vv{\boldsymbol{v}}
\def\ww{\boldsymbol{w}}

\def\BB{\boldsymbol{B}}
\def\ff{\boldsymbol{f}}
\def\vv{\boldsymbol{v}}
\def\ww{\boldsymbol{w}}
\def\xx{\boldsymbol{x}}
\def\zz{\boldsymbol{z}}

\def\JJ{\boldsymbol{\mathcal{J}}}
\def\RR{\boldsymbol{R}}

\def\zzeta{\boldsymbol{\zeta}}
\def\vvarphi{\boldsymbol{\varphi}}
\def\ppsi{\boldsymbol{\psi}}
\def\eeta{\boldsymbol{\eta}}

\def\nnu{\boldsymbol{\nu}}

\def\Ppi{\boldsymbol{\Pi}}
\def\Ttheta{\boldsymbol{\Theta}}
\def\II{\boldsymbol{\mathcal{I}}}

\def\L#1#2{L^#1 (#2) }
\def\LL#1#2{\boldsymbol{L}^#1 (#2) }
\def\H#1#2{H^#1 (#2) }
\def\HH#1#2{\boldsymbol{H}^#1 (#2) }

\def\Hcurl#1{\boldsymbol{H} \left( \operatorname{\bf curl} ;#1 \right)}

\def\cchi{\mathcal{X}}
\def\Kh#1{{\boldsymbol{\mathcal{K}}_h(#1)}}
\def\K#1{{\boldsymbol{\mathcal{K}}(#1)}}

\def\Vh{\boldsymbol{V}_{\!\!h}}

\def\Cinfty#1{\boldsymbol{C}^{\infty} \left(#1\right)}

\def\Int#1#2{\int_{#1}^{#2}}

\def\Sum#1#2{\sum_{#1}^{#2}}

\def\d#1{\mathop{\mathrm{d}#1}} 

\def\Mh{\boldsymbol{\mathcal{M}}_h}

\def\N{\mathbb{N}}
\def\R{\mathbb{R}}

\def\lr#1{(#1)}

\def\ltr{\!\left(t\right)}

\def\l0r{\!\left(0\right)}
\def\Trian{\mathcal{T}_h}

\def\weight{W}

\def\Cpi{C_{\mathrm{\bf \ppi}}}
\def\CPi{C_{\mathrm{\bf \Ppi}}}

\def\Cmh0{{\tilde{C}}_{{0}}}

\def\pih#1#2#3#4{\boldsymbol{\pi}_h^{#1}(#2;#3,#4)}
\def\Pih#1#2#3#4{\boldsymbol{\Pi}_h^{#1}(#2;#3,#4)}
\def\DPi#1#2{\boldsymbol{D}(#1,#2)}
\def\DPishort{\boldsymbol{D}}
\def\DPih#1#2{\boldsymbol{D}_h(#1,#2)}

\newcounter{statement}
\newenvironment{statement}[2][!]{%
\vskip3mm
\hrule
\hrule
\hrule
\vskip1mm
\noindent%
\refstepcounter{statement}%
\bf#2~\thestatement%
\ifthenelse{\equal{#1}{!}}{.\ }{~(#1).\ }%
\it%
}{%
\vskip1mm
\hrule
\hrule
\hrule
\vskip2mm
}

\newenvironment{theorem}[1][!]{\begin{statement}[#1]{Theorem}}{\end{statement}}
\newenvironment{lemma}[1][!]{\begin{statement}[#1]{Lemma}}{\end{statement}}
\newenvironment{definition}[1][!]{\begin{statement}[#1]{Definition}}{\end{statement}}
\newenvironment{proposition}[1][!]{\begin{statement}[#1]{Proposition}}{\end{statement}}

\newenvironment{remark}[1][!]{\begin{statement}[#1]{Remark}}{\end{statement}}
\newenvironment{algorithm}[1][!]{\begin{statement}[#1]{Algorithm}}{\end{statement}}

\mdfdefinestyle{theoremstyle}{
} 

\def\subsection#1{%
 \vskip2mm
 \refstepcounter{subsection}%
 {\bf\arabic{section}.\arabic{subsection}.~#1.~}%
}

\usepackage{graphicx}
\usepackage{subcaption}
\usepackage{siunitx}
\usepackage{pgfplots}
\pgfplotsset{compat=1.9}
\usepackage{pgfplotstable}

\title{Linear second-order IMEX-type integrator for the (eddy current) Landau--Lifshitz--Gilbert equation}

\author{Giovanni Di Fratta, Carl-Martin Pfeiler,
Dirk Praetorius, \\
Michele Ruggeri,
Bernhard Stiftner}

\address{TU Wien, Institute for Analysis and Scientific Computing, Wiedner Hauptstr. 8-10/E101/4, 1040 Vienna, Austria}
\email{giovanni.difratta@asc.tuwien.ac.at}
\email{carl-martin.pfeiler@asc.tuwien.ac.at}
\email{dirk.praetorius@asc.tuwien.ac.at}
\email{michele.ruggeri@asc.tuwien.ac.at}
\email{bernhard.stiftner@asc.tuwien.ac.at \quad \rm (corresponding author)}

\thanks{{\bf Acknowledgement.} The authors thankfully acknowledge support by the Austrian Science Fund (FWF) through the doctoral school \emph{Dissipation and dispersion in nonlinear PDEs} (grant W1245) and the special research program \emph{Taming complexity in partial differential systems} (grant SFB F65) as well as the Vienna Science and Technology Fund (WWTF) through the project \emph{Thermally controlled magnetization dynamics} (grant MA14-44).}

\makeatother

\def\revision#1{{\color{red}#1}}

\usepackage{fancyhdr}
\advance\footskip0.4cm
\advance\textheight-0.4cm
\calclayout
\begin{document}


\begin{abstract}
Combining ideas from [Alouges et al.\ (Numer.\ Math., 128, 2014)] and [Praetorius et al.\ (Comput. Math. Appl., 2017)], we propose a numerical algorithm for the integration of the nonlinear and time-dependent Landau--Lifshitz--Gilbert (LLG) equation which is unconditionally convergent, formally (almost) second-order in time, and requires only the solution of one linear system per time-step.
Only the exchange contribution is integrated implicitly in time, while the lower-order contributions like the computationally expensive stray field are treated explicitly in time.
Then, we extend the scheme to the coupled system of the Landau--Lifshitz--Gilbert equation with the eddy current approximation of Maxwell equations (ELLG).
Unlike existing schemes for this system, the new integrator is unconditionally convergent, (almost) second-order in time, and requires only the solution of two linear systems per time-step.
\end{abstract}


\maketitle
\thispagestyle{fancy}

\section{Introduction}
\subsection{State of the art}
Nowadays, the study of magnetization processes in magnetic materials and the development of fast and reliable tools to perform large-scale micromagnetic simulations are the focus of considerable research as they play a fundamental role in the design of many technological devices.
Applications to magnetic recording, in which the external field can change fast so that the hysteresis properties are not accurately described by a static approach, require accurate numerical methods to study the dynamics of the magnetization distribution.
In this context, a well-established model to describe the time evolution of the magnetization in ferromagnetic materials is the Landau--Lifshitz--Gilbert equation (LLG) introduced in~\cite{ll1935,gilbert1955}.

Recently, the numerical integration of LLG has been the subject of many mathematical studies; see, e.g., the review articles~\cite{kp2006,garciaCervera2007,cimrak2008b}, the monograph~\cite{prohl2001} and the references therein.
The main challenges concern the strong nonlinearity of the problem, a nonconvex unit-length constraint which enforces the solutions to take their pointwise values in the unit sphere, an intrinsic energy law, which combines conservative and dissipative effects and should be preserved by the numerical scheme, as well as the presence of nonlocal field contributions, which prescribe the (possibly nonlinear) coupling with other partial differential equations (PDEs).
One important aspect of the research is related to the development of unconditionally convergent methods, for which the numerical analysis does not require to impose any CFL-type condition on the relation of the spatial and temporal discretization parameters.

The seminal works~\cite{bp2006,alouges2008a} propose numerical time-marching schemes, based on lowest-order finite elements in space, which are proven to be unconditionally convergent towards a weak solution of LLG in the sense of~\cite{as1992}.
The implicit midpoint scheme of~\cite{bp2006} is formally of second order in time. It inherently preserves some of the fundamental properties of LLG, such as the pointwise constraint (at the nodes of the mesh) and the energy law. However, it requires the solution of a nonlinear system of equations per time-step.
The tangent plane scheme of~\cite{alouges2008a} is based on an equivalent reformulation of LLG, which gives rise to a variational formulation in the tangent space of the current magnetization state. It requires only the solution of one linear system per time-step and employs the nodal projection at each time-step to enforce the pointwise constraint at a discrete level. 
For an implicit-explicit approach for the full effective field, we refer to~\cite{akt2012,bffgpprs2014}.
Moreover, extensions for the discretization of the coupling of LLG with the full Maxwell equations, the eddy current equation, a balance law for the linear momentum (magnetostriction), and a spin diffusion equation for the spin accumulation have been considered in~\cite{lt2013,bppr2013,ahpprs2014,lppt2015,bpp2015,ft2017}. Formally, the tangent plane scheme and its aforementioned extensions are of first order in time. The
 nodal projection step 
can be omitted and this leads to an additional consistency error which is also (formally) first-order in time~\cite{ahpprs2014}. For this projection-free variant, the recent work~\cite{ft2017} derives rigorous \emph{a~priori} error estimates which are of first order in time and space.

A tangent plane scheme with an improved convergence order in time is introduced and analyzed in~\cite{akst2014}.
Like~\cite{alouges2008a}, the proposed method is based on a predictor-corrector approach which combines a linear reformulation of the equation with the use of the nodal projection for the numerical treatment of the pointwise constraint. However, the variational formulation for the linear update is designed in such a way that the scheme has a consistency error of order $2-\varepsilon$ for any $0<\varepsilon<1$.
For this reason, this method is named \emph{(almost) second-order tangent plane scheme} in~\cite{akst2014}.

\subsection{Contributions of the present work}
In this work, 
we propose a threefold extension of the improved tangent plane scheme from~\cite{akst2014}:

{\large$\bullet$}
During the design and the implementation of a micromagnetic code, one of the main issues concerns the computation of the nonlocal magnetostatic interactions. In many situations, it turns out to be the most time-consuming part of micromagnetic simulations~\cite{aesds2013}.
To cope with this problem, we follow the approach of~\cite{prs2016} and propose an implicit-explicit treatment for the lower-order effective field contributions. Then, only one expensive stray field computation per time-step needs to be carried out.
Nevertheless, our time-stepping preserves the (almost) second-order convergence in time of the scheme as well as the unconditional convergence result.

{\large$\bullet$}
The discovery of the giant magnetoresistance (GMR) effect in~\cite{bbfvnpecfc1988,bgsz1989} determined a breakthrough in magnetic hard disk storage capacity and encouraged several extensions of the micromagnetic model, which aim to describe the effect of spin-polarized currents on magnetic materials.
The most used approaches involve extended forms of LLG, in which the classical energy-based effective field is augmented by additional terms in order to take into account the spin transfer torque effect; see~\cite{slonczewski1996,zl2004,tnms2005}.
In this work, we extend the abstract setting, the proposed algorithm, and the convergence result so that the aforementioned extended forms of LLG are covered by our analysis.

{\large$\bullet$}
For the treatment of systems in which LLG is bidirectionally coupled with another time-dependent PDE, the works~\cite{bppr2013,ahpprs2014,lppt2015,bpp2015} propose integrators which completely decouple the time integration of LLG and the coupled equation.
This treatment is very attractive in terms of computational cost and applicability of the scheme, since existing implementations, including solvers and preconditioning strategies for the building blocks of the system, can easily be reused.
We show how such an approach can also be adopted for 
the improved tangent plane scheme. Combined with a second-order method for the coupled equation, this leads to algorithms of global (almost) second order, for which the convergence result can be generalized. 
As an illustrative example, we analyze the coupling of LLG with 
eddy currents (ELLG), which is of relevant interest in several concrete applications~\cite{bmms2002,scmw2004,hseskdf2005}.

\subsection{Outline}
The remainder of this work is organized as follows:
We conclude this section by recalling the notation used throughout the paper.
In Section~\ref{section:llg} and in Section~\ref{section:ellg}, we present the problem formulation and introduce the numerical algorithm. Then, we state the convergence result for pure LLG and ELLG, respectively.
In Section~\ref{section:derivation}, for the convenience of the reader, we reformulate the argument of~\cite{akst2014} and propose a formal derivation of the (almost) second-order algorithm.
In Section~\ref{section:proof:maintheorem} and Section~\ref{section:proof:maintheorem_ellg}, we prove the main results for pure LLG (Theorem~\ref{thm:maintheorem}) and ELLG (Theorem~\ref{thm:maintheorem_ellg}), respectively.
Finally, Section~\ref{section:numerics} is devoted to numerical experiments.
\subsection{General notation}
Throughout this work, we use the standard notations for Lebesgue, Sobolev, and Bochner spaces and norms.
For any domain $U$, we denote the scalar product in $L^2(U)$ by $\prod{\cdot}{\cdot}_U$ and the corresponding norm by $\norm{\cdot}{U}$.
Vector-valued functions (as well as the corresponding function spaces) are indicated by bold letters.
To abbreviate notation in proofs, we write $A \lesssim B$ when $A \le c B$ for some generic constant $c > 0$, which is clear from the context and always independent of the discretization parameters.
Moreover, $A \simeq B$ abbreviates $A \lesssim B \lesssim A$.
\section{Landau--Lifshitz--Gilbert equation} 
\label{section:llg}

\subsection{LLG and weak solutions}%
\label{subsection:model_LLG}
For a bounded Lipschitz domain $\omega\subset\R^3$, the Gilbert form of LLG reads
\begin{subequations}\label{eq:llg}
\begin{align}
\label{eq:llg1} \partial_t \mm &= -\mm\times \big( \heff(\mm) \, + \, \Ppi(\mm) \big) 
+ \alpha\,\mm\times\partial_t \mm &&\text{ in } \omega_T := (0,T)\times\omega \,, \\
\label{eq:llg2} \partial_{\nn}\mm &= \0 &&\text{ on } (0,T)\times\partial\omega \,, \\
\label{eq:llg3} \mm(0) &= \mm^0 &&\text{ in } \omega,
\end{align}
where $T>0$ is the final time, $\alpha \in (0,1]$ is the Gilbert damping constant, and $\mm^0\in\HH{1}{\omega}$ is the initial data with $|\mm^0|=1$ a.e.\ in $\omega$. The so-called effective field reads
\begin{align}\label{eq:heffdef}
\heff(\mm) = \Cex\Delta\mm + \boldsymbol{\pi}(\mm) + \ff,
\end{align}
\end{subequations}
where $\ell_{\rm ex}>0$ is the exchange length, while the linear, self-adjoint, and bounded operator $\boldsymbol{\pi}: \LL{2}{\omega} \rightarrow \LL{2}{\omega}$ collects the lower-order terms such as stray field and uniaxial anisotropy, and $\ff \in C^1([0,T],\LL{2}{\omega})$ is the applied field. We note that
\begin{align}\label{eq:energy_LLG}
 \heff(\mm) = -\frac{\delta \mathcal{E}_{\textrm{LLG}}(\mm)}{\delta \mm},
 \text{ where }~
 \mathcal{E}_{\textrm{LLG}}(\mm) 
 = \frac{\Cex}{2} \norm{\nabla \mm}{\omega}^2 
 \!-\! \frac12 \prod{\boldsymbol{\pi}(\mm)}{\mm}_{\omega} 
 \!-\! \prod{\ff}{\mm}_{\omega}
\end{align}
is the micromagnetic energy functional.
Further dissipative (e.g., spintronic) effects such as the Slonczewski contribution~\cite{slonczewski1996} or the Zhang--Li contribution~\cite{zl2004,tnms2005} are collected in the (not necessarily linear) operator $\boldsymbol{\Pi}: \HH{1}{\omega} \cap \LL{\infty}{\omega} \rightarrow \LL{2}{\omega}$. 

Note that solutions of LLG~\eqref{eq:llg} formally satisfy $0 = \partial_t \mm\cdot\mm = \frac12\,\partial_t|\mm|^2$ and hence $|\mm|=1$ a.e.\ in $\omega_T$. As in~\cite{as1992}, we use the following notion of weak solutions of~\eqref{eq:llg}:

\begin{definition}\label{def::weaksolution}
With the foregoing notation, a function $\mm:\omega_T\to\R^3$ is called a weak solution of LLG~\eqref{eq:llg}, if the following properties \ref{item:weak_llg1}--\ref{item:weak_llg4} are satisfied:
\begin{enumerate}[label=\rm{(\roman*)}]
\item\label{item:weak_llg1} $\mm\in\HH{1}{\omega_T} \, \cap \, \L{\infty}{0,T;\HH{1}{\omega}}$ with $|\mm|=1$ a.e.\ in $\omega_T$;
\item\label{item:weak_llg3} $\mm(0)=\mm^0$ in the sense of traces;
\item\label{item:weak_llg2} for all $\vvarphi\in\HH{1}{\omega_T}$, it holds that
\begin{align}\label{eq:variational_llg}
 \int_0^T\prod{\partial_t \mm}{\vvarphi}_{\omega}\,\d{t} 
 &= \Cex\int_0^T\prod{\mm\times\nabla\mm}{\nabla\vvarphi}_{\omega}\,\d{t} 
 - \int_0^T\prod{\mm\times\ff}{\vvarphi}_{\omega}\,\d{t} 
 \notag \\ &\quad 
 - \int_0^T\prod{\mm\times\ppi(\mm)}{\vvarphi}_{\omega}\,\d{t} 
 - \int_0^T\prod{\mm\times\Ppi(\mm)}{\vvarphi}_{\omega}\,\d{t}  
 \notag \\ &\quad 
 + \alpha \, \int_0^T\prod{\mm\times\partial_t \mm}{\vvarphi}_{\omega}\,\d{t};
\end{align}
\item\label{item:weak_llg4} for almost all $\tau \in (0,T)$, it holds that
\begin{align}\label{eq:strongerenergyestimate_LLG}
 \!\!\! \mathcal{E}_{\rm{LLG}}(\mm(\tau)) 
 + \alpha\!\int_0^{\tau} \! \norm{\partial_t \mm}{\omega}^2\,\d{t}
 +\! \int_0^{\tau} \!\! \prod{\partial_t \ff}{\mm}_{\omega} \d{t} 
 -\! \int_0^{\tau} \!\! \prod{\Ppi(\mm)}{\partial_t \mm}_{\omega} \d{t} 
\le \mathcal{E}_{\rm{LLG}}(\mm^0).
\end{align} 
\end{enumerate}
\end{definition}

\subsection{Time discretization}%
\label{section:discretization:time}
Let $N \in \N$ and $k:= T/N$. Consider the uniform time-steps $t_i := ik$ for $i=0,\dots,N$ and the corresponding midpoints $t_{i+1/2} := (t_{i+1}+t_{i})/2$. For a Banach space $\BB$, e.g., $\BB \in \{ \LL{2}{\omega}, \HH{1}{\omega} \}$, and a sequence $\left( \vvarphi^i \right)_{i=-1}^N \in \BB$, define the mean value and the discrete time derivative by
\begin{subequations}
\begin{align}\label{eq:discreteobjects}
 \mid{\vvarphi}{i} := \frac{\vvarphi^{i+1} + \vvarphi^i}{2} 
 \quad \textrm{and} \quad 
 \dt{\vvarphi}{i} := \frac{\vvarphi^{i+1} - \vvarphi^i}{k} 
 \quad \textrm{for} \quad i = 0,\dots,N-1.
\end{align}
For $\Ttheta^k \in \R^{N \times 3}$, define the two-step approach by
\begin{align}\label{eq:gen_timestepping}
 \gentime{\vvarphi}{i} 
 := \Ttheta_{i3}^k \vvarphi^{i+1} + \Ttheta_{i2}^k \vvarphi^{i} + \Ttheta_{i1}^k   \vvarphi^{i-1}
 \quad \textrm{for} \quad i = 0,\dots,N-1.
\end{align}
\end{subequations}
For $i=0,\dots,N-1$ and $t \in [t_i,t_{i+1})$, define
\begin{align}\label{eq:discretefunctions}
\begin{split}
 & \vvarphi_{k}^{=} \ltr := \vvarphi^{i-1}, \quad
 \vvarphi_{k}^-\ltr := \vvarphi^i, \quad 
 \vvarphi_{k}^+\ltr := \vvarphi^{i+1}, \quad 
 \overline{\vvarphi}_{k}\ltr := \mid{\vvarphi}{i}, \quad 
  \\
 &\vvarphi_{k}^{\Ttheta} \ltr := \gentime{\vvarphi}{i}, 
 \quad \textrm{and} \quad
 \vvarphi_{k}\ltr := \vvarphi^{i+1} \frac{t-t_i}{t_{i+1}-t_i} + \vvarphi^i \frac{t_{i+1}-t}{t_{i+1}-t_i}.
\end{split}
\end{align}
Note that $\vvarphi_{k}^{=}, \vvarphi_{k}^{-}, \vvarphi_{k}^{+}, \overline{\vvarphi}_{k}, \vvarphi_{k}^{\Ttheta} \in \L{2}{0,T;\BB}$ and $\vvarphi_k \in \H{1}{0,T;\BB}$ with $\partial_t \vvarphi_k (t) = \dt{\vvarphi}{i}$ for $t \in (t_i,t_{i+1})$.

\subsection{Space discretization}%
\label{subsection:spacediscretization}
Let $\Trian$ be a conforming triangulation of $\omega$ into compact tetrahedra $K \in \Trian$
which is $\Cmesh$-quasi-uniform, i.e., the global mesh-size $h>0$ assures that
\begin{align}\label{eq:quasiuniform}
\Cmesh^{-1} \, h \leq \abs{ K }^{1/3} \leq \diam{K} \leq \Cmesh \, h \quad \textrm{for all } K \in \Trian.
\end{align}
We suppose that $\Trian$ is weakly acute, i.e., the dihedral angles of all elements are $\leq \pi / 2$. Define the space of $\Trian$-piecewise affine and globally continuous functions by
\begin{align*}
 \Vh := (V_h)^3 \subset \HH{1}{\omega}
 \quad \text{with} \quad
 V_h := \{ v_h \in C \lr{\omega} : v_h|_{K} \in \mathcal{P}^1 \lr{K} \textrm{ for all } K \in \Trian \}. 
\end{align*}
Let $\NN_h$ be the set of nodes of $\Trian$. Recall that $|\mm|=1$ and $\mm \cdot\partial_t \mm =0$. 
To mimic the latter properties, we define the set of discrete admissible magnetizations on $\omega$ by
\begin{align}
\Mh := \set{\ppsi_h\in\Vh}{|\ppsi_h(\zz_h)|=1\text{ for all } \zz_h\in\NN_h}
\end{align}
as well as the discrete tangent space
\begin{align}\label{eq::defKpsih}
\Kh{\ppsi_h} := \set{\vvarphi_h\in\Vh}{\ppsi_h(\zz_h)\cdot\vvarphi_h(\zz_h)=0\text{ for all } \zz_h\in\NN_h} \,
\quad
\text{for all } \ppsi_h\in\Mh.
\end{align}

\subsection{Almost second-order tangent plane scheme}%
\label{section:llg:algorithm}
In this section, we formulate our numerical integrator which is analyzed below.
For each time-step $t_i$, 
it approximates $\mm(t_i)$ by $\mm_h^i \in \Mh$ and $\partial_t \mm(t_i)$ by $\vv_h^i \in \Kh{\mm_h^i}$. For a practical formulation, suppose (computable) operators  
\begin{align*}
 \ppi_h,\Ppi_h : \Mh \to \LL{2}{\omega}
 \quad\text{and}\quad
 \ppi_h^i, \Ppi_h^i : \Vh \times \Mh \times \Mh \rightarrow \LL{2}{\omega}
\end{align*}
such that $\ppi_h(\mm_h^i) \approx \ppi(\mm_h^i)$ and $\Ppi_h(\mm_h^i) \approx \Ppi(\mm_h^i)$ as well as
$\ppi_h^i(\vv_h^i;\mm_h^i,\mm_h^{i-1}) \approx \ppi(\mm_h^{i+1/2})$ 
and $\Ppi_h^i(\vv_h^i;\mm_h^i,\mm_h^{i-1}) \approx \Ppi(\mm_h^{i+1/2})$. We suppose that the operators $\ppi_h^i$ and $\Ppi_h^i$ are affine in $\vv_h^i$ (cf.~Remark~\ref{remark:algorithm:llg}).
We extend the construction in~\eqref{eq:discretefunctions} to the operator sequences $(\ppi_h^i)_{i=0}^{N-1}$ and $(\Ppi_h^i)_{i=0}^{N-1}$ via 
\begin{align}\label{eq:discreteoperator}
\boldsymbol{\pi}_{hk}^-(t) := \boldsymbol{\pi}_{h}^i 
\quad \textrm{and} \quad
\boldsymbol{\Pi}_{hk}^-(t) := \boldsymbol{\Pi}_{h}^i
\quad \textrm{for all } t \in [t_i,t_{i+1}).
\end{align}
Finally, we require two stabilizations $M: \R_{>0} \to \R_{>0}$ and $\rho: \R_{>0} \to \R_{>0}$ such that
\begin{subequations}\label{eq:rhoM}
\begin{align}\label{eq:assumption:M}
 \lim_{k\to0}M(k) = \infty 
 \quad &\text{with} \quad
 \lim_{k\to0}M(k)k = 0,
 \\ \label{eq:assumption:rho}
 \lim_{k\to0} \rho(k) = 0
 \quad &\text{with} \quad
 \lim_{k\to0} k^{-1}\rho(k) = \infty.
\end{align}
The canonical choices are 
\begin{align}\label{eq:assumption:choice}
 \rho(k) := |k\,\log k|
 \quad\text{and}\quad
  M(k) := |k\,\log k|^{-1}, 
\end{align}
\end{subequations}
We proceed as in \cite[p.415]{akst2014} and define the weight function
\begin{align}\label{eq:weightingtilde}
\weight_{M(k)}(s) :=
\begin{cases}
\displaystyle\alpha + \frac{k}{2} \, \min\{s,M(k)\} \quad&\text{for } s\ge 0 \,,\\
\displaystyle\frac{\alpha}{1+\frac{k}{2\alpha}\min\{-s,M(k)\}} \quad&\text{for } s<0 \,.
\end{cases}
\end{align}
With these preparations, our numerical integrator reads as follows:

\begin{algorithm}[Implicit-explicit tangent plane scheme]\label{alg:abtps}
{\bf Input:} Approximation $\mm_h^{-1} := \mm_h^0\in\Mh$ of initial magnetization.\\
{\bf Loop:} For $i=0,\dots,N-1$, iterate the following steps {\rm(a)}--{\rm(c)}:\\
{\rm (a)} Compute the discrete function
\begin{align}
\lambda_h^i := - \Cex \, |\nabla \mm_h^i |^2 + \big( \ff(t_i) + \ppi_h(\mm_h^i) + \Ppi_h(\mm_h^i) \big) \cdot \mm_h^i.
\end{align}
{\rm(b)} Find $\vv_h^{i}\in\Kh{\mm_h^i}$ such that, for all $\vvarphi_h\in\Kh{\mm_h^i}$, it holds that
\begin{align}\label{eq:tps2_llg_variational}
\begin{split}
&\prod{\weight_{M(k)}(\lambda_h^i) \, \vv_h^i}{\vvarphi_h}_{\omega} + \prod{\mm_h^i\times\vv_h^i}{\vvarphi_h}_{\omega} + 
\frac{\Cex}{2} \, k \, (1+\rho(k)) \, \prod{\nabla\vv_h^i}{\nabla\vvarphi_h}_{\omega} 
\\
&= - \Cex \, \prod{\nabla\mm_h^i}{\nabla\vvarphi_h}_{\omega}
 + \prod{\ff(t_{i+1/2})}{\vvarphi_h}_{\omega}
 + \prod{\pih{i}{\vv_h^i}{\mm_h^i}{\mm_h^{i-1}}}{\vvarphi_{h}}_{\omega}
\\& \qquad 
 + \prod{\Pih{i}{\vv_h^i}{\mm_h^i}{\mm_h^{i-1}}}{\vvarphi_h}_{\omega}  \,.
\end{split}
\end{align}
{\rm(c)} Define $\mm_h^{i+1}\in\Mh$ by 
\begin{align}\label{alg:abtpsupdate}
\mm_h^{i+1}(\zz_h) := \frac{\mm_h^i(\zz_h)+k\vv_h^i(\zz_h)}{|\mm_h^i(\zz_h)+k\vv_h^i(\zz_h)|} \quad\text{ for all } \zz_h\in\NN_h.
\end{align}
{\bf Output:} Approximations $\mm_h^i\approx\mm(t_i)$ for all $i=0,\dots,N$.
\qed
\end{algorithm}

\begin{remark}\label{remark:algorithm:llg}
{\rm(i)} The variational formulation~\eqref{eq:tps2_llg_variational} gives rise to a linear system for $\vv_h^i$. 
However, in particular for stray field computations, the part of the resulting system matrix which corresponds to the $\ppi_h^i$-contribution may be fully populated (resp.\ not explicitly available for hybrid FEM-BEM methods~\cite{fk1990}). To deal with this issue, we can, on the one hand, employ the fixpoint iteration~\eqref{eq:fixedpoint_llg_variational} and, on the other hand, choose $\ppi_h^i$ and $\Ppi_h^i$ independent of $\vv_h^i$. For the latter, see {\rm (ii)} below for a choice that does not impair the convergence order.

{\rm (ii)} Natural choices for the operators $\ppi_h^i$ and $\Ppi_h^i$ are the following: The chain rule gives rise to an operator $\DPi{\mm}{\partial_t \mm} := \partial_t[\Ppi(\mm)]$ which is linear in the second argument. Suppose that $\boldsymbol{D}_h$ is a discretization of $\boldsymbol{D}$.
For $i=0$, we follow~\cite[Algorithm 2]{akst2014} and choose 
\begin{subequations}\label{eq:llg_choice1}
\begin{align}\label{eq:llg_choice1a}
\pih{i}{\vv_h^i}{\mm_h^i}{\mm_h^{i-1}} &:= \ppi_h(\mm_h^i) + 
\frac{k}{2} \ppi_h(\vv_h^{i}),\\\label{eq:llg_choice1b}
\Pih{i}{\vv_h^i}{\mm_h^i}{\mm_h^{i-1}} &:= \Ppi_h(\mm_h^i) + 
\frac{k}{2} \DPih{\mm_h^i}{\vv_h^i};
\end{align}
\end{subequations}
cf.\ Proposition~\ref{proposition:variational}{\rm (i)}.
For the subsequent time-steps $i\geq1$ and unlike~\cite{akst2014}, we set 
\begin{subequations}\label{eq:llg_choice2}
\begin{align}\label{eq:llg_choice2a}
\pih{i}{\vv_h^i}{\mm_h^i}{\mm_h^{i-1}} &:= \frac{3}{2} \ppi_h(\mm_h^i) - 
\frac{1}{2} \ppi_h(\mm_h^{i-1}),
\\\label{eq:llg_choice2b}
\Pih{i}{\vv_h^i}{\mm_h^i}{\mm_h^{i-1}} &:= \Ppi_h(\mm_h^i) + 
\frac{1}{2} \DPih{\mm_h^i}{\mm_h^i} - \frac{1}{2} \DPih{\mm_h^i}{\mm_h^{i-1}};
\end{align}
\end{subequations}
cf.\ Proposition~\ref{proposition:variational}{\rm (ii)}.
In this way, the right-hand side of~\eqref{eq:tps2_llg_variational} is independent of $\vv_h^i$.

{\rm(iii)} The choice of $\rho$ in~\eqref{eq:assumption:choice} leads to formally almost second-order in time convergence of Algorithm~\ref{alg:abtps}. In principle, it suffices to choose a sufficiently large constant $M(k)\equiv M > 0$; for details see Proposition~\ref{proposition:variational}.

{\rm(iv)} Theorem~\ref{thm:maintheorem} states well-posedness of Algorithm~\ref{alg:abtps} as well as unconditional convergence towards a weak solution of LLG in the sense of Definition~\ref{def::weaksolution}. Moreover, the proof of Theorem~\ref{thm:maintheorem}{\rm (i)} provides a convergent fixpoint solver for the first time-step $i=0$.

{\rm (v)} Unlike the stabilized scheme, we note that the non-stabilized scheme $\rho(k) \equiv 0$ requires a CFL-type coupling $k = \oo(h)$ 
for the convergence proof; see also Remark~\ref{remark:CFL}.
\end{remark}

\subsection{Main theorem for LLG integration}%
\label{section:maintheorem:llg}
To formulate the main result which generalizes~\cite[Theorem 2]{akst2014}, we require the following assumptions:

{\large$\bullet$}
Let $\TT_h$ satisfy the assumptions of Section~\ref{subsection:spacediscretization}. 

{\large$\bullet$}
Let $\mm_h^0\in\Mh$ satisfy
\begin{align}\label{eq:initconvergence_m0}
\mm_h^0 \rightharpoonup \mm^0 \quad \textrm{in } \HH{1}{\omega} \quad\text{as } h \rightarrow 0.
\end{align}

{\large$\bullet$}
For all $i \in \{0,\dots,N-1\}$, let $\ppi_h^i$ satisfy, for all $\ppsi_h, \widetilde{\ppsi}_h \in \Vh$ and all $\vvarphi_{h},\widetilde{\vvarphi}_{h} \in \Mh$, 
\begin{subequations}\label{eq:assumptions:pi}
the Lipschitz-type continuity
\begin{align}\label{eq:pihlipschitz}
 \norm{\pih{i}{\ppsi_h}{\vvarphi_h}{\widetilde{\vvarphi}_h} - \pih{i}{\widetilde{\ppsi}_h}{\vvarphi_h}{\widetilde{\vvarphi}_h}}{\omega} 
 \le \Cpi \, k \, \norm{\ppsi_h\!-\!\widetilde{\ppsi}_h}{\omega}
\end{align}
and the stability-type estimate
\begin{align}\label{eq:pihstability}
 \norm{\pih{i}{\ppsi_h}{\vvarphi_h}{\widetilde{\vvarphi}_h}}{\omega} 
\le \Cpi \, \big( k \, \norm{\ppsi_h}{\omega} + \norm{\vvarphi_h}{\omega} + \norm{\widetilde{\vvarphi}_h}{\omega} \big). \end{align}
Moreover, for all sequences $\ppsi_{hk}$ in $\L{2}{0,T;\Vh}$ and $\vvarphi_{hk}$, $\widetilde{\vvarphi}_{hk}$ in $\L{2}{0,T;\Mh}$ with $\ppsi_{hk} \rightharpoonup \ppsi$ and $\vvarphi_{hk}, \widetilde{\vvarphi}_{hk} \rightarrow \vvarphi$ in $\LL{2}{\omega_T}$, we suppose consistency
\begin{align}\label{eq:pihkconvergence}
 \ppi_{hk}^-(\ppsi_{hk};\vvarphi_{hk},\widetilde{\vvarphi}_{hk}) 
 \rightharpoonup \ppi(\vvarphi) 
 \quad \textrm{in } \LL{2}{\omega_{T}} \quad\text{as } h,k \rightarrow 0.
\end{align}
\end{subequations}

{\large$\bullet$}
For all $i \in \{0,\dots,N-1\}$, let $\Ppi_h^i$ satisfy, for all $\ppsi_h, \widetilde{\ppsi}_h \in \Vh$ and all $\vvarphi_{h},\widetilde{\vvarphi}_{h} \in \Mh$, 
\begin{subequations}\label{eq:assumptions:Pi}
the Lipschitz-type continuity
\begin{align}\label{eq:Pihlipschitz}
 \norm{\Pih{i}{\ppsi_h}{\vvarphi_h}{\widetilde{\vvarphi}_h} 
 - \Pih{i}{\widetilde{\ppsi}_h}{\vvarphi_h}{\widetilde{\vvarphi}_h}}{\omega} 
 \le \CPi \, k \, \norm{\ppsi_h-\widetilde{\ppsi}_h}{\HH{1}{\omega}}
\end{align}
and the stability-type estimate
\begin{align}\label{eq:Pihstability}
 \prod{\Pih{i}{\ppsi_h}{\vvarphi_h}{\widetilde{\vvarphi}_h}}{ \ppsi_h }_{\omega} 
 \le \CPi \, \norm{\ppsi_h}{\omega} \big( k \, \norm{\ppsi_h}{\HH{1}{\omega}} +
 \norm{\vvarphi_h}{\HH{1}{\omega}} 
+  \norm{\widetilde{\vvarphi}_h}{\HH{1}{\omega}} \big).
\end{align}
Moreover, for all sequences $\ppsi_{hk}$ in $\L{2}{0,T;\Vh}$ and $\vvarphi_{hk}$, $\widetilde{\vvarphi}_{hk}$ in $\L{2}{0,T;\Mh}$ with $\ppsi_{hk} \rightharpoonup \ppsi$ and $\vvarphi_{hk}, \widetilde{\vvarphi}_{hk} \rightarrow \vvarphi$ in $\LL{2}{\omega_T}$, we suppose consistency
\begin{align}
\label{eq:Pihkconvergence}
 \Ppi_{hk}^-(\ppsi_{hk};\vvarphi_{hk},\widetilde{\vvarphi}_{hk}) & \rightharpoonup \Ppi(\vvarphi)
 \quad \textrm{in } \LL{2}{\omega_T} \quad \text{as } h,k \rightarrow 0.
\end{align}
\end{subequations}

\begin{theorem}\label{thm:maintheorem}
{\rm(i)} If the Lipschitz-type estimates~\eqref{eq:pihlipschitz} and~\eqref{eq:Pihlipschitz} are satisfied, then there exists $k_0>0$ such that for all $0<k<k_0$ the discrete variational formulation~\eqref{eq:tps2_llg_variational} admits a unique solution $\vv_h^i \in \Kh{\mm_h^i}$. In particular, Algorithm~\ref{alg:abtps} is well-defined.

{\rm(ii)} Suppose that all preceding assumptions~\eqref{eq:initconvergence_m0}--\eqref{eq:assumptions:Pi} are satisfied. Let $\mm_{hk}$ be the postprocessed output~\eqref{eq:discretefunctions} of Algorithm~\ref{alg:abtps}.
Then, there exists $\mm \in \HH{1}{\omega_T} \cap \L{\infty}{0,T;\HH{1}{\omega}}$ which satisfies Definition~\ref{def::weaksolution}\ref{item:weak_llg1}--\ref{item:weak_llg2}, and a subsequence of $\mm_{hk}$ converges weakly in $\HH{1}{\omega_T}$ towards $\mm$ as $h,k\to0$.

{\rm(iii)} Under the assumptions of {\rm(ii)}, suppose that the convergence properties~\eqref{eq:initconvergence_m0} as well as~\eqref{eq:pihkconvergence} and~\eqref{eq:Pihkconvergence} hold with strong convergence. Moreover, suppose additionally that $\ff \in C^1([0,T],\LL{2}{\omega}) \cap C([0,T],\LL{3}{\omega})$ and that $\ppi$ is $\boldsymbol{L}^3$-stable, i.e.,
\begin{align}\label{eq:strongerboundedness}
\norm{\ppi(\vvarphi)}{\LL{3}{\omega}} \leq C_{\ppi}' \, \norm{\vvarphi}{\LL{3}{\omega}} \quad \textrm{for all } \vvarphi \in \HH{1}{\omega}.
\end{align}
Then, $\mm$ from {\rm(ii)} is a weak solution of LLG in the sense of Definition~\ref{def::weaksolution}\ref{item:weak_llg1}--\ref{item:weak_llg4}.
\end{theorem}

\newpage

\begin{remark}\label{remarks:llg}
{\rm (i)} If the solution $\mm \in \HH{1}{\omega_T} \cap \L{\infty}{0,T;\HH{1}{\omega}}$ of Definition~\ref{def::weaksolution} is unique, then the full sequence $\mm_{hk}$ converges weakly in $\HH{1}{\omega_T}$ towards $\mm$.

{\rm(ii)} The assumptions~\eqref{eq:assumptions:pi} on $\ppi_h$ (even with strong convergence) as well as~\eqref{eq:strongerboundedness} on $\ppi$ are satisfied, e.g., for uniaxial anisotropy and stray field. Possible stray field discretizations include hybrid FEM-BEM approaches~\cite{fk1990}; see~\cite{bffgpprs2014,prs2016}.

{\rm(iii)} The Lipschitz-type estimates~\eqref{eq:pihlipschitz} and~\eqref{eq:Pihlipschitz} are only used to prove that Algorithm~\ref{alg:abtps} is well-defined for sufficiently small $k>0$. The stability estimates~\eqref{eq:pihstability} and~\eqref{eq:Pihstability} are then used to prove some discrete energy estimate (Lemma~\ref{lemma:energy_ellg}). Finally, the consistency assumptions~\eqref{eq:pihkconvergence} and~\eqref{eq:Pihkconvergence} are used to show that the existing limit satisfies the variational formulation of Definition~\ref{def::weaksolution}\ref{item:weak_llg2}.

{\rm (iv)} The assumptions on $\ppi_h$ and $\Ppi_h$ from~\eqref{eq:assumptions:pi}--\eqref{eq:assumptions:Pi} are only exploited for $\vvarphi_{h} = \mm_{h}^i$, $\widetilde{\vvarphi}_{h} = \mm_{h}^{i-1}$, and $\vvarphi_{hk} = \mm_{hk}^-$, $\widetilde{\vvarphi}_{hk} = \mm_{hk}^=$, $\ppsi_{hk} = \vv_{hk}^-$. Moreover, only
~\eqref{eq:pihlipschitz} and~\eqref{eq:Pihlipschitz} require general $\ppsi_{h}, \widetilde{\ppsi}_h$, while
~\eqref{eq:pihstability} and~\eqref{eq:Pihstability} are only exploited for $\ppsi_h = \vv_h^i$.
\end{remark}

The proof of Theorem~\ref{thm:maintheorem} is postponed to Section~\ref{section:proof:maintheorem}.
Unique solvability of~\eqref{eq:tps2_llg_variational} will be proved with the Banach fixpoint theorem.  Theorem~\ref{thm:maintheorem}(ii)--(iii) will be proved through an energy argument which consists of the following three steps:
\begin{itemize}
\item bound the discrete energy (Lemma~\ref{lemma:energy});
\item extract weakly convergent subsequences having a common limit (Lemma~\ref{lemma:extractsubsequences});
\item verify that this limit is a weak solution in the sense of Definition~\ref{def::weaksolution}.
\end{itemize}

\section{Landau--Lifshitz--Gilbert equation with eddy currents} 
\label{section:ellg}

\subsection{ELLG and weak solutions}%
\label{subsection:model_ELLG}
Let $\Omega \subset \R^3$ be a bounded Lipschitz domain with $\Omega \supseteq \omega$ that represents a conducting body with its ferromagnetic part $\omega$. ELLG reads
\begin{subequations}\label{eq:ELLG}
\begin{align}
\partial_t \mm &= - \mm \times \big(\heff(\mm) + \Ppi(\mm) + \hh \big) 
+ \alpha \, \mm \times \partial_t \mm 
&&\hspace*{-3mm}
\textrm{ in } \omega_T := (0,T) \times \omega, \label{eq:ELLG1} 
\\
- \mu_0 \, \partial_t \mm &= \mu_0 \, \partial_t \hh +  \nabla \times  (\sigma^{-1} \nabla \times \hh ) 
&&\hspace*{-3mm}
\textrm{ in } \Omega_T := (0,T) \times \Omega, \label{eq:ELLG2} 
\\
\partial_{\mathbf{n}} \mm 
&= \boldsymbol{0} 
&&\hspace*{-3mm}
\textrm{ on } \left(0,T\right) \times \partial \omega, \label{eq:ELLG3} 
\\
\hspace*{-1.5mm}( \nabla \times \hh ) \times \mathbf{n} 
&= \boldsymbol{0}  
&&\hspace*{-3mm}
\textrm{ on } \left(0,T\right) \times \partial \Omega, \label{eq:ELLG4}
\\
(\mm, \hh) \lr{0} &= (\mm^0,\hh^0) 
&&\hspace*{-3mm}
\textrm{ in } \omega \times \Omega.
\label{eq:ELLG5}
\end{align}
\end{subequations}
Here, $\sigma \in \L{\infty}{\Omega}$ with $\sigma \ge \sigma_0 > 0$ is the conductivity of $\Omega$, and $\mu_0 > 0$ is the vacuum permeability. The initial condition $\hh^0 \in \Hcurl{\Omega}$ satisfies the compatibility conditions
\begin{align}\label{eq:hcurlinit}
\nabla \cdot ( \hh^0 + \chi_{\omega} \mm^0 ) = 0 \textrm{ in } \Omega \quad \textrm{and} \quad 
(\hh^0 + \chi_{\omega} \mm^0 ) \cdot \mathbf{n} = 0 \textrm{ on } \partial \Omega.
\end{align}
Unlike \cite[Definition 2.1]{lt2013}, we define the energy functional
\begin{align}\label{eq:energyfunctional_ellg}
\mathcal{E}_{\textrm{ELLG}}(\mm,\hh) := \frac{\Cex}{2} \, \norm{\nabla \mm}{\omega}^2 - \frac12 \prod{\boldsymbol{\pi}(\mm)}{\mm}_{\omega} 
- \prod{\ff}{\mm}_{\omega} +  \frac12 \norm{\hh}{\Omega}^2.
\end{align}
As in~\cite{as1992,lt2013}, we use the following notion of weak solutions of ELLG~\eqref{eq:ELLG}:

\begin{definition}\label{def:weak_ellg}
With the foregoing notation, the pair $(\mm,\hh)$ is called a weak solution of ELLG~\eqref{eq:ELLG}, if the following conditions \ref{item:weak_ellg1}--\ref{item:weak_ellg5} are satisfied:
\begin{enumerate}[label=\rm{(\roman*)}]
\item \label{item:weak_ellg1} $\mm \in \HH{1}{\omega_T} \, \cap \, \L{\infty}{0,T;\HH{1}{\omega}}$ with $\left|\mm \right| = 1$ a.e. in $\omega_T$;
\item \label{item:weak_ellg2} $\hh \in \H{1}{0,T;\LL{2}{\Omega}} \cap 
\L{\infty}{0,T;\boldsymbol{H}({\rm curl};\Omega)}$.
\item \label{item:weak_ellg3} $\mm \l0r = \mm^0$ and $\hh \l0r = \hh^0$ in the sense of traces;
\item \label{item:weak_ellg4} for all $\vvarphi\in\HH{1}{\omega_T}$ and for all $\zzeta \in \L{2}{0,T;\boldsymbol{H}({\rm curl},\Omega)}$, it holds that
\begin{subequations}\label{eq:variational_ellg}
\begin{align}\label{eq:variational_ellg1}
\Int{0}{T} \prod{\partial_t \mm}{\vvarphi}_{\omega} \d{t} 
&=\Cex \Int{0}{T} \prod{\mm \times \nabla  \mm}{\nabla \vvarphi}_{\omega} \d{t} 
 - \Int{0}{T} \prod{\mm \times \ff}{\vvarphi}_{\omega} \d{t}\notag \\
&\quad- \Int{0}{T} \prod{\mm \times \ppi(\mm)}{\vvarphi}_{\omega} \d{t} 
 - \int_0^T\prod{\mm\times\Ppi(\mm)}{\vvarphi}_{\omega}\,\d{t}  \notag \\
&\quad -\Int{0}{T} \prod{\mm  \times \hh}{\vvarphi}_{\omega} \d{t} + 
\alpha \Int{0}{T} \prod{\mm \times \partial_t \mm}{\vvarphi}_{\omega} \d{t}, \\
\label{eq:variational_ellg2}
-\mu_0 \Int{0}{T} \prod{\partial_t \mm}{\zzeta}_{\omega} \d{t} &=
\mu_0 \Int{0}{T} \prod{\partial_t \hh}{\zzeta}_{\Omega} \d{t} +  \Int{0}{T} \prod{\sigma^{-1} \nabla \times \hh}{\nabla \times \zzeta}_{\Omega} \d{t};
\end{align}
\end{subequations}
\item\label{item:weak_ellg5} for almost all $\tau \in (0,T)$, it holds that
\begin{align}\label{eq:strongerenergyestimate_ellg}
& \mathcal{E}_{{\rm ELLG}}(\mm(\tau),\hh(\tau)) + \alpha \Int{0}{\tau} \norm{\partial_t \mm}{\omega}^2 \d{t}
+ \int_0^{\tau} \prod{\partial_t \ff}{\mm}_{\omega} \d{t} \notag \\
& \quad + \frac{1}{\mu_0} \Int{0}{\tau} \norm{\sigma^{-1/2}\nabla \times \hh}{\Omega}^2 \d{t}
 - \int_0^{\tau} \prod{\Ppi(\mm)}{\partial_t \mm}_{\omega} \d{t} \leq \mathcal{E}_{{\rm ELLG}}(\mm^0,\hh^0).
\end{align}
\end{enumerate}
\end{definition}

\subsection{Discretization}%
\label{section:discretization:ellg}
We adopt the notation of Section~\ref{section:discretization:time}. Let $\Trian$ be a conforming triangulation of $\Omega$ into compact tetrahedra $K \in \Trian$. Suppose that $\Trian$ is $\Cmesh$-quasi-uniform (cf.~\eqref{eq:quasiuniform}). We suppose that $\Trian$ resolves $\omega$, i.e., 
\begin{align}
 \overline\omega = \bigcup_{K\in\Trian|_\omega}K,
 \quad\text{where}\quad
 \Trian|_\omega := \set{K\in\Trian}{K \subseteq \overline\omega}.
\end{align}
Let $\NN_h$ (resp.\ $\NN_h|_\omega$) be the set of nodes of $\Trian$ (resp.\ $\Trian|_\omega$). We suppose that $\Trian|_\omega$ is weakly acute and define $\Vh$, $\Mh$, and $\Kh{\cdot}$ with respect to $\Trian|_\omega$ as in Section~\ref{subsection:spacediscretization}. Finally, the space of N\'{e}d\'{e}lec edge elements of second type~\cite{nedelec1986} on $\Omega$ reads
\begin{align*}
\cchi_h := \{ \zzeta_h \in \Hcurl{\Omega} : \zzeta_h|_{K} \in \mathcal{P}^1 \lr{K} \textrm{ for all } K \in \Trian \} \subset \Hcurl{\Omega}.
\end{align*}

\subsection{Almost second-order tangent plane scheme}%
\label{section:ellg:algorithm}
In this section, we extend Algorithm~\ref{alg:abtps} to ELLG~\eqref{eq:ELLG}. More precisely, we combine Algorithm~\ref{alg:abtps} for the LLG-part with an implicit midpoint scheme for the eddy current part. 

\begin{algorithm}[Decoupled ELLG algorithm]\label{alg:abtps_ellg}
{\bf Input:} $\Ttheta^k \in \R^{N \times 3}$, approximations $\mm_h^{-1} := \mm_h^0 \in \Mh$ and $\hh_h^{-1} := \hh_h^0 \in \cchi_h$ of the initial values.
\\
{\bf Loop:} For $i=0,\dots,N-1$, iterate the following steps {\rm(a)}--{\rm(d)}:\\
{\rm (a)} Compute the discrete function
\begin{align}
\lambda_h^i := - \Cex \, |\nabla \mm_h^i |^2 
+ \big(\ff(t_i) + \ppi_h(\mm_h^i) + \Ppi_h(\mm_h^i) + \hh_h^i \big) \cdot \mm_h^i.
\end{align} 
{\rm (b)} Find $\vv_h^{i} \in \Kh{\mm_h^i}$ such that, for all $\vvarphi_h \in 
\Kh{\mm_h^i}$, it holds that
\begin{subequations}\label{eq:abtps_ellg}
\begin{align}\label{eq:tps2_ellg_variational1}
&\prod{\weight_M(\lambda_h^i)\vv_h^i}{\vvarphi_h}_{\omega} + \prod{\mm_h^i\times\vv_h^i}{\vvarphi_h}_{\omega} + 
\frac{\Cex}{2} \, k \, (1+\rho(k)) \, \prod{\nabla\vv_h^i}{\nabla\vvarphi_h}_{\omega} \notag \\
&= -\Cex\prod{\nabla\mm_h^i}{\nabla\vvarphi_h}_{\omega} 
+ \prod{\ff(t_{i+1/2})}{\vvarphi_h}_{\omega}
+ \prod{\pih{i}{\vv_h^i}{\mm_h^i}{\mm_h^{i-1}}}{\vvarphi_h}_{\omega}
\notag 
\\
& \quad +  \prod{\Pih{i}{\vv_h^i}{\mm_h^i}{\mm_h^{i-1}}}{\vvarphi_h}_{\omega}  + \prod{\gentimeh{\hh}{i}}{\vvarphi_h}_{\omega}.
\end{align}
{\rm (c)} Define $\mm_h^{i+1}\in\Mh$ by 
\begin{align}\label{alg:abtpsupdate_ellg}
\mm_h^{i+1}(\zz_h) := \frac{\mm_h^i(\zz_h)+k\vv_h^i(\zz_h)}{|\mm_h^i(\zz_h)+k\vv_h^i(\zz_h)|} \quad\text{ for all } \zz_h\in\NN_h.
\end{align}
{\rm(d)} Find $\hh_h^{i+1} \in \cchi_h$ such that, for all $\zzeta_h \in \cchi_h$, it holds that
\begin{align}\label{eq:tps2_ellg_variational2}
- \mu_0 \prod{\dth{\mm}{i}}{\zzeta_h}_{\omega} = \mu_0 \prod{\dth{\hh}{i}}{\zzeta_h}_{\Omega} + 
\prod{\sigma^{-1} \nabla \times \midh{\hh}{i}}{\nabla \times \zzeta_h}_{\Omega}.
\end{align}
\end{subequations}
{\bf Output:} Approximations $\mm_h^i \approx \mm(t_i)$ and $\hh_h^i \approx \hh(t_i)$ for all $i=0,\dots,N$.
\qed
\end{algorithm}

\begin{remark}\label{remark:adellg}
{\rm(i)} It depends on the choice of $\Ttheta^k$ (and hence of $\gentimeh{\hh}{i}$ in~\eqref{eq:tps2_ellg_variational1}) whether step~{\rm(b)--(d)} of Algorithm~\ref{alg:abtps_ellg} have to be solved simultaneously or sequentially.

{\rm(ii)} For the first time-step $i=0$, we choose $\Ttheta^k_{03} : = \Ttheta^k_{02} := 1/2$ and $\Ttheta^k_{01} = 0$, i.e.,
\begin{align}\label{eq:elle_ellg_choice1}
\gentimeh{\hh}{i} = \midh{\hh}{i} = \hh_h^{i} + \frac{k}{2} \dth{\hh}{i}.
\end{align}
Then, step~{\rm(b)--(d)} of Algorithm~\ref{alg:abtps_ellg} have to be solved simultaneously. Moreover, because of step~{\rm(c)}, the overall system is nonlinear in $\vv_h^i$. For the subsequent time-steps $i \geq 1$, we choose $\Ttheta^k_{i3} : = 0$, $\Ttheta^k_{i2} := 3/2$, and $\Ttheta^k_{i1} = -1/2$, i.e.,
\begin{align}\label{eq:elle_choice1}
\gentimeh{\hh}{i} = \frac32 \hh_h^{i} - \frac12 \hh_h^{i-1}.
\end{align}
Then, the numerical integrator decouples the time-stepping for LLG and eddy currents, and we need to solve only two linear systems per time-step.

{\rm(iii)} The choice of $\Ttheta^k_i$ for $i\ge1$ is motivated by the explicit Adams--Bashforth two-step method which is of second order in time. By choice~\eqref{eq:assumption:choice} of $\rho$, Algorithm~\ref{alg:abtps_ellg} then is formally of almost second order in time.

{\rm(iv)} In principle, we can replace $-\mu_0\prod{\dth{\mm}{i}}{\zzeta_h}_{\omega}$ 
in equation~\eqref{eq:tps2_ellg_variational2} of Algorithm~\ref{alg:abtps_ellg} by $- \mu_0 \prod{\vv_h^i}{\zzeta_h}_{\omega}$. Even in the implicit case~\eqref{eq:elle_ellg_choice1}, the resulting scheme is then linear in $\vv_h^i$ and $\midh{\hh}{i}$. However, Lemma~\ref{lemma::vformula} predicts that $\partial_t \mm = \vv + \mathcal{O}(k)$, and thus we only expect first-order convergence in time. This is also confirmed numerically in Section~\ref{example4}.

{\rm(v)} In practice, we solve~\eqref{eq:tps2_ellg_variational2} for the unknown $\nnu_h := \midh{\hh}{i} \in \cchi_h$, i.e., we compute the unique $\nnu_h \in \cchi_h$ such that, for all $\zzeta_h \in \cchi_h$, 
\begin{align}\label{eq:solvemidpoint_eddy}
\frac{2\mu_0}{k} \prod{\nnu_h}{\zzeta_h}_{\Omega} + 
\prod{\sigma^{-1} \nabla \times \nnu_h}{\nabla \times \zzeta_h}_{\Omega}
= 
- \mu_0 \prod{\dth{\mm}{i}}{\zzeta_h}_{\omega} + 
\frac{2\mu_0}{k} \prod{\hh_h^i}{\zzeta_h}_{\Omega}.
\end{align}
Then, $\hh_h^{i+1} := 2 \, \nnu_h - \hh_h^i$ solves~\eqref{eq:tps2_ellg_variational2}.

{\rm(vi)} Theorem~\ref{thm:maintheorem_ellg} states well-posedness of Algorithm~\ref{alg:abtps_ellg} as well as unconditional convergence towards a weak solution of ELLG in the sense of Definition~\ref{def:weak_ellg}. The proof of Theorem~\ref{thm:maintheorem_ellg}~{\rm (i)} provides a convergent fixpoint solver for the first (nonlinear) time-step $i = 0$.
\end{remark}

\subsection{Main theorem for ELLG integration}
This section states our convergence result for Algorithm~\ref{alg:abtps_ellg}. In~\cite{lt2013,lppt2015} and~\cite{bpp2015}, similar results are proved for a first-order tangent plane scheme for ELLG and for the Maxwell-LLG system, respectively. 

\begin{theorem}\label{thm:maintheorem_ellg}
{\rm(i)} Suppose that the Lipschitz-type estimates~\eqref{eq:pihlipschitz} and~\eqref{eq:Pihlipschitz} are satisfied. Suppose that the coupling parameter $\Ttheta^k \in \R^{N \times 3}$ satisfies
\begin{align}\label{eq:thetainfnone_ellg}
\max_{ij} \abs{\Ttheta^k_{ij}} \leq C_{\Ttheta} < \infty \quad \textrm{and} \quad
\sum\limits_{j=1}^3 \Ttheta^k_{ij} = 1 \quad \textrm{for all } i \in \{0,\dots,N-1\}.
\end{align}
Then, there exists $k_0>0$ such that for all $0<k<k_0$ the discrete variational formulation~\eqref{eq:abtps_ellg} admits a unique solution $(\vv_h^i,\hh_h^{i+1}) \in \Kh{\mm_h^i} \times \cchi_h$. In particular, Algorithm~\ref{alg:abtps_ellg} is well-defined.
 
{\rm(ii)} Suppose that~\eqref{eq:thetainfnone_ellg} as well as  all assumptions from Section~\ref{section:maintheorem:llg} are satisfied. Let $(\mm_{hk}, \hh_{hk})$ be the postprocessed output~\eqref{eq:discretefunctions} of Algorithm~\ref{alg:abtps_ellg}. 
Then, there exist $\mm \in \HH{1}{\omega_T} \cap \L{\infty}{0,T;\HH{1}{\omega}}$ and $\hh \in \H{1}{0,T;\LL{2}{\Omega}} \cap \L{\infty}{0,T;\boldsymbol{H}({\rm curl},\Omega)}$ which satisfy Definition~\ref{def:weak_ellg}\ref{item:weak_ellg1}--\ref{item:weak_ellg4}, and a subsequence of $(\mm_{hk},\hh_{hk})$ converges weakly in $\HH{1}{\omega_T}\times\LL{2}{\Omega_T}$ towards $(\mm,\hh)$ as $h,k\to0$.

{\rm(iii)} Under the assumptions of {\rm(ii)}, suppose that the convergence properties~\eqref{eq:initconvergence_m0} as well as~\eqref{eq:pihkconvergence} and~\eqref{eq:Pihkconvergence} hold with strong convergence. Moreover, suppose additionally that $\ff \in C^1([0,T],\LL{2}{\omega}) \cap  C([0,T],\LL{3}{\omega})$ and that $\ppi$ is $\boldsymbol{L}^3$-stable~\eqref{eq:strongerboundedness}. 
Finally, let the CFL condition $k = \oo(h^2)$ be satisfied.
Then, the pair $(\mm,\hh)$ from {\rm(ii)} is a weak solution of ELLG in the sense of Definition~\ref{def:weak_ellg}\ref{item:weak_ellg1}--\ref{item:weak_ellg5}.
\end{theorem}

\begin{remark}\label{remarks:ellg}
{\rm(i)} If the solution $(\mm,\hh)$ of Definition~\ref{def:weak_ellg} is unique, then the full sequence $(\mm_{hk},\hh_{hk})$ converges weakly in $\HH{1}{\omega_T}\times\LL{2}{\Omega_T}$ towards $(\mm,\hh)$.

{\rm(ii)} If the CFL condition $k = \oo(h^2)$ is satisfied (as required for Theorem~\ref{thm:maintheorem_ellg}{\rm(iii)}), then Theorem~\ref{thm:maintheorem_ellg}{\rm(ii)} holds also for vanishing stabilization $\rho(k)\equiv0$.

{\rm(iii)} The compatibility condition~\eqref{eq:hcurlinit} of the initial data is not exploited in the proof of Theorem~\ref{thm:maintheorem_ellg}. In particular, the discrete initial data do not have to satisfy any ``discrete compatibility condition''.

{\rm(iv)} If $\prod{\dth{\mm}{i}}{\zzeta_h}_{\omega}$ in~\eqref{eq:tps2_ellg_variational2} 
is replaced by $\prod{\vv_h^i}{\zzeta_h}_{\omega}$, then Theorem~\ref{thm:maintheorem_ellg}{\rm(iii)} holds without the CFL condition $k=\oo(h^2)$ which is only exploited in~\eqref{eq:crucialterm_ellg} below. However, as mentioned in Remark~\ref{remark:adellg}, the overall integrator then appears to be only of first order in time.
\end{remark}

The proof of Theorem~\ref{thm:maintheorem_ellg} is postponed to Section~\ref{section:proof:maintheorem_ellg}. For the LLG-part of ELLG~\eqref{eq:ELLG}, we follow our proof of Theorem~\ref{thm:maintheorem}. For the eddy current part, we adapt the techniques of~\cite{lt2013,lppt2015,bpp2015,prs2016} to the setting of Algorithm~\ref{alg:abtps_ellg}.

\section{Derivation of second-order tangent plane scheme} 
\label{section:derivation}

In this section, we adapt \cite{akst2014,prs2016} in order to motivate Algorithm~\ref{alg:abtps} and to underpin that it is of (almost) second order in time.
Since solutions $\mm$ to LLG satisfy $\mm \cdot \partial_t \mm = 0$ a.e.\ in $\omega$,
we define, for any $\ppsi\in\Cinfty{\overline{\omega}}$ with $|\ppsi|=1$ its tangent space
\begin{align}\label{eq:Kpsi}
\K{\ppsi} := \set{\vvarphi\in\Cinfty{\overline{\omega}}}{\ppsi\cdot\vvarphi=0\text{ in }\omega}
\end{align}
as well as the pointwise projection onto $\K{\ppsi}$ by
\begin{align}\label{eq::morthproj}
\P_{\ppsi}(\uu) := \uu - (\uu\cdot\ppsi)\ppsi
\quad \text{for all }\uu \in \Cinfty{\overline{\omega}}.
\end{align}

\begin{lemma}[{{\cite[p.~413]{akst2014}}}]\label{lemma::vformula}
For $\mm\in\Cinfty{\overline{\omega_T}}$ with $|\mm|=1$,  it holds that
\begin{align}\label{eq::vformula}
\frac{\mm(t) + k\vv(t)}{|\mm(t) + k\vv(t)|} = \mm(t+k) + \OO(k^3),
\quad\text{where}\quad
\vv(t) := \partial_t \mm(t) + \frac{k}{2} \, \P_{\mm(t)}(\partial_{tt}\mm(t))
\end{align}
for all $t \in [0,T-k]$.\qed
\end{lemma}

Recall the stabilizations $M(\cdot)$ and $\rho(\cdot)$ from~\eqref{eq:rhoM} as well as the weight function $\weight_{M(k)}(\cdot)$ from~\eqref{eq:weightingtilde}.
The following lemma is implicitly found in~\cite[p.415]{akst2014}. With~\eqref{eq:assumption:M}, it proves, in particular, that $\weight_{M(k)}(\cdot)\ge \alpha/2$, if $k>0$ is sufficiently small.

\begin{lemma}\label{lemma:weight}
The weight function $\weight_{M(k)}(\cdot)$ satisfies the following two assertions {\rm(i)}--{\rm(ii)}: \newline
{\rm (i)} $|\weight_{M(k)}(s)-\alpha| \le \frac{M(k)k}{2}$ for all $s\in\R$. \newline
{\rm (ii)} If $M(k)\ge B>0$ and $k<k_0:=\alpha/B$, then
\begin{align}\label{eq::weightapprox}
\Big| \, \alpha + \frac{k}{2}s - \weight_{M(k)}(s) \, \Big| \le \frac{B^2}{2 \alpha} \, k^2 \quad \textrm{for all } s\in[-B,B] .
\end{align}
\end{lemma}

\begin{proof}
For $s \geq 0$, {\rm (i)} follows immediately from the definition~\eqref{eq:weightingtilde}. For $s < 0$, we get that
\begin{align*}
|\weight_{M(k)}(s)-\alpha| \stackrel{\eqref{eq:weightingtilde}}{=} 
\frac{\frac{k}{2} \min\{-s,M(k)\}}{1+\frac{k}{2 \alpha}\min\{-s,M(k)\}} \leq \frac{k}{2} \min\{-s,M(k)\} \leq \frac{M(k)k}{2}.
\end{align*}
This proves {\rm (i)}. We come to the proof of {\rm (ii)}: If $0\le s\le B$, then $B\leq M$ yields that
\begin{align*}
\alpha + \frac{k}{2} s -  \weight_{M(k)}(s) \stackrel{\eqref{eq:weightingtilde}}{=} 0.
\end{align*}
If $-B\le s < 0$ and $j\in\N$, the $j$-th derivative of $\weight_{M(k)}(\cdot)$ reads
\begin{align*}
\weight_{M(k)}^{(j)}(s) =\frac{2\alpha^2k^jj!}{(2\alpha-ks)^{1+j}}\,.
\end{align*}
Since $k<k_0:=\alpha/B$, it holds that
\begin{align*}
q := \Big| \frac{ks}{2 \alpha} \Big| \leq \Big| \frac{s}{2 B} \Big| \leq \frac12
\quad \textrm{and thus} \quad
\sum\limits_{j=0}^{\infty} q^j = \frac{1}{1-q} \leq 2.
\end{align*}
The Taylor expansion at $s=0^-$ shows that
\begin{align*}
\Big| \alpha + \frac{k}{2}s - \weight_{M(k)}(s)\Big| = \Big| \sum_{j=2}^\infty \frac{k^j s^j}{2^j\alpha^{j-1}} \Big| = 
\Big| \frac{k^2 s^2}{4\alpha}\sum_{j=0}^\infty \Big(\frac{ks}{2\alpha}\Big)^j \Big| 
\le 2 \, \frac{k^2 s^2}{4 \alpha} \leq \frac{B^2 k^2}{2 \alpha}.
\end{align*}
This concludes the proof.
\end{proof}

The following proposition is the main result of this section. It shows that the consistency error of Algorithm~\ref{alg:abtps} is of order $\OO(k^2+k\rho(k))$. For $\Ppi = \0$ and $\ff$ constant in time,~\eqref{eq:cont_variational1} is implicitly contained in~\cite[Section 6]{akst2014}, while~\eqref{eq:cont_variational2} adapts some further ideas from~\cite{prs2016}.
Note that~\eqref{eq:cont_variational1} corresponds to~\eqref{eq:tps2_llg_variational} in combination with~\eqref{eq:llg_choice1} and requires the evaluation of $\ppi(\vv)$ and $\DPi{\mm}{\vv}$. In contrast to that,~\eqref{eq:cont_variational2} corresponds to~\eqref{eq:tps2_llg_variational} in combination with~\eqref{eq:llg_choice2} and avoids these evaluations at~$\vv$.

\begin{proposition}\label{proposition:variational}
Let $\DPi{\cdot}{\cdot}$ be obtained from $\DPi{\mm(t)}{\partial_t \mm(t)} := \partial_t[\Ppi(\mm(t))]$ and the chain rule.
For $\ppsi \in \Cinfty{\overline{\omega}}$ and $\ww, \vvarphi \in \HH{1}{\omega}$, define the bilinear form
\begin{subequations}\label{eq:cont_variational}
\begin{align}\label{eq:firstweakform1}
 \hspace*{-1.5mm}
 \mathbf{b} (\ppsi;\ww,\vvarphi) 
 &:= \langle \weight_{M(k)}\big(\lambda(\ppsi)\big) \ww,\vvarphi \rangle_{\omega} 
 + \prod{\ppsi\times\ww}{\vvarphi}_{\omega} 
 + \frac{\Cex}{2} \, k \, (1 + \rho(k)) \, \prod{\nabla\ww}{\nabla\vvarphi}_{\omega},
 \hspace*{-2mm}
 \\
 \intertext{where}
\label{eq:firstweakform2}
\lambda(\ppsi) &:= \big( \heff(\ppsi) + \Ppi(\ppsi) \big) \cdot \ppsi.
\end{align}
\end{subequations}
Let $\mm \in \Cinfty{\overline{\omega_T}}$ be a strong solution of LLG~\eqref{eq:llg} which satisfies
\begin{align}\label{eq:defBandM}
B := \norm{ \lambda(\mm) }{\L{\infty}{\omega_T}} \leq M(k) <\infty.
\end{align}
Then, $\vv$ from~\eqref{eq::vformula} satisfies the following two assertions {\rm (i)}--{\rm (ii)}:
\newline
\begin{subequations}
{\rm (i)} For $t \in [0,T-k]$, there exists $\RR_1 = \OO(k^2+ k \rho(k))$ such that, for all $\vvarphi\in\K{\mm}$,
 \begin{align}\label{eq:cont_variational1}
\begin{split}
& \mathbf{b} (\mm;\vv(t),\vvarphi) - \frac{k}{2} \, \prod{\ppi(\vv(t))}{\vvarphi}_{\omega}
- \frac{k}{2} \, \prod{\DPi{\mm(t)}{\vv(t)}}{\vvarphi}_{\omega} \\
&\quad= -\Cex\prod{\nabla\mm(t)}{\nabla\vvarphi}_{\omega} 
+ \prod{\ff\big( \, t + k/2 \, \big)}{\vvarphi}_{\omega} 
+ \prod{\ppi(\mm(t))}{\vvarphi}_{\omega} \\
& \quad\quad\quad
+ \prod{\Ppi(\mm(t))}{\vvarphi}_{\omega} \!+\! \prod{\RR_1}{\vvarphi}_{\omega}.
\end{split}
\end{align} 
\newline
{\rm (ii)} For $t \in [k,T-k]$, there exists $\RR_2 = \OO(k^2+ k\rho(k))$ such that, for all $\vvarphi\in\K{\mm}$,
 \begin{align}\label{eq:cont_variational2}
\mathbf{b} (\mm;\vv(t),\vvarphi)
&= -\Cex\prod{\nabla\mm(t)}{\nabla\vvarphi}_{\omega} 
+ \prod{\ff\big( \, t + k/2 \, \big)}{\vvarphi}_{\omega} 
+ \frac{3}{2} \,\prod{\ppi(\mm(t))}{\vvarphi}_{\omega} 
\notag \\
& \quad 
- \frac{1}{2} \, \prod{\ppi(\mm(t-k))}{\vvarphi}_{\omega} 
+ \prod{\Ppi(\mm(t))}{\vvarphi}_{\omega} + \frac{1}{2} \, \prod{\DPi{\mm(t)}{\mm(t)}}{\vvarphi}_{\omega} 
\notag \\
& \quad 
- \frac{1}{2} \, \prod{\DPi{\mm(t)}{\mm(t-k)}}{\vvarphi}_{\omega} 
+ \prod{\RR_2}{\vvarphi}_{\omega}.
\end{align} 
\end{subequations}
Moreover, it holds that
\begin{align}\label{eq:lambda:identity}
\lambda(\mm) = - \Cex |\nabla \mm|^2 + \big( \ff + \ppi(\mm) + \Ppi(\mm) \big) \cdot \mm.
\end{align}
\end{proposition}

\begin{proof}
The proof is split into the following three steps.

{\bf Step~1}.
For the proof of {\rm (i)}, we extend~\cite[Section 6]{akst2014} to our setting: Define
\begin{subequations}\label{eq:heffextended}
\begin{align*}
\heffPi(\mm) := \heff(\mm) + \Ppi(\mm).
\end{align*}
The chain rule yields that
\begin{align*}
\partial_t \big[ \heffPi(\mm) \big] = 
\partial_t \big[ \heff(\mm) \big] + 
\partial_t \big[ \Ppi(\mm) \big]
=
\Cex \Delta \partial_t \mm + \partial_t \ff \, + \ppi(\partial_t \mm) + \DPi{\mm}{\partial_t \mm}.
\end{align*}
\end{subequations}
Since $|\mm|=1$ a.e.,~\eqref{eq:llg1} is equivalent to
\begin{align}\label{eq:alternative_plain}
 \alpha \, \partial_t \mm + \mm \times \partial_t \mm 
 = \heffPi(\mm) - \big( \heffPi(\mm) \cdot \mm \big) \, \mm;
\end{align}
see, e.g.,~\cite[Lemma 1.2.1]{goldenits2012}.
Differentiating~\eqref{eq:alternative_plain} with respect to time, we obtain that
\begin{align*}
\begin{split}
\alpha \, \partial_{tt}\mm + \mm\times\partial_{tt}\mm &= 
\partial_t \big[ \heffPi(\mm) \big] - 
(\partial_t \big[ \heffPi(\mm) \big] \cdot\mm)\mm \\
&\quad - (\heffPi(\mm)\cdot\partial_t \mm)\mm - (\heffPi(\mm)\cdot\mm)\partial_t \mm \,.
\end{split}
\end{align*}
Since $\partial_t \mm \cdot \mm = 0$ in $\omega_T$, it holds that $\alpha \, \partial_t \mm + \mm \times \partial_t \mm = \P_{\mm}( \alpha \, \partial_t \mm + \mm \times \partial_t \mm )$. Therefore, we further obtain that
\begin{align}\label{eq:formalcomputation1}
\begin{split}
\alpha \, \vv + \mm\times\vv 
&\stackrel{\eqref{eq::vformula}}{=}  \P_{\mm}\Big( \alpha \, \partial_t \mm + \mm\times\partial_t \mm + \frac{k}{2}\alpha\,\partial_{tt}\mm + \frac{k}{2}\,\mm\times\partial_{tt}\mm \Big) \\
&\stackrel{\eqref{eq:alternative_plain}}{=}  
\P_{\mm}\Big( \heffPi(\mm) + 
\frac{k}{2} \, \partial_t \big[ \heffPi(\mm) \big] - 
\frac{k}{2} \, (\heffPi(\mm)\cdot\mm)\, \partial_t \mm \Big).
\end{split}
\end{align}
For $t\in[0,T-k]$, we multiply~\eqref{eq:formalcomputation1} by $\vvarphi\in\K{\mm}$ and integrate over $\omega$. We rearrange the terms and use that the definition~\eqref{eq::vformula} of $\vv$ yields that $\partial_t \mm = \vv + \OO(k)$. Then,
\begin{align*}
& \langle \alpha \, \vv(t), \vvarphi \rangle_{\omega} 
+ \prod{\mm(t)\times\vv(t)}{\vvarphi}_{\omega} 
+ \frac{k}{2} \, \langle \P_{\mm}\Big( \big( \heffPi(\mm(t))\cdot\mm(t) \big) \vv(t) \Big), \vvarphi \rangle_{\omega} 
\\&
\stackrel{\eqref{eq:formalcomputation1}}{=} 
\prod{\P_{\mm}\big( \heffPi(\mm(t)) \big)}{\vvarphi}_{\omega} 
+ \frac{k}{2} \, \prod{\P_{\mm} \big( \partial_t \big[ \heffPi(\mm(t)) \big] \big) }{\vvarphi}_{\omega}
+ \OO(k^2).
\end{align*}
Since any $\vvarphi \in \K{\mm}$ is invariant under the pointwise definition~\eqref{eq::morthproj}, the projection $\P_{\mm}$ can be omitted in the latter equation. By definition~\eqref{eq:firstweakform2} of $\lambda(\mm)$, it follows that 
\begin{align*}
& \langle \big( \alpha + \frac{k}{2} \lambda(\mm(t))\big) \vv(t), \vvarphi \rangle_{\omega} + 
\prod{\mm(t)\times\vv(t)}{\vvarphi}_{\omega} \\
&\qquad
= \prod{\heffPi(\mm(t))}{\vvarphi}_{\omega} 
+ \frac{k}{2} \, \prod{\partial_t \big[ \heffPi(\mm(t)) \big] }{\vvarphi}_{\omega}
+ \OO(k^2).
\end{align*}
Note that $\partial_{\nn} \vv = \partial_t \partial_{\nn} \mm + \OO(k) = \OO(k)$. Therefore,
integration by parts yields that
\begin{align*}
& \prod{\heffPi(\mm(t))}{\vvarphi}_{\omega} 
+  \frac{k}{2} \prod{\partial_t \big[ \heffPi(\mm(t)) \big] }{\vvarphi}_{\omega} 
+ \OO(k^2)
\notag 
\\&
\stackrel{\eqref{eq:heffextended}}{=} -\Cex\prod{\nabla\mm(t)}{\nabla\vvarphi}_{\omega} 
+ \prod{\ff\big( t \big)}{\vvarphi}_{\omega}  
+ \prod{\ppi(\mm(t))}{\vvarphi}_{\omega} 
+ \prod{\Ppi(\mm(t))}{\vvarphi}_{\omega}
\notag 
\\&\quad  
- \frac{\Cex}{2} \, k \, \prod{\nabla\vv(t)}{\nabla\vvarphi}_{\omega} 
+ 
\frac{k}{2} \, \prod{\partial_t \ff(t)}{\vvarphi}_{\omega}
+ \frac{k}{2} \, \prod{\ppi(\vv(t))}{\vvarphi}_{\omega} 
+ \frac{k}{2} \, \prod{ \DPi{\mm(t)}{\vv(t)}}{\vvarphi}_{\omega}.
\end{align*}
Together with $\ff(t) + \frac{k}{2} \, \partial_t \ff(t) = \ff(t+k/2) + \OO(k^2)$, the latter two equations prove that
\begin{align*}
 &\langle \big( \alpha + \frac{k}{2} \lambda(\mm(t))\big) \vv(t), \vvarphi \rangle_{\omega} + 
\prod{\mm(t)\times\vv(t)}{\vvarphi}_{\omega} 
\\
 &\qquad
 = - \frac{\Cex}{2} \, k \, \prod{\nabla\vv(t)}{\nabla\vvarphi}_{\omega}
 + \prod{\ff(t+k/2)}{\vvarphi}_\omega
 + \prod{\ppi(\mm(t))}{\vvarphi}_{\omega}
 + \prod{\Ppi(\mm(t))}{\vvarphi}_{\omega}
 \\&\qquad\qquad
 + \frac{k}{2} \, \prod{\ppi(\vv(t))}{\vvarphi}_{\omega} 
 + \frac{k}{2} \, \prod{ \DPi{\mm(t)}{\vv(t)}}{\vvarphi}_{\omega}
 + \OO(k^2).
\end{align*}
It remains to replace $\big( \alpha + \frac{k}{2} \, \lambda(\mm(t))\big)$ by $\weight_{M(k)}\big(\lambda(\mm(t))\big)$, which yields an additional error of $\OO(k^2)$ (see~\eqref{eq:defBandM} and Lemma~\ref{lemma:weight}(ii)), and to add the stabilization term $\frac{\Cex}{2} \, k \, \rho(k) \, \prod{\nabla\vv}{\nabla\vvarphi}_{\omega}$, which yields an additional error of $\OO(k\rho(k))$. This proves~(i).

{\bf Step~2.}
For the proof of~{\rm (ii)}, we first observe that an implicit Euler step satisfies that
\begin{align*}
\mm(t-k) + k \, \partial_t \mm(t) = \mm(t) + \OO(k^2) \,.
\end{align*}
Together with Lemma~\ref{lemma::vformula}, we obtain that
\begin{align*}
\vv(t) \stackrel{\eqref{eq::vformula}}{=} \partial_t \mm(t) + \OO(k) 
= \frac{\mm(t)-\mm(t-k)}{k} + \OO(k) \,.
\end{align*}
Since $\boldsymbol{\pi}$ is linear and bounded, it follows that
\begin{eqnarray*}\label{eq::firstab}
\frac{k}{2}\,\ppi(\vv(t)) = \frac{1}{2}\,\ppi(\mm(t)) - \frac{1}{2}\,\ppi(\mm(t-k)) + \OO(k^2) \,.
\end{eqnarray*}
Similarly, since $\DPishort(\cdot,\cdot)$ is linear in the second argument, we also obtain that
\begin{align*}
\frac{k}{2} \, \DPi{\mm(t)}{\vv(t)} =
\frac{1}{2} \, \DPi{\mm(t)}{\mm(t)} - 
\frac{1}{2} \, \DPi{\mm(t)}{\mm(t-k)} + \OO(k^2).
\end{align*}
Combining the latter two equations with {\rm (i)}, we conclude the proof of {\rm (ii)}.

{\bf Step~3.}
Note that $|\mm| = 1$ in $\omega$ yields that
$\Delta \mm \cdot \mm =  - | \nabla \mm |^2$. In particular,
$$
 \lambda(\mm) = \big( \heff(\mm)  + \Ppi(\mm)  \big) \cdot \mm 
 = - \Cex |\nabla \mm|^2 + \big( \ff + \ppi(\mm) + \Ppi(\mm) \big) \cdot \mm.
$$
This proves~\eqref{eq:lambda:identity} and concludes the proof.
\end{proof}

\begin{remark}\label{re::rhok}
{\rm(i)} For $\rho(k) = \OO(k)$, Proposition~\ref{proposition:variational} yields that $\RR_1,\RR_2 = \OO(k^2)$ and therefore second-order accuracy of the consistency error. In this case, however, the convergence result of Theorem~\ref{thm:maintheorem} 
requires the CFL-type condition $k = \oo(h)$. Instead, the choice $\rho(k) := k |\log(k)|$ requires no CFL-type condition and leads to $\rho(k)= \OO(k^{1-\varepsilon})$ and hence to $\RR_1,\RR_2 = \OO(k^{2-\varepsilon})$ for any $\varepsilon > 0$. For details, we refer to Remark~\ref{remark:CFL} 
below.

{\rm(ii)} In the proof of Proposition~\ref{proposition:variational}, $\big( \alpha + \frac{k}{2} \, \lambda(\mm)\big)$ is replaced by $\weight_{M(k)}\big(\lambda(\mm)\big)$ to ensure ellipticity of the bilinear form $\mathbf{b}(\mm;\cdot,\cdot)$.
\end{remark}

\section{Proof of Theorem~\ref{thm:maintheorem} (Numerical integration of LLG)} 
\label{section:proof:maintheorem}

\begin{proof}[Proof of Theorem~\ref{thm:maintheorem}{\rm(i)}]
Together with the Lax--Milgram theorem, we employ a fixpoint iteration in order to solve~\eqref{eq:tps2_llg_variational}
with $\eeta_{h}^{\ell} \approx \vv_h^i$: 
Let $\eeta_h^0 := \0$. For $\ell\in\N_0$, let $\eeta_{h}^{\ell+1} \in \Kh{\mm_h^i}$ solve, for all $\vvarphi_h \in \Kh{\mm_h^i}$,
\begin{align}\label{eq:fixedpoint_llg_variational}
\begin{split}
&\prod{\weight_{M(k)}(\lambda_h^i) \, \eeta_{h}^{\ell+1}}{\vvarphi_h}_{\omega} 
+ 
\prod{\mm_h^i\times \eeta_{h}^{\ell+1}}{\vvarphi_h}_{\omega} 
+ \frac{\Cex}{2} \, k \, (1+\rho(k))\prod{\nabla\eeta_{h}^{\ell+1}}{\nabla\vvarphi_h}_{\omega} 
\\& \quad
= -\Cex\prod{\nabla\mm_h^i}{\nabla\vvarphi_h}_{\omega} 
+ \prod{\ff(t_{i+1/2})}{\vvarphi_h}_{\omega} 
+ \prod{\pih{i}{\eeta_{h}^{\ell}}{\mm_h^i}{\mm_h^{i-1}}}{\vvarphi_{h}}_{\omega} 
\\& \quad \quad \quad
+ \prod{\Pih{i}{\eeta_{h}^{\ell}}{\mm_h^i}{\mm_h^{i-1}}}{\vvarphi_h}_{\omega}  \,.
\end{split}
\end{align}
We equip $\Kh{\mm_h^i}$ with the norm $|\!|\!|\eeta_h|\!|\!|^2 := (\alpha/4) \, \norm{\eeta_h}{\omega}^2 + (\Cex/2) \, k \, (1+\rho(k)) \, \norm{\nabla \eeta_h}{\omega}^2$. 
For sufficiently small $k$, Lemma~\ref{lemma:weight}(i) and~\eqref{eq:assumption:M} prove that $\weight_{M(k)}(\cdot) \ge \alpha/2$. 
Since $\rho(k) \geq 0$, the bilinear form on the left-hand side of~\eqref{eq:fixedpoint_llg_variational} is positive definite on $\Kh{\mm_h^i}$, i.e., the fixpoint iteration is well-defined.
We subtract~\eqref{eq:fixedpoint_llg_variational} for $\eeta_h^{\ell+1}$ from~\eqref{eq:fixedpoint_llg_variational} for $\eeta_h^{\ell}$ and test with $\vvarphi_h := \eeta_h^{\ell+1} - \eeta_h^{\ell} \in \Kh{\mm_h^i}$. 
With ~\eqref{eq:pihlipschitz} and~\eqref{eq:Pihlipschitz}, we get that
\begin{align}\label{eq:uselipschitzhere}
& \frac{\alpha}{2} \, \norm{\eeta_h^{\ell+1} - \eeta_h^{\ell}}{\omega}^2 + \frac{\Cex}{2} \, k \, (1+\rho(k)) \, \norm{\nabla \eeta_h^{\ell+1} - \nabla \eeta_h^{\ell}}{\omega}^2 
\notag \\
& \quad \stackrel{\eqref{eq:fixedpoint_llg_variational}}{\leq}  
\norm{\pih{i}{\eeta_h^{\ell}}{\mm_h^i}{\mm_h^{i-1}} - \pih{i}{\eeta_h^{\ell-1}}{\mm_h^i}{\mm_h^{i-1}}}{\omega} 
\norm{\eeta_h^{\ell+1} - \eeta_h^{\ell}}{\omega} \notag \\
& \quad\quad \quad  + \norm{\Pih{i}{\eeta_h^{\ell}}{\mm_h^i}{\mm_h^{i-1}} - 
\Pih{i}{\eeta_h^{\ell-1}}{\mm_h^i}{\mm_h^{i-1}}}{\omega} \norm{\eeta_h^{\ell+1} - \eeta_h^{\ell}}{\omega}  \notag \\
& \quad \,\,\leq\,\, (\Cpi + \CPi) \,k \,\norm{\eeta_h^{\ell} - \eeta_h^{\ell-1}}{\omega} \norm{\eeta_h^{\ell+1} - \eeta_h^{\ell}}{\omega} \notag \\
& \quad\quad \quad + \CPi \, k \, \norm{\nabla(\eeta_h^{\ell} - \eeta_h^{\ell-1})}{\omega} \norm{\eeta_h^{\ell+1} - \eeta_h^{\ell}}{\omega}.
\end{align}
For arbitrary $\delta>0$, the Young inequality thus yields that
\begin{align*}
& \min\Big\{1 \, , \, 2 - \frac{1}{\alpha\delta} \, (\Cpi + 2\,\CPi) \, k \Big\}
\, |\!|\!| \eeta_h^{\ell+1} - \eeta_h^{\ell} |\!|\!|^2
\notag \\
&\quad\le
\Big[\frac{\alpha}{2} - \frac{1}{4\delta} \, (\Cpi + 2\,\CPi) \, k \Big] \, \norm{\eeta_h^{\ell+1} - \eeta_h^{\ell}}{\omega}^2 
+ \frac{\Cex}{2} \, k \, (1+\rho(k)) \, \norm{\nabla \eeta_h^{\ell+1} - \nabla \eeta_h^{\ell}}{\omega}^2 
\notag \\
&\quad\le
\delta \, (\Cpi + \CPi) \,k \,\norm{\eeta_h^{\ell} - \eeta_h^{\ell-1}}{\omega}^2 
+ \delta \, \CPi \, k \, \norm{\nabla(\eeta_h^{\ell} - \eeta_h^{\ell-1})}{\omega}^2
\notag \\
&\quad\le \delta \, \max\Big\{\frac{4}{\alpha} \, (\Cpi + \CPi) \,k \, , \, 2 \, \frac{\CPi}{\Cex} \Big\} 
\, |\!|\!| \eeta_h^{\ell} - \eeta_h^{\ell-1} |\!|\!|^2
 \quad \textrm{for all } \ell \in \N.
\end{align*}
For sufficiently small $\delta$ and $k$, the iteration is thus a contraction with respect to $|\!|\!|\cdot|\!|\!|$.
The Banach fixpoint theorem yields the existence and uniqueness of the solution $\vv_h^i \in \Kh{\mm_h^i}$ to~\eqref{eq:tps2_llg_variational}.
For all $\zz_h \in \NN_h$, it holds that $|\mm_h^i(\zz_h) + k \, \vv_h^i(\zz_h)|^2 = 1 + k^2 \, |\vv_h^i(\zz_h)|^2 \ge 1$ so that $\mm_h^{i+1} \in \Mh$ in~\eqref{alg:abtpsupdate}
is well-defined. Altogether, Algorithm~\ref{alg:abtps} is thus well-posed.
\end{proof}

\begin{lemma}\label{lemma:energy}
Under the assumptions of Theorem~\ref{thm:maintheorem}{\rm(ii)}, the following assertions {\rm (i)}--{\rm (ii)} are satisfied, if $k>0$ is sufficiently small.

{\rm (i)} For all $i = 0,\dots,N-1$, it holds that
\begin{align}\label{eq:energyestimate}
& \frac{\Cex}{2} \dtshort \norm{\nabla\mm_h^{i+1}}{\omega}^2 
+ \prod{\weight_{M(k)}(\lambda_h^i)\vv_h^i}{\vv_h^i}_{\omega} 
+ \frac{\Cex}{2} \, k \, \rho(k) \, \norm{\nabla\vv_h^i}{\omega}^2 
\notag 
\\& \quad 
\le \prod{\ff(t_{i+1/2})}{\vv_h^i}_{\omega}
+ \prod{\pih{i}{\vv_h^i}{\mm_h^i}{\mm_h^{i-1}}}{\vv_h^i}_{\omega} 
+ \prod{\Pih{i}{\vv_h^i}{\mm_h^i}{\mm_h^{i-1}}}{\vv_h^i}_{\omega} 
\end{align}

{\rm (ii)} For all $j = 0,\dots,N$, it holds that 
\begin{align}\label{eq:energyabtps}
\norm{\nabla \mm_h^j}{\omega}^2 + k\sum_{i=0}^{j-1}\norm{\vv_h^i}{\omega}^2 + k^2\rho(k)\sum_{i=0}^{j-1}\norm{\nabla\vv_h^i}{\omega}^2 \le C \,,
\end{align}
where $C>0$ only depends on $T$, $\omega$, $\alpha$, $\Cex$, $\ff$, $\ppi$, $\Ppi$, $\mm^0$, and $\Cmesh$.
\end{lemma}

\begin{proof}
Testing \eqref{eq:tps2_llg_variational} with $\vvarphi_h := \vv_h^i \in \Kh{\mm_h^i}$, we see that
\begin{align}\label{eq:abstpstimesv}
&\prod{\weight_{M(k)}(\lambda_h^i)\vv_h^i}{\vv_h^i}_{\omega} +
\frac{\Cex}{2} \, k \, \norm{\nabla\vv_h^i}{\omega}^2 + 
\frac{\Cex}{2} \,  k \rho(k) \, \norm{\nabla\vv_h^i}{\omega}^2 \notag \\
&= -\Cex\prod{\nabla\mm_h^i}{\nabla\vv_h^i}_{\omega} 
+ \prod{\ff(t_{i+1/2})}{\vv_h^i}_{\omega} 
\\& \qquad 
+ \prod{\pih{i}{\vv_h^i}{\mm_h^i}{\mm_h^{i-1}}}{\vv_h^i}_{\omega} \notag 
+ \prod{\Pih{i}{\vv_h^i}{\mm_h^i}{\mm_h^{i-1}}}{\vv_h^i}_{\omega}.
\end{align}
Since $\Trian$ is weakly acute, \cite[Lemma 3.2]{bartels2005} provides the estimate
\begin{align*}
\begin{split}
\norm{\nabla\mm_h^{i+1}}{\omega}^2 
\le \norm{\nabla ( \mm_h^i + k\vv_h^i )}{\omega}^2
&= \norm{\nabla\mm_h^i}{\omega}^2 + 2k \, \prod{\nabla\mm_h^i}{\nabla\vv_h^i}_{\omega} 
+ k^2\norm{\nabla\vv_h^i}{\omega}^2.
\end{split}
\end{align*}
Rearranging this estimate and multiplying it by $\Cex/(2k)$, we derive that
\begin{align}\label{eq:mn1mkv2}
\frac{\Cex}{2} \dtshort \norm{\nabla\mm_h^{i+1}}{\omega}^2 
- \frac{\Cex}{2} \, k \, \norm{\nabla\vv_h^i}{\omega}^2
\le
\Cex\prod{\nabla\mm_h^i}{\nabla\vv_h^i}_{\omega}.
\end{align}
Adding \eqref{eq:abstpstimesv} and \eqref{eq:mn1mkv2}, we prove {\rm (i)}. 

To prove {\rm (ii)}, we sum~\eqref{eq:energyestimate} over $i=0,\dots,j-1$ and multiply by $k$. For $k$ sufficiently small, we have $\weight_{M(k)}(\cdot) \geq \alpha/2 > 0$ (cf. Lemma~\ref{lemma:weight}{\rm (i)}) and altogether get that
\begin{align}\label{eq:discreteenergy_step1}
\begin{split}
{{\rm \chi}_j}  
&:=  \frac{\Cex}{2} \, \norm{ \nabla \mm_h^j }{\omega}^2 
+ \frac{\alpha}{2} \, k \sum_{i=0}^{j-1} \norm{ \vv_h^i }{\omega}^2 
+ \frac{\Cex}{2} \, k^2\rho(k)\sum_{i=0}^{j-1}\norm{\nabla\vv_h^i}{\omega}^2 
\\& 
\stackrel{\eqref{eq:energyestimate}}{\le} 
\frac{\Cex}{2} \, \norm{ \nabla \mm_h^0 }{\omega}^2 
+ k \sum_{i=0}^{j-1} \prod{\ff(t_{i+1/2})}{\vv_h^i}_{\omega} 
+ k \sum_{i=0}^{j-1} \prod{\pih{i}{\vv_h^i}{\mm_h^i}{\mm_h^{i-1}}}{\vv_h^i}_{\omega} 
\\& \quad 
+ k \sum_{i=0}^{j-1} \prod{\Pih{i}{\vv_h^i}{\mm_h^{i}}{\mm_h^{i-1}}}{\vv_h^i}_{\omega}.
\end{split}
\end{align}
Let $\delta > 0$. The Young inequality proves that
\begin{align}\label{dpr:llg1}
k^2 \sum_{i=0}^{j-1} \norm{ \nabla \vv_h^i }{\omega} \norm{ \vv_h^i }{\omega} \lesssim
\delta k^2 \rho(k) \sum_{i=0}^{j-1} \norm{ \nabla \vv_h^i }{\omega}^2 + 
\frac{k^2}{\delta \rho(k)} \,
\sum_{i=0}^{j-1} \norm{\vv_h^i }{\omega}^2.
\end{align}
Together with~\eqref{eq:pihstability} and~\eqref{eq:Pihstability}, further applications of the Young inequality prove that
\begin{align*}
\chi_j
&\stackrel{\eqref{eq:discreteenergy_step1}}\lesssim
\norm{ \nabla \mm_h^0 }{\omega}^2 
+ k \sum_{i=0}^{j-1} \norm{\ff(t_{i+1/2)}}{\omega}\norm{ \vv_h^i }{\omega}
+  k \sum_{i=0}^{j-1} \Big( k \, \norm{\vv_h^i}{\omega} 
\!+\! \norm{\mm_h^i}{\omega} 
\!+\! \norm{\mm_h^{i-1}}{\omega} \Big) \norm{\vv_h^i}{\omega} 
\\& \quad 
+ k \sum_{i=0}^{j-1} \Big( k \, \norm{\vv_h^i}{\HH{1}{\omega}}
+ \norm{\mm_h^i}{\HH{1}{\omega}} 
+ \norm{\mm_h^{i-1}}{\HH{1}{\omega}} \Big) \norm{\vv_h^i}{\omega}
\\&
\stackrel{\eqref{dpr:llg1}}{\lesssim}
\norm{ \nabla \mm_h^0 }{\omega}^2 
+ \frac{k}{\delta} \sum_{i=0}^{j-1} \norm{\ff(t_{i+1/2})}{\omega}^2 
+ \frac{k}{\delta} \sum_{i=0}^{j-1} \norm{ \mm_h^i }{\omega}^2
+ \frac{k}{\delta} \sum_{i=0}^{j-1} \norm{ \nabla \mm_h^i }{\omega}^2
\\& \quad 
+ \Big(\delta + k + \frac{k}{\delta \rho(k)}\Big) \, k \, \sum_{i=0}^{j-1} \norm{ \vv_h^i }{\omega}^2 
+ \delta \, k^2\rho(k) \, \sum_{i=0}^{j-1} \norm{ \nabla \vv_h^i }{\omega}^2.
\end{align*}
Since $k/\rho(k)\to0$ as $k\to0$, we choose $\delta>0$ sufficiently small and can absorb the sums $k \sum_{i=0}^{j-1} \norm{ \vv_h^i }{\omega}^2$ and $k^2 \rho(k) \sum_{i=0}^{j-1} \norm{ \nabla \vv_h^i }{\omega}^{2}$ in $\chi_j$.
Using that $\ff \in C^1([0,T],\LL{2}{\omega})$, we altogether
arrive at the estimate
\begin{align}
{{\rm \chi}_j}  
\lesssim  
\norm{ \nabla \mm_h^0 }{\omega}^2 
+ k \sum_{i=0}^{j-1} \norm{\ff(t_{i+1/2})}{\omega}^2 
+ k\sum_{i=0}^{j-1} \norm{ \mm_h^i }{\omega}^2
+ k \sum_{i=0}^{j-1} \norm{\nabla \mm_h^i}{\omega}^2 
\lesssim 
1 + k \sum_{i=0}^{j-1} {{\rm \chi_i}} . \notag
\end{align}
This fits in the setting of the discrete Gronwall lemma (cf.\ \cite[Lemma 1.4.2]{qv1994}), i.e.,
\begin{align}
{{\rm \chi}_j} \lesssim \alpha_0 + \beta \sum_{i=0}^{j-1} {{\rm \chi}_i} \notag \quad \textrm{with } \alpha_0 > 0 \textrm{ and } \beta \simeq k.
\end{align}
We obtain that
\begin{align}\label{eq:discrete6}
{\rm \chi}_j  \lesssim \alpha_0 \exp \left( \sum_{i=0}^{j-1} \beta \right) \lesssim \exp(T).
\end{align}
This concludes the proof.
\end{proof}

\begin{lemma}\label{lemma:extractsubsequences}
Under the assumptions of Theorem~\ref{thm:maintheorem}{\rm(ii)}, consider the postprocessed output~\eqref{eq:discretefunctions} of Algorithm~\ref{alg:abtps}.
Then, there exists $\mm\in\HH{1}{\omega_T}\cap\L{\infty}{0,T;\HH{1}{\omega}}$ as well as a subsequence of each $\mm_{hk}^{\star} \in \{\mm_{hk}, \mm_{hk}^{=} , \mm_{hk}^-, \mm_{hk}^+\}$ and of $\vv_{hk}^-$ such that the following convergence statements~\ref{item:weakly1}--\ref{item:weakly7} hold true simultaneously for the same subsequence as $h,k\to0$:
\begin{enumerate}[label=\rm{(\roman*)}]
\item \label{item:weakly1} $\mm_{hk} \rightharpoonup \mm$ in $\HH{1}{\omega_T}$,
\item \label{item:weakly2} $\mm_{hk}^{\star} \stackrel{*}\rightharpoonup \mm$ in $\L{\infty}{0,T;\HH{1}{\omega}}$,
\item \label{item:weakly2b} $\mm_{hk}^{\star} \rightharpoonup \mm$ in $\L{2}{0,T;\HH{1}{\omega}}$,
\item \label{item:weakly3} $\mm_{hk}^{\star} \to \mm$ in $\LL{2}{\omega_T}$,
\item \label{item:weakly4} $\mm_{hk}^{\star} (t) \rightarrow \mm(t)$ in $\LL{2}{\omega}$ for $t \in (0,T)$ a.e.,
\item \label{item:weakly5} $\mm_{hk}^{\star} \rightarrow \mm$ pointwise a.e.\ in $\omega_T$,
\item \label{item:weakly6} $\vv_{hk}^- \rightharpoonup \partial_t \mm$ in $\LL{2}{\omega_T}$,
\item \label{item:weakly7} $k \, \nabla \vv_{hk}^- \rightarrow 0$ in $\LL{2}{\omega_T}$.
\end{enumerate}
\end{lemma}

\begin{proof}
Lemma~\ref{lemma:energy} yields uniform boundedness of $\mm_{hk}$ in $\HH{1}{\omega_T}$ and of $\mm_{hk}^\star \in \L{\infty}{0,T;\linebreak\HH{1}{\omega}}$. Therefore, \ref{item:weakly1}--\ref{item:weakly6} follow as in~\cite{alouges2008a,bffgpprs2014}. Lemma~\ref{lemma:energy}(ii) yields that
\begin{align}\label{dpr:llg2}
k^2\norm{\nabla\vv_{hk}^-}{\LL{2}{\omega_T}}^2 = 
k \rho(k)^{-1} \rho(k) k^2 \sum_{i=0}^{N-1} \norm{\nabla\vv_{h}^i}{\omega}^2 \stackrel{\eqref{eq:energyabtps}}\lesssim k \rho(k)^{-1} 
\stackrel{\eqref{eq:assumption:rho}}{\longrightarrow} 0 \quad\text{ as } h,k\to 0.
\end{align}
This proves~(viii) and concludes the proof.
\end{proof}

\begin{remark}\label{remark:CFL}
Under the CFL-type condition $k = \oo(h)$, one may choose $\rho(k)\equiv0$ and hence violate~\eqref{eq:assumption:rho}. To see this, note that~\eqref{eq:assumption:rho} is only used for the proof of~\eqref{dpr:llg1} and~\eqref{dpr:llg2}. An inverse inequality yields that
\begin{align*}
k^2 \sum_{i=0}^{j-1} \norm{ \nabla \vv_h^i }{\omega} \norm{ \vv_h^i }{\omega} \lesssim 
(kh^{-1}) \, k \sum_{i=0}^{j-1} \norm{ \vv_h^i }{\omega}^2,
\end{align*}
where $k/h \to 0$ as $k\to0$. Similarly,
\begin{align*}
k^2 \, \norm{\nabla\vv_{hk}^-}{\LL{2}{\omega_T}}^2 = k^3 \sum_{i=0}^{N-1} \norm{\nabla\vv_{h}^i}{\omega}^2 \lesssim 
k^3 h^{-2} \sum_{i=0}^{N-1} \norm{\vv_{h}^i}{\omega}^2 \stackrel{\eqref{eq:energyabtps}}\lesssim k^2h^{-2} \longrightarrow 0 \quad\text{ as } h,k\to 0.
\end{align*}
Therefore, Lemma~\ref{lemma:energy} as well as Lemma~\ref{lemma:extractsubsequences} (and hence also Theorem~\ref{thm:maintheorem}) remain valid.\qed
\end{remark}

\begin{proof}[Proof of Theorem~\ref{thm:maintheorem}{\rm(ii)}]
We verify that $\mm\in\HH{1}{\omega_T}\cap\L{\infty}{0,T;\HH{1}{\omega}}$ from Lemma~\ref{lemma:extractsubsequences} satisfies Definition~\ref{def::weaksolution}\ref{item:weak_llg1}--\ref{item:weak_llg2}. The modulus constraint $|\mm|=1$ a.e.\ in $\omega_T$ as well as $\mm(0) = \mm^0$ in the sense of traces follow as in~\cite{alouges2008a,bffgpprs2014}. 
Hence, $\mm$ satisfies Definition~\ref{def::weaksolution}\ref{item:weak_llg1}--\ref{item:weak_llg3}.

It remains to prove that $\mm$ from Lemma~\ref{lemma:extractsubsequences} satisfies the variational formulation~\eqref{eq:variational_llg} from Definition \ref{def::weaksolution}\ref{item:weak_llg2}: To that end, let $\vvarphi\in\boldsymbol{C}^{\infty}(\overline{\omega_T})$.
Let $\II_{h}:\boldsymbol{C}(\overline{\omega}) \to \Vh$ be the (vector-valued) nodal interpolation operator.
For $t \in [t_i,t_{i+1})$ and $i=0,\dots,N-1$, we test~\eqref{eq:tps2_llg_variational} with $\II_{h}\big( \mm_h^i \times \vvarphi(t) \big) \in \Kh{\mm_h^i}$ and integrate over time. With the definition of the postprocessed output~\eqref{eq:discretefunctions} and $\overline{\ff}_k(t) := \ff(t_{i+1/2})$ for $t \in [t_i,t_{i+1})$, we obtain that
\begin{align}\label{eq:testedabtpsformulation}
& I_{hk}^1  + I_{hk}^2  + \frac{\Cex}{2} \, I_{hk}^3  
\notag 
\\&
:= \int_0^T\prod{\weight_{M(k)}(\lambda_{hk}^-)\vv_{hk}^-}{\II_{h}(\mm_{hk}^-\times\vvarphi)}_{\omega} \d{t} 
+ \int_0^T\prod{\mm_{hk}^-\times\vv_{hk}^-}{\II_{h}(\mm_{hk}^-\times\vvarphi)}_{\omega} \d{t} 
\notag 
\\& \quad 
+ \frac{\Cex}{2} \, k \, (1+\rho(k))\int_0^T\prod{\nabla\vv_{hk}^-}{\nabla\II_{h}(\mm_{hk}^-\times\vvarphi)}_{\omega} \d{t} 
\notag 
\\&  
\stackrel{\eqref{eq:tps2_llg_variational}}{=}
-\Cex\int_0^T\prod{\nabla\mm_{hk}^-}{\nabla\II_{h}(\mm_{hk}^-\times\vvarphi)}_{\omega} \d{t}
+ \int_0^T\prod{\ppi_{hk}^-(\vv_{hk}^-;\mm_{hk}^-,\mm_{hk}^=)}{\II_{h}(\mm_{hk}^-\times\vvarphi)}_{\omega} \d{t} 
\notag 
\\& \quad 
+ \int_0^T\prod{\overline{\ff}_{k}}{\II_{h}(\mm_{hk}^-\times\vvarphi)}_{\omega} \d{t} 
+ \int_0^T\prod{\Ppi_{hk}^-(\vv_{hk}^-;\mm_{hk}^-,\mm_{hk}^=)}{\II_{h}(\mm_{hk}^-\times\vvarphi)}_{\omega} \d{t} 
\notag 
\\& 
=: - \Cex I_{hk}^4  + I_{hk}^5  + I_{hk}^6  + I_{hk}^7 .
\end{align}
In the following, we prove convergence of the integrals from~\eqref{eq:testedabtpsformulation} towards their continuous counterparts in the variational formulation~\eqref{eq:variational_llg}: To this end,  recall the approximation properties of the nodal interpolation operator $\II_{h}$ and note that $\overline{\ff}_k\to\ff$ in $C([0,T];\LL{2}{\Omega})$.
With Lemma~\ref{lemma:extractsubsequences}, we get as in~\cite{alouges2008a,bffgpprs2014} that
\begin{align*}
I_{hk}^2  & \longrightarrow \int_0^T\prod{\mm\times\partial_t \mm}{\mm\times\vvarphi}_{\omega} \d{t} = 
\int_0^T\prod{\partial_t \mm}{\vvarphi}_{\omega} \d{t}  
&&\text{as $h,k\to0$},
\\
I_{hk}^4  & \longrightarrow
- \int_0^T\prod{\mm\times\nabla\mm}{\nabla\vvarphi}_{\omega} \d{t}
&&\text{as $h,k\to0$},
\\
I_{hk}^6  & \longrightarrow \int_0^T\prod{\ff}{\mm\times\vvarphi}_{\omega} \d{t}
&&\text{as $h,k\to0$}.
\end{align*}
Since $M(k)k \to 0$, Lemma~\ref{lemma:weight}(i) yields that $\weight_{M(k)}(\lambda_{hk}^-) \rightarrow \alpha$ in $\L{\infty}{\omega_T}$. Together with the assumption $\rho(k) \rightarrow 0$ as $k\to0$, we get as in~\cite{akst2014} that
\begin{align}
I_{hk}^1  & \longrightarrow \alpha \int_0^T\prod{\partial_t \mm}{\mm\times\vvarphi}_{\omega} \d{t} 
\quad \textrm{and} \quad I_{hk}^3  \rightarrow 0
\quad\text{as $h,k\to 0$}.
\end{align} 
With the assumptions~\eqref{eq:pihkconvergence} and~\eqref{eq:Pihkconvergence} on $\boldsymbol{\pi}_{hk}^-$ and $\boldsymbol{\Pi}_{hk}^-$, respectively, we conclude that
\begin{align}
I_{hk}^5  
&=
\int_0^T\prod{\ppi_{hk}^-(\vv_{hk}^-;\mm_{hk}^-,\mm_{hk}^=)}{\II_{h}(\mm_{hk}^-\times\vvarphi)}_{\omega} \d{t} \stackrel{\eqref{eq:pihkconvergence}}{\longrightarrow} \int_0^T\prod{\ppi(\mm)}{\mm\times\vvarphi}_{\omega} \d{t}, 
\notag \\
\label{eq:useconsistencyhere}
I_{hk}^7 
&= 
\int_0^T\prod{\Ppi_{hk}^-(\vv_{hk}^-;\mm_{hk}^-,\mm_{hk}^=)}{\II_{h}(\mm_{hk}^-\times\vvarphi)}_{\omega} \d{t} \stackrel{\eqref{eq:Pihkconvergence}}{\longrightarrow} \int_0^T\prod{\Ppi(\mm)}{\mm\times\vvarphi}_{\omega} \d{t}, 
\end{align}
as $h,k \rightarrow 0$. Altogether, $\mm$ from Lemma~\ref{lemma:extractsubsequences} satisfies the variational formulation~\eqref{eq:variational_llg}. 
\end{proof}

\begin{proof}[Proof of Theorem~\ref{thm:maintheorem}{\rm(iii)}]
It remains to verify that $\mm$ from Lemma~\ref{lemma:extractsubsequences} satisfies the energy estimate of Definition~\ref{def::weaksolution}\ref{item:weak_llg4}: To that end, let $\tau\in(0,T)$ be arbitrary and $j \in \{1,\dots,N\}$ such that $\tau \in [t_{j-1},t_{j})$. Besides the shorthand notation $\ff^j := \ff(t_j)$, define the time reconstructions $\ff_k$ and $\overline{\ff}_k$ according to~\eqref{eq:discretefunctions}.
For any $i = 0,\dots, j-1$, Lemma~\ref{lemma:energy}{\rm (i)} shows that
\begin{eqnarray}\label{eq:energy_strong_base}
&& \hspace*{-1.5cm}
\mathcal{E}_{\textrm{LLG}}(\mm_{hk}(t_{i+1}))
- \mathcal{E}_{\textrm{LLG}}(\mm_{hk}(t_{i})) \notag \\
&\stackrel{\eqref{eq:energy_LLG}}{=} &
 \frac{\Cex}{2} \, \norm{\nabla\mm_h^{i+1}}{\omega}^2 - \frac{\Cex}{2} \, \norm{\nabla\mm_h^i}{\omega}^2 \notag \\
&& \quad
 -\frac{1}{2} \, \prod{\ppi(\mm_h^{i+1})}{\mm_h^{i+1}}_{\omega} + \frac{1}{2} \, \prod{\ppi(\mm_h^i)}{\mm_h^i}_{\omega} - \prod{\ff^{i+1}}{\mm_h^{i+1}}_{\omega} + \prod{\ff^i}{\mm_h^i}_{\omega} \notag \\
&\stackrel{\eqref{eq:energyestimate}}{\le}&
- k \, \prod{\weight_{M(k)}(\lambda_h^i)\vv_h^i}{\vv_h^i}_{\omega} - \frac{\Cex}{2} \,  k^2 \rho(k) \, \norm{\nabla\vv_h^i}{\omega}^2 + k \, \prod{\Ppi_h^i(\vv_h^i;\mm_h^i,\mm_h^{i-1})}{\vv_h^i}_{\omega} \notag \\
&&
\quad +
\underbrace{ k \, \prod{\ppi_h^i(\vv_h^i;\mm_h^i,\mm_h^{i-1})}{\vv_h^i}_{\omega}
-\frac{1}{2} \, \prod{\ppi(\mm_h^{i+1})}{\mm_h^{i+1}}_{\omega} + 
\frac{1}{2} \, \prod{\ppi(\mm_h^i)}{\mm_h^i}_{\omega}}_{=:T_{\ppi}} \notag \\
&&
\quad + 
\underbrace{ k \, \prod{\ff(t_{i+1/2})}{\vv_h^i}_{\omega} - \prod{\ff^{i+1}}{\mm_h^{i+1}}_{\omega} + \prod{\ff^i}{\mm_h^i}_{\omega}}_{=:T_{\ff}} \,.
\end{eqnarray}
Since $\ppi$ is linear and self-adjoint, we obtain that
\begin{align}\label{eq:energystrong1}
T_{\ppi} &= -\frac{1}{2} \, \prod{\ppi(\mm_h^{i+1})}{\mm_h^{i+1}}_{\omega} + 
\frac{1}{2} \, \prod{\ppi(\mm_h^i)}{\mm_h^i}_{\omega} + 
k \, \prod{\ppi_h^i(\vv_h^i;\mm_h^i,\mm_h^{i-1})}{\vv_h^i}_{\omega} \notag \\
&= k \, \prod{\ppi_h^i(\vv_h^i;\mm_h^i,\mm_h^{i-1}) - \ppi(\mm_h^i)}{\vv_h^i}_{\omega} 
-\prod{\ppi(\mm_h^i)}{\mm_h^{i+1}-\mm_h^i-k\vv_h^i} \notag \\
&\quad -\frac{1}{2} \, \prod{\ppi(\mm_h^{i+1}-\mm_h^i)}{\mm_h^{i+1}-\mm_h^i}_{\omega} .
\end{align}
Similarly, it holds that 
\begin{align*}
T_{\ff} 
&= k \, \prod{\ff(t_{i+1/2})}{\vv_h^i}_{\omega} 
- k \, \prod{\dtshort \ff^{i+1}}{\mm_h^i}_{\omega}  
- \prod{\ff^{i+1}}{\mm_h^{i+1}-\mm_h^i}_{\omega} 
\\& 
= k \, \prod{\ff(t_{i+1/2})-\ff^{i+1}}{\vv_h^i}_{\omega} 
- k \, \prod{\dtshort \ff^{i+1}}{\mm_h^i}_{\omega}  
- \prod{\ff^{i+1}}{\mm_h^{i+1}-\mm_h^i-k\vv_h^i}_{\omega} 
\\&  
=:  T_{\ff,1} - k \, \prod{\dtshort \ff^{i+1}}{\mm_h^i}_{\omega} + T_{\ff,2}.
\end{align*}
Elementary calculations (see, e.g.,~\cite[Lemma~3.3.2]{goldenits2012}) show that
\begin{align*}
 |\mm_h^{i+1}(\zz_h) - \mm_h^i(\zz_h)| 
 \le k \, |\vv_h^i(\zz_h)|
 \quad \text{for all } \zz_h \in \NN_h.
\end{align*}
A scaling argument thus proves that
\begin{align}\label{eq:auxiliary_estimate2}
 k \, \norm{\dtshort \mm_h^{i+1}}{\LL{2}{\omega}} 
 = \norm{\mm_h^{i+1} - \mm_h^i}{\LL{2}{\omega}}
 \lesssim k \, \norm{\vv_h^i}{\LL{2}{\omega}}.
\end{align}
Elementary calculations (see, e.g.,~\cite[eq.~(22)]{aj2006} or~\cite[Lemma~3.3.3]{goldenits2012}) show that
\begin{align*}
 |\mm_h^{i+1}(\zz_h) - \mm_h^i(\zz_h) - k \, \vv_h^i(\zz_h)| 
 \le \frac{k^2}{2} \, |\vv_h^i(\zz_h)|^2
 \quad \text{for all } \zz_h \in \NN_h.
\end{align*}
A scaling argument thus proves that
\begin{align}\label{eq:auxiliary_estimate3}
 k \, \norm{\dtshort \mm_h^{i+1} - \vv_h^i}{\LL{p}{\omega}} 
 = \norm{\mm_h^{i+1} - \mm_h^i - k \, \vv_h^i}{\LL{p}{\omega}}
 \lesssim k^2 \, \norm{\vv_h^i}{\LL{{2p}}{\omega}}^2
 \text{ for } 1 \le p < \infty.
\end{align}%
The Sobolev embedding $\HH{1}{\omega} \subset \LL{6}{\omega}$ yields that
\begin{align*}
 \norm{\vv_h^i}{\LL{3}{\omega}}^2
 \le \norm{\vv_h^i}{\LL{2}{\omega}} \, \norm{\vv_h^i}{\LL{6}{\omega}}
 \lesssim \norm{\vv_h^i}{\LL{2}{\omega}} \, \norm{\vv_h^i}{\HH{1}{\omega}}
 \le \norm{\vv_h^i}{\omega}^2 + \norm{\vv_h^i}{\omega} 
 \, \norm{\nabla\vv_h^i}{\omega}.
\end{align*}
Therefore, we obtain that
\begin{align*}
\norm{\mm_h^{i+1} - \mm_h^i - k \vv_h^i}{\LL{{3/2}}{\omega}} 
\stackrel{\eqref{eq:auxiliary_estimate3}}\lesssim k^2 \, \norm{\vv_h^i}{\LL{3}{\omega}}^2
\lesssim
k^2 \, \norm{\vv_h^i}{\omega} \norm{\nabla \vv_h^i}{\omega} +  k^2 \, \norm{\vv_h^i}{\omega}^2.
\end{align*}%
With the stronger boundedness~\eqref{eq:strongerboundedness} of $\ppi$ and the Hölder inequality, we derive that
\begin{align}\label{eq:energystrong6}
\begin{split}
 |T_{\ppi}| 
 & \stackrel{\eqref{eq:energystrong1}}{\lesssim} 
 k \, | \prod{\ppi_h^i(\vv_h^i;\mm_h^i,\mm_h^{i-1}) - \ppi(\mm_h^i)}{\vv_h^i}_{\omega} | 
\\& \quad 
 + \norm{\ppi(\mm_h^i)}{\LL{3}{\omega}} \, \norm{\mm_h^{i+1}-\mm_h^i-k\vv_h^i}{\LL{{3/2}}{\omega}} 
\\& \quad 
+ \norm{\ppi(\mm_h^{i+1} - \mm_h^i)}{\omega} \, \norm{\mm_h^{i+1} - \mm_h^i}{\omega} 
\\& 
\,\lesssim\, 
k \, | \prod{\ppi_h^i(\vv_h^i;\mm_h^i,\mm_h^{i-1}) - \ppi(\mm_h^i)}{\vv_h^i}_{\omega} | 
+ k^2 \, \norm{\vv_h^i}{\omega} \norm{\nabla \vv_h^i}{\omega} 
+ k^2 \, \norm{\vv_h^i}{\omega}^2.
\end{split}
\end{align}
Similarly, the additional assumption $\ff \in C^1([0,T];\LL{2}{\omega}) \cap C([0,T];\LL{3}{\omega})$ yields that
\begin{eqnarray}\label{eq:energystrong7}
|T_{\ff,1} + T_{\ff,2}| 
& \lesssim &
k |\prod{\ff(t_{i+1/2})-\ff^{i+1}}{\vv_h^i}_{\omega}| 
+ k^2 \, \norm{\vv_h^i}{\omega} \norm{\nabla \vv_h^i}{\omega} 
+ k^2 \, \norm{\vv_h^i}{\omega}^2.
\end{eqnarray}
The combination of~\eqref{eq:energy_strong_base} with~\eqref{eq:energystrong6}--\eqref{eq:energystrong7} and summation over $i = 0,\dots, j-1$ yields that
\begin{align}\label{eq::proofthmenergyeq1}
\begin{split}
&\mathcal{E}_{\textrm{LLG}}(\mm_{hk}^+(\tau)) - \mathcal{E}_{\textrm{LLG}}(\mm_h^0) + \int_0^{t_j}\prod{\weight_{M(k)}(\lambda_{hk}^-)\vv_{hk}^-}{\vv_{hk}^-}_{\omega} \,\d{t} 
\\
& \,\,\,  +  
\int_0^{t_j} \prod{\partial_t \ff_k}{\mm_{hk}^-}_{\omega}\,\d{t} 
- \int_0^{t_j} \prod{\Ppi_{hk}(\vv_{hk}^-;\mm_{hk}^-,\mm_{hk}^=)}{\vv_{hk}^-}_{\omega}\,\d{t} 
\\
&\lesssim k\int_0^{t_j}\norm{\vv_{hk}^-}{\omega}^2\,\d{t} 
+ k\int_0^{t_j}\norm{\vv_{hk}^-}{\omega} \norm{\nabla\vv_{hk}^-}{\omega}\,\d{t} 
\\
&\,\,\, + \int_0^{t_j}| \prod{\ppi_{hk}(\vv_{hk}^-;\mm_{hk}^-,\mm_{hk}^=) - \ppi(\mm_{hk}^-)}{\vv_h^j}_{\omega} |\,\d{t} + \int_0^{t_j}|\prod{\overline{\ff}_{k}-\ff_k^+}{\vv_{hk}^-}_{\omega}|\,\d{t}.
\end{split}
\end{align}
From strong convergence in~\eqref{eq:initconvergence_m0}, it follows that $\mathcal{E}_{\textrm{LLG}}(\mm_{h}^0) \rightarrow \mathcal{E}_{\textrm{LLG}}(\mm^0)$ as $h,k \rightarrow 0$. 
The first and second term on the right-hand side vanish as $h,k \rightarrow 0$ due to Lemma~\ref{lemma:extractsubsequences}\ref{item:weakly6}--\ref{item:weakly7}. Thanks to~\eqref{eq:pihkconvergence} and~\eqref{eq:Pihkconvergence} with strong convergence,
 the last two terms on the right-hand side of~\eqref{eq::proofthmenergyeq1} vanish as $h,k \rightarrow 0$. Standard lower semicontinuity arguments for the remaining terms in~\eqref{eq::proofthmenergyeq1} conclude the proof.
\end{proof}

\section{Proof of Theorem~\ref{thm:maintheorem_ellg} (Numerical integration of ELLG)} 
\label{section:proof:maintheorem_ellg}

\begin{proof}[Proof of Theorem~\ref{thm:maintheorem_ellg}{\rm(i)}]
Note that the right-hand side of~\eqref{eq:tps2_ellg_variational1} can depend non-linearly on $\vv_h^i$. As in the proof of Theorem~\ref{thm:maintheorem}(i), we employ a fixpoint iteration, where 
$\big(\eeta_h^\ell,\nnu_h^\ell\big) \approx \big(\vv_h^i,\midh{\hh}{i}\big) \in \Kh{\mm_h^i} \times \cchi_h$. To this end, let $\eeta_h^0:=0$ and $\nnu_h^0:=\hh_h^i$. For $\ell\in\N_0$, let
$\eeta_h^{\ell+1}\in \Kh{\mm_h^i}$ satisfy, for all $\vvarphi_h \in \Kh{\mm_h^i}$,
\begin{subequations}\label{eq:fixedpoint_ellg_variational}
\begin{align}\label{eq:fixedpoint_ellg_variational1}
\begin{split}
&\!\! \prod{\weight_{M(k)}(\lambda_h^i)\eeta_h^{\ell+1}}{\vvarphi_h}_{\omega} 
+ \prod{\mm_h^i\times\eeta_h^{\ell+1}}{\vvarphi_h}_{\omega} + 
\frac{\Cex}{2} \, k \, (1+\rho(k))\prod{\nabla\eeta_h^{\ell+1}}{\nabla\vvarphi_h}_{\omega} \\
&\!\! = -\Cex\prod{\nabla\mm_h^i}{\nabla\vvarphi_h}_{\omega} + 
\prod{\pih{i}{\eeta_h^{\ell}}{\mm_h^i}{\mm_h^{i-1}}}{\vvarphi_h}_{\omega}
+ \prod{\midh{\ff}{i}}{\vvarphi_h}_{\omega} \\
&\!\! \quad +  \prod{\Pih{i}{\eeta_h^{\ell}}{\mm_h^i}{\mm_h^{i-1}}}{\vvarphi_h}_{\omega}  + 
2 \Ttheta^k_{i3} \prod{\nnu_h^\ell}{\vvarphi_h}_{\omega} + 
 \prod{( \Ttheta^k_{i2}\!-\!\Ttheta^k_{i3}) \hh_h^i  \!+\! \Ttheta^k_{i1} \hh_h^{i-1} }{\vvarphi_h}_{\omega}.
\end{split}
\end{align}
Moreover, let $\nnu_h^{\ell+1} \in \cchi_h$ satisfy, for all $\zzeta_h \in \cchi_h$,
\begin{align}\label{eq:fixedpoint_ellg_variational2}
\begin{split}
& \frac{2\mu_0}{k} \prod{\nnu_h^{\ell+1}}{\zzeta_h}_{\Omega} + 
\prod{\sigma^{-1} \nabla \times \nnu_h^{\ell+1}}{\nabla \times \zzeta_h}_{\Omega} \\
& \quad = 
- \frac{\mu_0}{k} \langle \, \II_h \bigg( \, \frac{\mm_h^i + k \eeta_h^{\ell+1}}{|\mm_h^i + k \eeta_h^{\ell+1}|} \, \bigg),\zzeta_h \rangle_{\omega} + \frac{\mu_0}{k} 
\prod{\mm_h^i}{\zzeta_h}_{\omega} +
\frac{2\mu_0}{k} \prod{\hh_h^i}{\zzeta_h}_{\Omega},
\end{split}
\end{align}
where $\II_h:\boldsymbol{C}(\overline\omega)\to\Vh$ denotes the nodal interpolation operator.
\end{subequations}
Since $\rho(k) \geq 0$ and $\weight_{M(k)}(\cdot) > 0$, the bilinear forms on the left-hand sides of~\eqref{eq:fixedpoint_ellg_variational} are elliptic on $(\Kh{\mm_h^i},\norm{\cdot}{\omega})$ and $(\cchi_h,\norm{\cdot}{\Omega})$, respectively. Since $\eeta_h^{\ell+1}$ is known for the computation of $\nnu_h^{\ell+1}$, the fixpoint iteration is thus well-defined. 
We subtract~\eqref{eq:fixedpoint_ellg_variational} for $(\eeta_h^\ell,\nnu_h^\ell)$
from~\eqref{eq:fixedpoint_ellg_variational} for $(\eeta_h^{\ell+1},\nnu_h^{\ell+1})$ and test with $\vvarphi_h := \eeta_h^{\ell+1} - \eeta_h^{\ell} \in \Kh{\mm_h^i}$ and $\zzeta_h := \nnu_h^{\ell+1} - \nnu_h^{\ell} \in \cchi_h$.
For sufficiently small $k$, Lemma~\ref{lemma:weight}(i) and~\eqref{eq:assumption:M} prove that $\weight_{M(k)}(\cdot) \ge \alpha/2$.
With $\sup_{k} \max_{ij} \abs{\Ttheta^k_{ij}} \leq C_{\Ttheta} < \infty$ as well as~\eqref{eq:pihlipschitz} and~\eqref{eq:Pihlipschitz}, we get as in the proof of Theorem~\ref{thm:maintheorem}(i) that
\begin{align}\label{eq:welldefinid_ell_step1}
&\frac{\alpha}{2} \, \norm{\eeta_h^{\ell+1} - \eeta_h^{\ell}}{\omega}^2 
+ \frac{\Cex}{2} \, k \, (1+\rho(k)) \, \norm{\nabla(\eeta_h^{\ell+1} - \eeta_h^{\ell})}{\omega}^2
\notag
\\&\quad
\stackrel{\eqref{eq:fixedpoint_ellg_variational1}}{\leq}
 \big( \Cpi + \CPi  \big) \, k \, \norm{\eeta_h^{\ell} - \eeta_h^{\ell-1}}{\omega}
 \norm{\eeta_h^{\ell+1} - \eeta_h^{\ell}}{\omega}
 + \CPi \, k \, \norm{\nabla(\eeta_h^{\ell} - \eeta_h^{\ell-1})}{\omega}
 \norm{\eeta_h^{\ell+1} - \eeta_h^{\ell}}{\omega}
 \notag 
 \\&\qquad
 + 2 C_{\Ttheta} \norm{\nnu_h^{\ell} - \nnu_h^{\ell-1}}{\Omega} \norm{\eeta_h^{\ell+1} - \eeta_h^{\ell}}{\omega} 
\quad \textrm{for all } \ell \in \N.
\end{align}
For all $\vvarphi_h \in \Vh$, it holds that $\norm{\II_h\vvarphi_h}{\omega} \leq \sqrt{5}\,\norm{\vvarphi_h}{\omega}$; see, e.g.,~\cite[Lemma 2.2.3]{goldenits2012}. Moreover, for all $x,y \in \R$ with $|x|, |y| \geq 1$, it holds that $| x / |x| - y / |y| | \leq |x - y|$. Since $|\mm_h^i(\zz_h) + k \eeta_h^{\ell}(\zz_h)| \geq 1$ for all $\ell \in \N_0$ and for all nodes $\zz_h \in \NN_h|_{\omega}$, we get that
\begin{eqnarray*}
& & \hspace*{-1.5cm} \norm{\nnu_h^{\ell+1} - \nnu_h^{\ell}}{\Omega}^2 + 
\frac{k}{2\mu_0} 
\norm{\sigma^{-1/2}\nabla \times \big( \nnu_h^{\ell+1} - \nnu_h^{\ell} )}{\Omega}^2 \\
&\stackrel{\eqref{eq:fixedpoint_ellg_variational2}}{=} &
- \frac12 \, 
\bigg\langle \II_h \bigg(\, \frac{\mm_h^i + k \eeta_h^{\ell+1}}{|\mm_h^i + k \eeta_h^{\ell+1}|} \,\bigg)
- \II_h \bigg(\, \frac{\mm_h^i + k \eeta_h^{\ell}}{|\mm_h^i + k \eeta_h^{\ell}|} \,\bigg)
,
\nnu_h^{\ell+1} - \nnu_h^{\ell} 
\bigg\rangle_{\omega} \notag \\
& \leq & \frac{\sqrt{5}k}{2} \, k \,  \norm{\eeta_h^{\ell+1} - \eeta_h^{\ell} }{\omega} \norm{\nnu_h^{\ell+1} - \nnu_h^{\ell}}{\Omega}
\quad \textrm{for all } \ell \in \N_0.
\end{eqnarray*}
The latter equation yields that
\begin{align}\label{eq:contraction_ellg1}
\norm{\nnu_h^{\ell+1} - \nnu_h^{\ell}}{\Omega} \leq \frac{\sqrt{5}}{2} \, k \,  \norm{\eeta_h^{\ell+1} - \eeta_h^{\ell} }{\omega}
\quad \textrm{for all } \ell \in \N_0.
\end{align}
We add~\eqref{eq:welldefinid_ell_step1}--\eqref{eq:contraction_ellg1} and obtain that
\begin{align}\label{eq:contraction_ellg2}
\begin{split}
&\frac{\alpha}{2} \, \norm{\eeta_h^{\ell+1} - \eeta_h^{\ell}}{\omega}^2 
+ \frac{\Cex}{2} \, k \, (1+\rho(k)) \, \norm{\nabla(\eeta_h^{\ell+1} - \eeta_h^{\ell})}{\omega}^2
+ \norm{\nnu_h^{\ell+1} - \nnu_h^{\ell}}{\Omega}^2
\\&\quad
\leq 
 \big( \Cpi + \CPi + \sqrt{5} \, C_{\Ttheta}\big) \, k \, \norm{\eeta_h^{\ell} - \eeta_h^{\ell-1}}{\omega}
 \norm{\eeta_h^{\ell+1} - \eeta_h^{\ell}}{\omega}
 \\&\qquad
 + \CPi \, k \, \norm{\nabla(\eeta_h^{\ell} - \eeta_h^{\ell-1})}{\omega}
 \norm{\eeta_h^{\ell+1} - \eeta_h^{\ell}}{\omega}
 + \frac{5}{4} \, k^2 \, \norm{\eeta_h^{\ell+1} - \eeta_h^{\ell} }{\omega}^2
\quad \textrm{for all } \ell \in \N.
\end{split}
\end{align}
We equip the product space $\Kh{\mm_h^i} \times \cchi_h$ with the norm
$|\!|\!|(\eeta_h,\nnu_h)|\!|\!|^2 := (\alpha/4) \, \norm{\eeta_h}{\omega}^2 + (\Cex/2) \, k \, (1+\rho(k)) \, \norm{\nabla\eeta_h}{\omega}^2 + \norm{\nnu_h}{\Omega}^2$. For $\delta>0$, the Young inequality then proves that
\begin{align*}
& \min\Big\{1 \, , \, 2 - \frac{1}{\alpha\delta} \, (\Cpi + 2\,\CPi + \sqrt{5} \, C_{\Ttheta}) \, k - \frac{5}{\alpha}\, k^2 \Big\}
\, |\!|\!| (\eeta_h^{\ell+1},\nnu_h^{\ell+1}) - (\eeta_h^{\ell},\nnu_h^\ell) |\!|\!|^2
\notag \\
&\quad\le
\Big[\frac{\alpha}{2} - \frac{1}{4\delta} \, (\Cpi + 2\,\CPi + \sqrt{5} \, C_{\Ttheta}) \, k - \frac{5}{4}\, k^2 \Big] \, \norm{\eeta_h^{\ell+1} - \eeta_h^{\ell}}{\omega}^2
\notag
\\&\quad\qquad
+ \frac{\Cex}{2} \, k \, (1+\rho(k)) \, \norm{\nabla \eeta_h^{\ell+1} - \nabla \eeta_h^{\ell}}{\omega}^2 
 + \norm{\nnu_h^{\ell+1} - \nnu_h^{\ell}}{\Omega}^2
\notag \\
&\quad\le
\delta \, (\Cpi + \CPi + \sqrt{5} \, C_{\Ttheta}) \,k \,\norm{\eeta_h^{\ell} - \eeta_h^{\ell-1}}{\omega}^2 
+ \delta \, \CPi \, k \, \norm{\nabla(\eeta_h^{\ell} - \eeta_h^{\ell-1})}{\omega}^2
\notag \\
&\quad\le \delta \, \max\Big\{\frac{4}{\alpha} \, (\Cpi + \CPi+ \sqrt{5} \, C_{\Ttheta}) \,k \, , \, 2 \, \frac{\CPi}{\Cex} \Big\} 
\, |\!|\!| (\eeta_h^{\ell},\nnu_h^{\ell}) - (\eeta_h^{\ell-1},\nnu_h^{\ell-1}) |\!|\!|^2
 \quad \textrm{for all } \ell \in \N.
\end{align*}
For sufficiently small $\delta$ and $k$, the iteration is thus a contraction with respect to $|\!|\!|\cdot|\!|\!|$.
The Banach fixpoint theorem yields existence and uniqueness of 
$(\vv_h^i,\midh{\hh}{i}) \in \Kh{\mm_h^i} \times \cchi_h$ of~\eqref{eq:fixedpoint_ellg_variational}. With $\hh_h^{i+1} := 2\midh{\hh}{i} - \hh_h^i \in \cchi_h$, \ $(\vv_h^i,\hh_h^{i+1}) \in \Kh{\mm_h^i} \times \cchi_h$ is the unique solution of~\eqref{eq:abtps_ellg}; cf.\ Remark~\ref{remark:adellg}{\rm (v)}.
The remainder of the proof follows as for Theorem~\ref{thm:maintheorem}(i).
\end{proof}

\begin{lemma}\label{lemma:energy_ellg}
Under the assumptions of Theorem~\ref{thm:maintheorem_ellg}{\rm(ii)}, the following assertions {\rm (i)}--{\rm (iii)} are satisfied, if $k>0$ is sufficiently small.

{\rm (i)} For all $i = 0,\dots,N-1$ it holds that
\begin{align}\label{eq:energyestimate_ellg1}
\begin{split}
& \frac{\Cex}{2} \dtshort \norm{\nabla\mm_h^{i+1}}{\omega}^2 + \prod{\weight_{M(k)}(\lambda_h^i)\vv_h^i}{\vv_h^i}_{\omega} + \frac{\Cex }{2} \, k\rho(k) \,\norm{\nabla\vv_h^i}{\omega}^2 
\\
& \quad + \frac12 \dtshort \norm{\hh_h^{i+1}}{\Omega}^2 + 
\frac{1}{\mu_0} \norm{\sigma^{-1/2} \nabla \times \midh{\hh}{i}}{\Omega}^2
\\
& \le \prod{\ff(t_{i+1/2})}{\vv_h^i}_{\omega} + \prod{\pih{i}{\vv_h^i}{\mm_h^i}{\mm_h^{i-1}}}{\vv_h^i}_{\omega}
+ \prod{\Pih{i}{\vv_h^i}{\mm_h^i}{\mm_h^{i-1}}}{\vv_h^i}_{\omega}
\\
& \quad + \prod{\vv_h^i - \dth{\mm}{i}}{\midh{\hh}{i}}_{\omega} + \prod{\vv_h^i}{\gentimeh{\hh}{i} - \midh{\hh}{i}}_{\omega}.
\end{split}
\end{align}

{\rm (ii)} For all $i = 0,\dots,N-1$, it holds that 
\begin{align}\label{eq:discreteenergy_ellg2}
\mu_0 \norm{\dth{\hh}{i}}{\Omega}^2 +  \dtshort \norm{\sigma^{-1/2} \nabla \times \hh_h^{i+1}}{\Omega}^2
&\leq \mu_0 \norm{\dth{\mm}{i}}{\omega}^2.
\end{align}

{\rm (iii)} For all $j = 0,\dots,N-1$, it holds that
\begin{align}\label{eq:energyabtps_ellg}
\begin{split}
& \norm{\nabla \mm_h^j}{\omega}^2 + k\sum_{i=0}^{j-1}\norm{\vv_h^i}{\omega}^2 
+ k^2\rho(k)\sum_{i=0}^{j-1}\norm{\nabla\vv_h^i}{\omega}^2 \\
& \quad + \norm{\hh_h^{j}}{\Omega}^2 + \norm{\sigma^{-1/2}\nabla \times \hh_h^{j}}{\Omega}^2 
+ k \sum_{i=0}^{j-1} \norm{\dth{\hh}{i}}{\Omega}^2  \le C,
\end{split}
\end{align}
where $C>0$ depends only on $T$, $\omega$, $\Omega$, $\alpha$, $\mu_0$, $\sigma_0$, $\Cex$, $\ff$, $\ppi$, $\Ppi$, $\mm^0$, $\hh^0$, $C_{\Ttheta}$, and $\Cmesh$.
\end{lemma}

\begin{proof}
For the LLG-part~\eqref{eq:tps2_ellg_variational1}, we argue as in the proof of Lemma~\ref{lemma:energy}(i) to see that
\begin{align}\label{eq:discrete1}
& \frac{\Cex}{2} \dtshort \norm{\nabla\mm_h^{i+1}}{\omega}^2 + \prod{\weight_{M(k)}(\lambda_h^i)\vv_h^i}{\vv_h^i}_{\omega} + \frac{\Cex}{2} \, k\rho(k) \, \norm{\nabla\vv_h^i}{\omega}^2 
\\& \notag
\,\,\, \le \prod{\ff(t_{i+1/2})}{\vv_h^i}_{\omega} \!+\! \prod{\pih{i}{\vv_h^i}{\mm_h^i}{\mm_h^{i-1}}}{\vv_h^i}_{\omega}
\!+\! \prod{\Pih{i}{\vv_h^i}{\mm_h^i}{\mm_h^{i-1}}}{\vv_h^i}_{\omega} \!+\! \prod{\gentimeh{\hh}{i}}{\vv_h^i}_{\omega}. 
\end{align}
Testing~\eqref{eq:tps2_ellg_variational2} with $\zzeta_h := - (1 / \mu_0) \, \midh{\hh}{i}$, we obtain that
\begin{eqnarray}\label{eq:discrete2}
 \prod{\dth{\mm}{i}}{\midh{\hh}{i}}_{\omega} & \stackrel{\eqref{eq:tps2_ellg_variational2}}{=} & 
- \prod{\dth{\hh}{i}}{\midh{\hh}{i}}_{\Omega} - 
\frac{1}{\mu_0} \norm{\sigma^{-1/2} \nabla \times \midh{\hh}{i}}{\Omega}^2 \notag \\
& = & - \frac12 \dtshort \norm{\hh_h^{i+1}}{\Omega}^2 - 
\frac{1}{\mu_0} \norm{\sigma^{-1/2} \nabla \times \midh{\hh}{i}}{\Omega}^2.
\end{eqnarray}
Inserting $\gentimeh{\hh}{i}$ and $\vv_h^i$ in~\eqref{eq:discrete2}, we are led to
\begin{align}\label{eq:discrete3}
\begin{split}
\prod{\vv_h^i}{\gentimeh{\hh}{i}}_{\omega} &= 
\prod{\vv_h^i}{\gentimeh{\hh}{i} - \midh{\hh}{i}}_{\omega}
+ \prod{\vv_h^i - \dth{\mm}{i}}{\midh{\hh}{i}}_{\omega} \\
& \quad -\frac12 \dtshort \norm{\hh_h^{i+1}}{\Omega}^2 - \frac{1}{\mu_0} \norm{\sigma^{-1/2} \nabla \times \midh{\hh}{i}}{\Omega}^2.
\end{split}
\end{align}
Adding~\eqref{eq:discrete1} and~\eqref{eq:discrete3}, we prove {\rm (i)}. To prove {\rm (ii)}, we test~\eqref{eq:tps2_ellg_variational2} with $\zzeta_h := \dth{\hh}{i}$. With the Young inequality, we obtain that
\begin{align}\label{eq:discrete8}
\mu_0 \norm{\dth{\hh}{i}}{\Omega}^2 + 
\frac{1}{2} \dtshort \norm{\sigma^{-1/2} \nabla \times \hh_h^{i+1}}{\Omega}^{2} 
&\stackrel{\eqref{eq:tps2_ellg_variational2}}{=} - \mu_0 \prod{\dth{\mm}{i}}{\dth{\hh}{i}}_{\omega} \notag \\
& \stackrel{\phantom{\eqref{eq:tps2_ellg_variational2}}}{\leq} 
\frac{\mu_0}{2} \norm{\dth{\mm}{i}}{\omega}^2 + \frac{\mu_0}{2} \norm{\dth{\hh}{i}}{\Omega}^2. 
\end{align}
This proves {\rm (ii)}. To prove~{\rm (iii)}, we sum~\eqref{eq:energyestimate_ellg1} over $i = 0,\dots,j-1$ and multiply with $k$.
With $\weight_{M(k)}(\cdot) \geq \alpha/2 > 0$ for $k$ being sufficiently small, we obtain that
\begin{align*}
{\rm \chi}_j
& := 
\frac{\Cex}{2} \, \norm{ \nabla \mm_h^{j} }{\omega}^2 + 
\frac{\alpha}{2} \, k \sum_{i=0}^{j-1} \norm{ \vv_h^i }{\omega}^2 
+ \frac{\Cex}{2} \, k^2\rho(k) \, \sum_{i=0}^{j-1}\norm{\nabla\vv_h^i}{\omega}^2 \notag \\
& \quad +\frac12 \, \norm{\hh_h^{j}}{\Omega}^2 + 
\frac{k}{\mu_0} \sum_{i=0}^{j-1} \norm{\sigma^{-1/2} \nabla \times \midh{\hh}{i}}{\Omega}^2 \notag \\
& \stackrel{\eqref{eq:energyestimate_ellg1}}{\le} 
\frac{\Cex}{2} \, \norm{ \nabla \mm_h^0 }{\omega}^2 
+ k \sum_{i=0}^{j-1} \prod{\ff(t_{i+1/2})}{\vv_h^i}_{\omega} 
+ k \sum_{i=0}^{j-1} \prod{\pih{i}{\vv_h^i}{\mm_h^i}{\mm_h^{i-1}}}{\vv_h^i}_{\omega}
\notag \\
& \quad + k \sum_{i=0}^{j-1} \prod{\Pih{i}{\vv_h^i}{\mm_h^i}{\mm_h^{i-1}}}{\vv_h^i}_{\omega} + \frac12 \, \norm{\hh_h^{0}}{\Omega}^2
 \notag \\
& \quad  
 + k \sum_{i=0}^{j-1} \prod{\vv_h^i - \dth{\mm}{i}}{\midh{\hh}{i}}_{\omega}
+ k \sum_{i=0}^{j-1} \prod{\vv_h^i}{\gentimeh{\hh}{i} - \midh{\hh}{i}}_{\omega}. 
\end{align*}
Recalling from~\eqref{eq:auxiliary_estimate2} that $\norm{\dth{\mm}{i}}{\omega} \lesssim \norm{\vv_h^i}{\omega}$, we proceed as in the proof of Lemma~\ref{lemma:energy}{\rm (ii)}.
Together with~\eqref{eq:pihstability} and~\eqref{eq:Pihstability}, the Young inequality proves that
\begin{align}\label{eq:discrete5}
\begin{split}
{\rm \chi}_j   
&\lesssim \norm{\nabla \mm_h^0}{\omega}^2 
+ \norm{\hh_h^0}{\Omega}^2
+ \frac{k}{\delta} \sum_{i=0}^{j-1} \norm{\ff(t_{i+1/2})}{\omega}^2 
+ \frac{k}{\delta} \sum_{i=0}^{j-1} \norm{\mm_h^i}{\omega}^2
+ \frac{k}{\delta} \sum_{i=0}^{j-1} \norm{ \nabla \mm_h^{i}}{\omega}^2
\\&\quad
+ \frac{k}{\delta} \sum_{i=0}^{j} \norm{\hh_h^{i}}{\omega}^2
+ \Big(\delta + k + \frac{k}{\delta\rho(k)}\Big) \, k \, \sum_{i=0}^{j-1} \norm{ \vv_h^i }{\omega}^2 
+ \delta k^2 \rho(k) \, \sum_{i=0}^{j-1} \norm{ \nabla \vv_h^i }{\omega}^2.
\end{split}
\end{align}
Since $k/\rho(k) \to 0$ as $k\to0$, we choose $\delta$ sufficiently small and can absorb
$k\delta^{-1} \, \norm{\hh_h^{j}}{\omega}^2 + k \, \sum_{i=0}^{j-1} \norm{\vv_h^i}{\omega}^2 + \delta k^2 \rho(k) \, \sum_{i=0}^{j-1} \norm{ \nabla \vv_h^i }{\omega}^2$ in $\chi_j$.
Altogether, we arrive at
\begin{align}
{\rm \chi}_j \lesssim 1 + k \sum_{i=0}^{j-1} \norm{ \nabla \mm_h^{i}}{\omega}^2 + k\sum_{i=0}^{j-1} \norm{\hh_{h}^i}{\Omega}^2 \lesssim 1 + k \Sum{i=0}{j-1} {\rm \chi}_i . \notag
\end{align}
Arguing as for Lemma~\ref{lemma:energy}{\rm (ii)}, we get that ${\rm \chi}_j  $ is uniformly bounded for all $j=1,\dots,N$.
In order to bound the remaining terms from~\eqref{eq:energyabtps_ellg}, we sum~\eqref{eq:discreteenergy_ellg2} for $i=0,\dots,j-1$ and multiply by $k$. Recall from~\eqref{eq:auxiliary_estimate2} that $\norm{\dth{\mm}{i}}{\omega} \lesssim \norm{\vv_h^i}{\omega}$. This yields that
\begin{align}
\mu_0 k \sum_{i=0}^{j-1} \norm{\dth{\hh}{i}}{\Omega}^2 
&+ \norm{\sigma^{-1/2} \nabla \times \hh_h^{j}}{\Omega}^{2} \stackrel{\eqref{eq:discreteenergy_ellg2}}{\leq} 
\mu_0 k \sum_{i=0}^{j-1} \norm{\dth{\mm}{i}}{\omega}^2 + \norm{\sigma^{-1/2}  \nabla \times \hh_h^{0}}{\Omega}^{2} \notag \\
&\lesssim k \sum_{i=0}^{j-1} \norm{\vv_h^i}{\omega}^2 + \norm{\sigma^{-1/2}  \nabla \times \hh_h^{0}}{\Omega}^{2} \stackrel{\eqref{eq:discrete5}}{\lesssim} {\rm \chi}_j   + \norm{\sigma^{-1/2}  \nabla \times \hh_h^{0}}{\Omega}^{2} . \notag
\end{align}
Altogether, this proves {\rm (iii)} and concludes the proof.
\end{proof}

\begin{lemma}\label{lemma:extractsubsequences_ellg}
Under the assumptions of Theorem \ref{thm:maintheorem_ellg}{\rm(ii)}
consider the  postprocessed output~\eqref{eq:discretefunctions} of Algorithm \ref{alg:abtps_ellg}.
Then, there exist $\mm \in \HH{1}{\omega_T} \cap \L{\infty}{0,T;\HH{1}{\omega}}$ and $\hh \in H^1(0,T;\LL{2}{\Omega}) \cap \L{\infty}{0,T;\boldsymbol{H}({\rm curl},\Omega)}$ and a subsequence of 
each $\mm_{hk}^{\star} \in \{ \mm_{hk}^{=}, \mm_{hk}^-, \mm_{hk}^+, \linebreak \mm_{hk} \}$, $\hh_{hk}^{\star} \in \{ \hh_{hk}^{=}, \hh_{hk}^-, \hh_{hk}^+, \hh_{hk}, \hh_{hk}^{\Ttheta} \}$, and of $\vv_{hk}^-$ such that the following convergence statements~\ref{item:weakly_ellg1}--\ref{item:weakly_ellg11} hold true simultaneously for the same subsequence as $h,k\to0$:
\begin{enumerate}[label=\rm{(\roman*)}]
\item \label{item:weakly_ellg1} $\mm_{hk} \rightharpoonup \mm$ in $\HH{1}{\omega_T}$,
\item \label{item:weakly_ellg2} $\mm_{hk}^{\star} \stackrel{*}\rightharpoonup \mm$ in $\L{\infty}{0,T;\HH{1}{\omega}}$,
\item \label{item:weakly_ellg2b} $\mm_{hk}^{\star} \rightharpoonup \mm$ in $\L{2}{0,T;\HH{1}{\omega}}$,
\item \label{item:weakly_ellg3} $\mm_{hk}^{\star} \rightarrow \mm$ in $\LL{2}{\omega_T}$,
\item \label{item:weakly_ellg4} $\mm_{hk}^{\star} (t) \rightarrow \mm(t)$ in $\LL{2}{\omega}$\ for $t \in (0,T)$ a.e.,
\item \label{item:weakly_ellg5} $\mm_{hk}^{\star} \rightarrow \mm$ pointwise\ a.e. in $\omega_T$,
\item \label{item:weakly_ellg6} $\vv_{hk}^- \rightharpoonup \partial_t \mm$ in $\LL{2}{\omega_T}$,
\item \label{item:weakly_ellg7} $k \, \nabla \vv_{hk}^- \rightarrow 0$ in $\LL{2}{\omega_T}$,
\item \label{item:weakly_ellg8} $\hh_{hk} \rightharpoonup \hh$ in $\H{1}{0,T;\LL{2}{\Omega}}$,
\item \label{item:weakly_ellg9} $\hh_{hk}^{\star} \rightharpoonup \hh$ in $\LL{2}{\Omega_T}$,
\item \label{item:weakly_ellg10} $\hh_{hk}^{\star} \stackrel{*}\rightharpoonup \hh$ in $\L{\infty}{0,T;\boldsymbol{H}({\rm curl},\Omega)}$,
\item \label{item:weakly_ellg10b} $\hh_{hk}^{\star} \rightharpoonup \hh$ in $\L{2}{0,T;\boldsymbol{H}({\rm curl},\Omega)}$,
\item \label{item:weakly_ellg11} $\hh_{hk}^{\star} - \hh_{hk} \rightarrow 0$ in $\LL{2}{\Omega_T}$.
\end{enumerate}
\end{lemma}

\begin{proof}
Lemma~\ref{lemma:energy_ellg}(iii) yields uniform boundedness of 
$\mm_{hk} \in \HH{1}{\omega_T}$,
$\mm_{hk}^\star \in \L{\infty}{0,T;\linebreak\HH{1}{\omega}}$, 
$\hh_{hk} \in H^1(0,T;\LL{2}{\Omega})$, and 
$\hh_{hk}^\star \in \L{\infty}{0,T;\boldsymbol{H}({\rm curl},\Omega)}$.
The proofs of~\ref{item:weakly1}--\ref{item:weakly7} and~\ref{item:weakly_ellg8}--\ref{item:weakly_ellg10} follow as in~\cite{alouges2008a,akst2014} 
resp.~\cite{lt2013,lppt2015,bpp2015}. Finally,~\eqref{eq:thetainfnone_ellg} and Lemma~\ref{lemma:energy_ellg}(iii) and prove that
\begin{align}\label{eq:identifylimits}
\norm{\hh_{hk} - \hh_{hk}^{\star} }{\Omega_T}^2 \stackrel{\eqref{eq:discretefunctions}}{\lesssim} k^2 \sum_{i=0}^{N-1} \norm{\dth{\hh}{i}}{\Omega}^2
\stackrel{\eqref{eq:energyabtps_ellg}}{\lesssim} k \,\to\, 0 \quad \textrm{as } h,k \rightarrow 0.
\end{align}
This verifies~\ref{item:weakly_ellg11} and concludes the proof.
\end{proof}

\begin{proof}[Proof of Theorem~\ref{thm:maintheorem_ellg}{\rm (ii)}]
We prove that $(\mm,\hh)$ satisfies Definition~\ref{def:weak_ellg}\ref{item:weak_ellg1}--\ref{item:weak_ellg4}. Definition~\ref{def:weak_ellg}~\ref{item:weak_ellg1} follows as for LLG. Definition~\ref{def:weak_ellg}\ref{item:weak_ellg2} is an immediate consequence of Lemma~\ref{lemma:extractsubsequences_ellg}\ref{item:weakly_ellg8}--\ref{item:weakly_ellg10}.
Definition~\ref{def:weak_ellg}\ref{item:weak_ellg3} follows as in~\cite{lt2013,lppt2015,bpp2015} from Lemma~\ref{lemma:extractsubsequences_ellg}{\ref{item:weakly_ellg1}} and~{\ref{item:weakly_ellg8}}.

It remains to verify Definition~\ref{def:weak_ellg}\ref{item:weak_ellg4}: 
To that end, adopt the notation of the proof of Theorem~\ref{thm:maintheorem}(ii). In addition, let
$\JJ_h:\Hcurl{\Omega} \rightarrow \cchi_h$ denote the interpolation operator for first-order N\'ed\'elec elements of second type; see~\cite{nedelec1986}.
Let $\vvarphi\in \boldsymbol{C}^{\infty}(\overline{\omega_T})$, $\zzeta \in \boldsymbol{C}^{\infty}(\overline{\Omega_T})$ and $t\in[0,T]$. 
For $t \in [t_i,t_{i+1})$ and $i=0,\dots,N-1$, we test~\eqref{eq:abtps_ellg} with $\II_h(\mm_h^i \times \vvarphi(t)) \in \Kh{\mm_h^i}$ and $\JJ_h(\zzeta(t)) \in \cchi_h$ and integrate over $(0,T)$. With the definition of the postprocessed output~\eqref{eq:discretefunctions}, we obtain that
\begin{subequations}\label{eq:weak_variationalhk}
\begin{align}\label{eq:variationalhk1}
& I_{hk}^1 + I_{hk}^2 + \frac{\Cex}{2} \, I_{hk}^3 \notag \\
&=:  \int_0^T\prod{\weight_{M(k)}(\lambda_{hk}^-)\vv_{hk}^-}{\II_h(\mm_{hk}^-\times\vvarphi)}_{\omega} \d{t} + \int_0^T\prod{\mm_{hk}^-\times\vv_{hk}^-}{\II_h(\mm_{hk}^-\times\vvarphi)}_{\omega} \d{t} \notag \\
& \quad +\frac{\Cex}{2} \, k \, (1+\rho(k))\int_0^T\prod{\nabla\vv_{hk}^-}{\nabla\II_h(\mm_{hk}^-\times\vvarphi)}_{\omega} \d{t} \notag \\
&  \stackrel{\eqref{eq:tps2_ellg_variational1}}{=}
 -\Cex\int_0^T\prod{\nabla\mm_{hk}^-}{\nabla\II_h(\mm_{hk}^-\times\vvarphi)}_{\omega} \d{t}
+ \int_0^T\prod{\ppi_{hk}^-(\vv_{hk}^-;\mm_{hk}^-,\mm_{hk}^=)}{\II_h(\mm_{hk}^-\times\vvarphi)}_{\omega} \d{t} \notag \\
& \quad + \int_0^T\prod{\overline{\ff}_{k}}{\II_h(\mm_{hk}^-\times\vvarphi)}_{\omega} \d{t}
+ \int_0^T\prod{\Ppi_{hk}^-(\vv_{hk}^-;\mm_{hk}^-,\mm_{hk}^=)}{\II_h(\mm_{hk}^-\times\vvarphi)}_{\omega} \d{t} \notag \\
& \quad + \int_0^T\prod{\hh_{hk}^{\Ttheta}}{\II_h(\mm_{hk}^-\times\vvarphi)}_{\omega} \d{t} 
 =: - \Cex I_{hk}^4 + I_{hk}^5 + I_{hk}^6 + I_{hk}^7 + I_{hk}^8,
\end{align}
and
\begin{eqnarray}\label{eq:variationalhk2}
 - \mu_0 I_{hk}^9 &:= & -\mu_0 \Int{0}{T} \prod{\partial_t \mm_{hk}}{\JJ_h\zzeta}_{\omega} \d{t} \notag \\ 
&  \stackrel{\eqref{eq:tps2_ellg_variational2}}{=} &
\mu_0 \Int{0}{T} \prod{\partial_t \hh_{hk}}{\JJ_h\zzeta}_{\Omega} \d{t} + \Int{0}{T} \prod{\sigma^{-1} \nabla \times \overline{\hh}_{hk}}{\nabla \times (\JJ_h\zzeta)}_{\Omega} \d{t} \notag \\
& =: & \mu_0 I_{hk}^{10} + I_{hk}^{11}.
\end{eqnarray}
\end{subequations}
With Lemma~\ref{lemma:extractsubsequences_ellg}, convergence of the integrals $I_{hk}^1,\dots,I_{hk}^7$  in~\eqref{eq:variationalhk1} towards their continuous counterparts in the variational formulation~\eqref{eq:variational_ellg} follows the lines of the proof of Theorem~\ref{thm:maintheorem}{\rm (ii)}. 
Thus, we only consider the integrals $I_{hk}^8,\dots,I_{hk}^{11}$ from~\eqref{eq:weak_variationalhk}:
As in~\cite{lt2013,lppt2015}, we get from Lemma~\ref{lemma:extractsubsequences_ellg} and the convergence properties of $\II_h$ and $\JJ_h$ that
\begin{align*}
I_{hk}^{10} &
 \to \Int{0}{T} \prod{\partial_t \hh}{\zzeta}_{\Omega} \d{t} \quad \textrm{as } h,k \rightarrow 0.
\end{align*}
For the remaining terms, Lemma~\ref{lemma:extractsubsequences_ellg} and the convergence properties of $\JJ_h$ yield that
\begin{align*}
I_{hk}^9 \to 
\Int{0}{T} \! \prod{\partial_t \mm}{\zzeta}_{\omega} \d{t}, 
\quad
I_{hk}^8 \to 
\int_0^T \! \prod{\hh}{\mm\times\vvarphi}_{\omega} \d{t},
\quad
 I_{hk}^{11} \to 
\Int{0}{T} \! \prod{\sigma^{-1}\nabla \times \hh}{ \nabla \times \zzeta}_{\Omega} \d{t}
\end{align*}
as $h,k\to0$.
This concludes the proof.
\end{proof}

\begin{proof}[Proof of Theorem~\ref{thm:maintheorem_ellg}{\rm (iii)}]
It remains to verify that $(\mm,\hh)$ from Lemma~\ref{lemma:extractsubsequences_ellg} satisfies the energy estimate of Definition~\ref{def:weak_ellg}{\rm(v)}. To that end, let $\tau\in(0,T)$ be arbitrary and $j \in \{1,\dots,N\}$ such that $\tau \in [t_{j-1},t_{j})$. 
Adopt the notation of the proof of Theorem~\ref{thm:maintheorem}(iii).
For any $i = 0,\dots,j-1$, Lemma~\ref{lemma:energy_ellg}{\rm (i)} shows that
\begin{align*}
&\mathcal{E}_{\textrm{ELLG}}(\mm_{hk}(t_{i+1}),\hh_{hk}(t_{i+1}))-\mathcal{E}_{\textrm{ELLG}}(\mm_{hk}(t_{i}),\hh_{hk}(t_{i})) \notag \\
&\stackrel{\eqref{eq:strongerenergyestimate_ellg}}{=} 
\frac{\Cex}{2} \, \norm{\nabla\mm_h^{i+1}}{\omega}^2 
- \frac{\Cex}{2} \, \norm{\nabla\mm_h^i}{\Omega}^2 + 
\frac{1}{2} \, \norm{\hh_h^{i+1}}{\omega}^2 - \frac{1}{2} \, \norm{\hh_h^i}{\Omega}^2 \notag \\
&\quad -\frac{1}{2} \, \prod{\ppi(\mm_h^{i+1})}{\mm_h^{i+1}}_{\omega} + \frac{1}{2} \, \prod{\ppi(\mm_h^i)}{\mm_h^i}_{\omega} - \prod{\ff^{i+1}}{\mm_h^{i+1}}_{\omega} + \prod{\ff^i}{\mm_h^i}_{\omega} \notag \\
&\stackrel{\phantom{\eqref{eq:energy_LLG}}}{\le} - k \, \prod{\weight_{M(k)}(\mm_h^i)\vv_h^i}{\vv_h^i}_{\omega} - \frac{\Cex }{2} \, k^2\rho(k) \, \norm{\nabla\vv_h^i}{\omega}^2 + k \, \prod{\Ppi_h^i(\vv_h^i;\mm_h^i,\mm_h^{i-1})}{\vv_h^i}_{\omega} \notag \\
&\quad +
k \, \prod{\ppi_h^i(\vv_h^i;\mm_h^i,\mm_h^{i-1})}{\vv_h^i}_{\omega}
-\frac{1}{2} \, \prod{\ppi(\mm_h^{i+1})}{\mm_h^{j+i}}_{\omega} + 
\frac{1}{2} \, \prod{\ppi(\mm_h^i)}{\mm_h^i}_{\omega} \notag \\
&\quad + k \, \prod{\ff(t_{i+1/2})}{\vv_h^i}_{\omega} \!- \prod{\ff^{i+1}}{\mm_h^{i+1}}_{\omega} + 
\prod{\ff^i}{\mm_h^i}_{\omega}
 + k \, \prod{\vv_h^i}{\gentimeh{\hh}{i} \!- \midh{\hh}{i}}_{\omega}
\\& \quad
 + k \, \prod{\vv_h^i - \dth{\mm}{i}}{\midh{\hh}{i}}_{\omega} .
\end{align*}
With the stronger assumptions on $\ppi$ and $\ff$ from~\eqref{eq:strongerboundedness}, we can follow the lines of the proof of Theorem~\ref{thm:maintheorem}{\rm(iii)}. This leads to
\begin{align}\label{eq:proofthmenergyeq1_ellg}
&\mathcal{E}_{\textrm{ELLG}}(\mm_{hk}^+(\tau),\hh_{hk}^+(\tau)) - \mathcal{E}_{\textrm{ELLG}}(\mm_h^0,\hh_{h}^0) + \int_0^{t_j}\prod{\weight_{M(k)}(\mm_{hk}^-)\vv_{hk}^-}{\vv_{hk}^-}_{\omega}\,\d{t} 
\notag \\ 
& \quad + \int_0^{t_j} \prod{\partial_t \ff_k}{\mm_{hk}^-}_{\omega}\,\d{t} 
- \int_0^{t_j} \prod{\Ppi_{hk}(\vv_{hk}^-;\mm_{hk}^-,\mm_{hk}^=)}{\vv_{hk}^-}_{\omega}\,\d{t} \notag \\
&\lesssim k\int_0^{t_j}\norm{\vv_{hk}^-}{\Omega}^2\,\d{t} 
+ k\int_0^{t_j}\norm{\vv_{hk}^-}{\Omega}\norm{\nabla\vv_{hk}^-}{\Omega}\,\d{t} \notag \\
&\quad + \int_0^{t_j}| \prod{\ppi_{hk}(\vv_{hk}^-;\mm_{hk}^-,\mm_{hk}^=) - \ppi(\mm_{hk}^-)}{\vv_{hk}^-}_{\omega} |\,\d{t} 
+ \int_0^{t_j}|\prod{\overline{\ff_{k}}-\ff}{\vv_{hk}^-}_{\omega}|\,\d{t} \notag \\
& \quad 
+ \int_0^{t_j}|\prod{\hh_{hk}^{\Ttheta}-\overline{\hh}_{hk}}{\vv_{hk}^-}_{\omega}|\,\d{t} 
+  \int_0^{t_j}|\prod{\overline{\hh}_{hk}}{\vv_{hk}^- - \partial_t \mm_{hk}}_{\omega}|\,\d{t} .
\end{align}
The only crucial term is the last one on the right-hand side of~\eqref{eq:proofthmenergyeq1_ellg}: Recall the Sobolev embedding $\HH{1}{\omega} \subset \LL{4}{\omega}$. Together with Lemma~\ref{lemma:energy_ellg}{\rm (iii)}, this yields that
\begin{align}\label{eq:crucialterm_ellg}
\begin{split}
&\int_0^{t_j}|\prod{\overline{\hh}_{hk}}{\vv_{hk}^- - \partial_t \mm_{hk}}_{\omega}|\,\d{t} 
\le \norm{\overline{\hh}_{hk}}{\L{\infty}{0,T;\LL{2}{\Omega}}} \norm{\vv_{hk}^- - \partial_t \mm_{hk}}{\L{1}{0,T;\LL{2}{\omega}}}
\\& \qquad
\stackrel{\eqref{eq:auxiliary_estimate3}}\lesssim
k \norm{\vv_{hk}^-}{\L{2}{0,T;\LL{4}{\omega}}}^2 \lesssim k \norm{\vv_{hk}^-}{\L{2}{0,T;\HH{1}{\omega}}}^2
\lesssim 
 k(1+ h^{-2}) \, \norm{\vv_{hk}^-}{\LL{2}{\omega_T}}^2 \longrightarrow 0
\end{split}
\end{align}
as $h,k \rightarrow 0$,
where we have used an inverse inequality and the assumption $k = \mathbf{o}(h^2)$. Arguing by lower semicontinuity, we conclude the proof as for Theorem~\ref{thm:maintheorem_ellg}.
\end{proof}

\section{Numerical experiments} 
\label{section:numerics}

This section provides some numerical experiments for Algorithm~\ref{alg:abtps} and Algorithm~\ref{alg:abtps_ellg}. 
Our implementation is based on the C++/Python library Netgen/NGSolve~\cite{ngsolve}.
The computation of the stray field $\hh_{\mathrm{s}} = -\nabla u$ requires the approximation of the magnetostatic potential $u \in H^1(\R^3)$, which solves the full space transmission problem
\begin{subequations} \label{eq:magnetostatic}
\begin{alignat}{2}
-\Delta u &= -\diver \mm &\quad& \textrm{in } \omega, \\
-\Delta u &= 0 && \textrm{in } \R^3\setminus\overline{\omega}, \\
u^{\mathrm{ext}} - u^{\mathrm{int}} &= 0 && \textrm{on } \partial \omega, \\
(\nabla u^{\mathrm{ext}} - \nabla u^{\mathrm{int}})\cdot{\nn} &= -\mm\cdot\nn && \textrm{on } \partial \omega, \\
u(\boldsymbol{x}) &= \mathcal{O}(\vert \boldsymbol{x} \vert^{-1}) && \textrm{as } \vert \boldsymbol{x} \vert \to \infty.
\end{alignat}
\end{subequations}
Here, the superscript \emph{ext} (resp.\ \emph{int}) refers to the traces of $u$ on $\partial \omega$ with respect to the exterior domain $\R^3 \setminus \overline{\omega}$ (resp.\ the interior domain $\omega$), and $\nn$ is the outer normal vector on $\partial \omega$. 
Recall from~\cite{praetorius2004} that $\ppi(\mm):=-\nabla u$ gives rise to a self-adjoint operator 
$\ppi\in L(\L{2}{\omega},\L{2}{\R^d})$ which satisfies the stronger stability assumption~\eqref{eq:strongerboundedness}.

To discretize~\eqref{eq:magnetostatic}, we employ the hybrid FEM-BEM method from~\cite{fk1990}. 
We note that the latter satisfies~\eqref{eq:assumptions:pi} with strong convergence in~\eqref{eq:pihkconvergence};
see~\cite[Section~4.4.1]{bffgpprs2014} or~\cite[Section~4.1]{prs2016} for details.
This part of the code builds upon the open-source Galerkin boundary element library BEM++~\cite{sbaps2015}.
The arising linear systems are solved with GMRES (resp.\ with CG for the hybrid FEM-BEM approach) with tolerance $10^{-12}$. The implicit first time-step of Algorithm~\ref{alg:abtps} is solved by the fixpoint iteration used in the proof of Theorem~\ref{thm:maintheorem}{\rm(i)} which is stopped if $\norm{\eeta_h^\ell - \eeta_h^{\ell-1}}{\LL{2}{\omega}} \le 10^{-10}$. 


\begin{table}
 \begin{tabular}{p{20mm}||p{15mm}||p{15mm}|p{15mm}|p{15mm}|p{15mm}|p{15mm}}
   & {\tt TPS2}\newline absolute & 
   {\tt TPS2}\newline relative & 
   {\tt TPS1+EE}\newline relative & 
   {\tt TPS1+AB}\newline relative & 
   {\tt TPS2+EE}\newline relative & 
   {\tt TPS2+AB}\newline relative \\
  \hline\hline
  $k=0.0016$            &       $0.63$ &  $100\%$ & $24.35\%$ & $28.13\%$& $31.80\%$& $35.48\%$ \\
  \hline
  $k=0.0008$            &        $0.66$ &  $100\%$   & $24.37\%$ & $27.89\%$ & $31.49\%$ & $35.03\%$ \\  
  \hline
  $k=0.0004$            &        $0.69$ &  $100\%$  & $24.93\%$  & $28.56\%$  & $31.62\%$  & $35.20\%$  \\  
  \hline
  $k=0.0002$            &        $0.71$ &  $100\%$  & $24.42\%$ & $27.84\%$ & $31.15\%$ & $34.63\%$  \\  
  \hline
  $k=0.0001$            &         $0.66$ &  $100\%$  & $26.70\%$ &  $30.30\%$ &  $33.84\%$ &  $37.48\%$  \\  
  \hline
 \end{tabular}
 \caption{Experiment of Section~\ref{example1}: Average computational time for one time-step of the different integrators, where we provide the absolute time (in s) for the algorithm ({\tt TPS2}) from~\cite{akst2014} as well as the relative times of all other integrators.}
 \label{example1:table:average_duration}
\end{table}

\begin{figure}[ht]
\centering
\begin{tikzpicture}
\pgfplotstableread{plots/cumulative_isolated.dat}{\data}
\begin{axis}[
xlabel={physical time},
ylabel={computational time in \si{\second}},
width = 110mm,
legend style={
legend pos= north west},
xmax=5,
xmin=0,
ymin=0,
ymax=10e3,
]
\addplot[blue,ultra thick] table[x=time, y=TPS2fullImpl] {\data};
\addplot[red,dotted,ultra thick] table[x=time, y=TPS2ab] {\data};
\addplot[green,dashed,ultra thick] table[x=time, y=TPS2ee] {\data};
\addplot[orange,dash pattern=on 1pt off 3pt on 3pt off 3pt,,ultra thick] table[x=time, y=TPS1ab] {\data};
\addplot[cyan,dash pattern=on 3pt off 6pt on 6pt off 6pt,ultra thick] table[x=time, y=TPS1ee] {\data};
\legend{{\tt TPS1+EE},{\tt TPS1+AB},{\tt TPS2+EE},{\tt TPS2+AB},{\tt TPS2}}
\legend{{\tt TPS2},{\tt TPS2+AB},{\tt TPS2+EE},{\tt TPS1+AB},{\tt TPS1+EE}} %
\end{axis}
\end{tikzpicture}
\caption{Experiment of Section~\ref{example1}: Cumulative computational time of the different integrators for $k = 4\cdot10^{-4}$.}
\label{example1:fig:duration}
\end{figure}
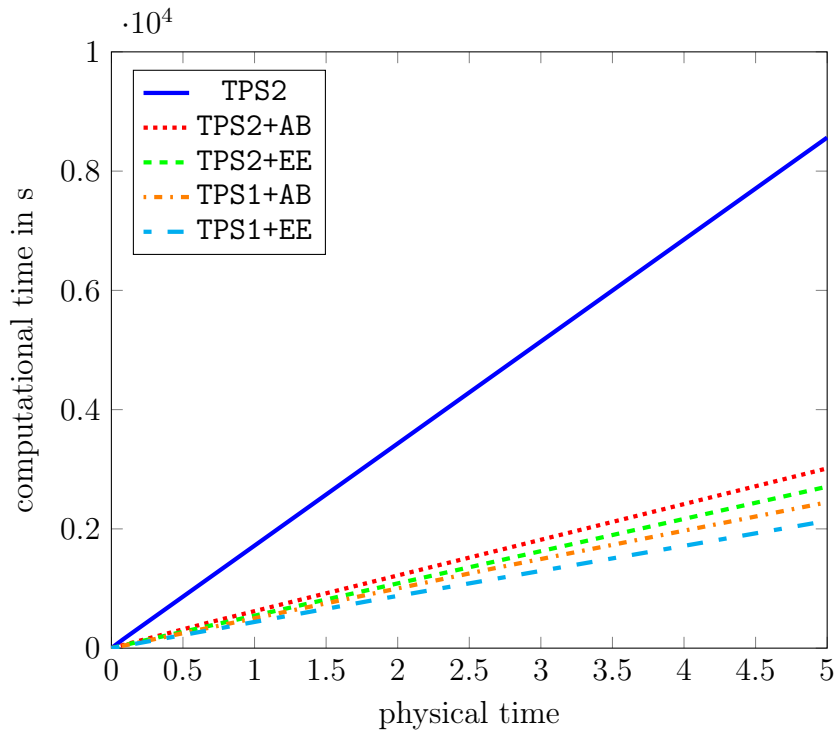
\begin{figure}[ht]
\centering
\begin{tikzpicture}
\pgfplotstableread{plots/H1_rates.dat}{\data}
\begin{loglogaxis}[
xlabel={Time-step size ($k$)},
ylabel={Error},
width = 110mm,
legend style={
legend pos= south west},
x dir=reverse
]
\addplot[only marks, blue, mark=x, mark size=5,ultra thick] table[x=k, y=TPS2fullImpl] {\data};
\addplot[only marks, red, mark=ball, ultra thick] table[x=k, y=TPS2ab] {\data};
\addplot[only marks, green, mark=triangle*, mark size=3,ultra thick] table[x=k, y=TPS2ee] {\data};
\addplot[only marks, orange, mark=square*, ultra thick] table[x=k, y=TPS1ab] {\data};
\addplot[only marks, cyan, mark=star, mark size=4,ultra thick] table[x=k, y=TPS1ee] {\data};
\addplot[purple,ultra thick] table[x=k, y expr={(\thisrow{k})/3}]{\data};
\addplot[black,ultra thick] table[x=k, y expr={(\thisrow{k}*\thisrow{k})*12}]{\data};
\node at (axis cs:1e-4,4e-5) [anchor=north east] {$\mathcal{O}(k)$};
\node at (axis cs:1e-4,1e-7) [anchor=east] {$\mathcal{O}(k^2)$};
\legend{{\tt TPS2},{\tt TPS2+AB},{\tt TPS2+EE},{\tt TPS1+AB},{\tt TPS1+EE}} %
\end{loglogaxis}
\end{tikzpicture}
\caption{Experiment of Section~\ref{example1}: Reference error $\max_j \norm{\mm_{hk_{\rm ref}}(t_j) - \mm_{hk}(t_j)}{\HH{1}{\omega}}$ for $k = 2^\ell \, k_{\rm ref}$ with $\ell  \in \{1,2,3,4,5\}$ and $k_{\rm ref} = 5\cdot10^{-5}$.}
\label{example1:fig:convergence_order}
\end{figure}

\subsection{Empirical convergence rates for LLG}
\label{example1}%
We aim to illustrate the accuracy and the computational effort of Algorithm~\ref{alg:abtps}.
We compare the following three strategies for the integration of the lower-order terms:
\begin{itemize}
\item {\tt TPS2}: fully implicit approach~\eqref{eq:llg_choice1} analyzed in~\cite{akst2014};
\item {\tt TPS2+AB}: Adams--Bashforth approach~\eqref{eq:llg_choice2} proposed and analyzed in this work;
\item {\tt TPS2+EE}: explicit Euler approach, where $\ppi_h^i(\vv_h^i;\mm_h^i,\mm_h^{i-1}) := \ppi_h^i(\mm_h^{i})$.
\end{itemize}
Besides Figure~\ref{fig:rhorates}, in which we compare the effect of different choices for the stabilization function $\rho$, we always employ the canonical choices~\eqref{eq:rhoM} for $M$ and $\rho$.
In addition, we consider the original (first-order) tangent plane scheme of~\cite{alouges2008a} with $\theta = 1/2$ (i.e., formal Crank--Nicolson-type integrator, which can be seen as a special case of Algorithm~\ref{alg:abtps} with $W_{M(k)} \equiv \alpha$ and $\rho \equiv 0$) together with the following two strategies for the integration of the lower-order terms:
\begin{itemize}
\item {\tt TPS1+AB}: Adams--Bashforth approach~\eqref{eq:llg_choice2};
\item {\tt TPS1+EE}: explicit Euler approach analyzed in~\cite{akt2012,bffgpprs2014}.
\end{itemize}
For the first time-step,
{\tt TPS2} is used
to preserve a possible second-order convergence.

To test the schemes, we use the model problem proposed in~\cite{prs2016}:
We consider the initial boundary value problem~\eqref{eq:llg} with $\omega = (0,1)^3$, $\mm^0 \equiv (1,0,0)$, $\alpha=1$, and $T=5$.
For the effective field~\eqref{eq:heffdef}, we choose $\ell_{\mathrm{ex}}=1$, a constant 
field $\ff \equiv (-2,-0.5,0)$, as well as an operator $\ppi$ which only involves the stray field, i.e., $\ppi(\mm) = -\nabla u$; see~\eqref{eq:magnetostatic}. 

We consider a fixed triangulation $\TT_h$ of $\omega$ generated by Netgen, which consists of \SI{3904}{} elements and 898 nodes (mesh size $h = 1/8$). The exact solution is unknown.
To compute the empirical convergence rates, we consider a reference solution $\mm_{hk_{\rm ref}}$ computed with {\tt TPS2}, using the above mesh and the time-step size $k_{\rm ref} = 5\cdot10^{-5}$.
Table~\ref{example1:table:average_duration} gives the average computational time per time-step for each of the considered five integrators.
In Figure~\ref{example1:fig:duration}, we plot the cumulative computational costs for the integration up to the final time $T$.
We observe a vast improvement if the lower-order terms (i.e., the stray field) are integrated explicitly in time.
The extended (first-order) tangent plane scheme from~\cite{alouges2008a} leads to the cheapest costs, since Algorithm~\ref{alg:abtps} involves additional computations in each time-step for $\lambda_h^i$ as well as the weighted mass matrix corresponding to $\prod{W_{M(k)}(\lambda_h^i)\,\cdot}{\cdot}_\omega$.

Figure~\ref{example1:fig:convergence_order} visualizes the experimental convergence order of the five integrators. As expected, {\tt TPS2} and {\tt TPS+AB} lead to second-order convergence in time. Essentially, both integrators even lead quantitatively to the same accuracy of the numerical solution. {\tt TPS2+EE} as well as {\tt TPS1+EE} yield first-order convergence, since the explicit Euler integration of the stray field is only first-order accurate. 
Differently from the classical $\theta$-method for linear second-order parabolic PDEs, due to the presence of the nodal projection, the original tangent plane scheme~\cite{alouges2008a} with $\theta=1/2$ (Crank--Nicolson-type {\tt TPS1+AB}) does not lead to any improvement of the convergence order in time (from first-order to second-order).

Overall, the proposed {\tt TPS2+AB} integrator appears to be the method of choice with respect to both computational time and empirical accuracy. 

Figure~\ref{fig:rhorates} shows the accuracy of the {\tt TPS2+AB} integrator for different choices of the stabilization $\rho(k) = k^{\delta}$ for $\delta \in \{0,0.1,0.2,0.3,0.4,0.5\}$, where $M(k) = |k \log k|^{-1}$ for all computations. As expected from Proposition~\ref{proposition:variational}, we observe an increase in convergence order for bigger values of $\delta$. In our example, however, the increase is stronger than the analysis predicts, e.g., for $\delta = 0.5$, Proposition~\ref{proposition:variational} predicts convergence order $3/2$, yet, Figure~\ref{fig:rhorates} shows order $2$.

\begin{figure}[ht]
\centering
\begin{tikzpicture}
\pgfplotstableread{plots/rho_rates.dat}{\data}
\begin{loglogaxis}[
xlabel={Time-step size ($k$)},
ylabel={Error},
width = 110mm,
legend style={
legend pos= south west},
x dir=reverse
]
\addplot[only marks, blue, mark=x, mark size=5,ultra thick] table[x=k, y=rhok00] {\data};
\addplot[only marks, red, mark=ball, ultra thick] table[x=k, y=rhok01] {\data};
\addplot[only marks, green, mark=triangle*, mark size=3,ultra thick] table[x=k, y=rhok02] {\data};
\addplot[only marks, orange, mark=square*, ultra thick] table[x=k, y=rhok03] {\data};
\addplot[only marks, cyan, mark=star, mark size=4,ultra thick] table[x=k, y=rhok04] {\data};
\addplot[only marks, magenta, mark=diamond*, mark size=3 , ultra thick] table[x=k, y=rhok05] {\data};
\addplot[purple,ultra thick] table[x=k, y expr={(\thisrow{k})/8}]{\data};
\addplot[magenta,ultra thick] table[x=k, y expr={(\thisrow{k}^1.5)}]{\data};
\addplot[black,ultra thick] table[x=k, y expr={(\thisrow{k}*\thisrow{k})*10}]{\data};
\node at (axis cs:1.5e-4,1.5e-5) [anchor=south west] {$\mathcal{O}(k)$};
\node at (axis cs:1.5e-4,4.2e-7) [anchor=south west] {$\mathcal{O}(k^{3/2})$};
\node at (axis cs:1e-4,1e-7) [anchor=north east] {$\mathcal{O}(k^2)$};
\legend{$\rho(k)=\revision{1}\hspace*{0.45cm}$, $\rho(k)=k^{0.1}$, $\rho(k)=k^{0.2}$, $\rho(k)=k^{0.3}$, $\rho(k)=k^{0.4}$, $\rho(k)=k^{0.5}$} %
\end{loglogaxis}
\end{tikzpicture}
\caption{
Experiment of Section~\ref{example1}: Reference error $\max_j \norm{\mm_{hk_{\textrm{ref}}}(t_j) - \mm_{hk}(t_j)}{\LL{2}{\Omega}}$ for $k = 2^\ell \, k_{\rm ref}$ with $\ell  \in \{1,2,3,4,5\}$ and $k_{\rm ref} = 5\cdot10^{-5}$.
}
\label{fig:rhorates}
\end{figure}

\subsection{Spintronic extensions of LLG}
\label{subsection:spintronic}
We consider two spintronic extensions of LLG, where the energy-based effective field is supplemented by terms which model the effect of the so-called spin transfer torque (STT)~\cite{slonczewski1996,berger1996}.
We show that these extended forms of LLG are covered by the abstract operator $\Ppi$ of Section~\ref{subsection:model_LLG} and perform physically relevant numerical experiments with the integrator {\tt TPS2+AB}; cf.~Section~\ref{example1}.
\subsubsection{Slonczewski model}
\label{subsubsec:slonczewski}
Since the discovery of the GMR effect~\cite{bbfvnpecfc1988,bgsz1989}, magnetic multilayers, i.e., systems consisting of alternating ferromagnetic and nonmagnetic sublayers, have become the subject of intense research in metal spintronics.
A phenomenological model to include the spin transfer torque in magnetic multilayer with current-perpendicular-to-plane injection geometry was proposed in~\cite{slonczewski1996}.
This model is covered by our framework for LLG~\eqref{eq:llg} by considering 
\begin{subequations}  \label{eq:slonczewski:Pi+DPi}
\begin{equation} \label{eq:slonczewski:Pi}
\boldsymbol{\Pi}: \LL{\infty}{\omega} \rightarrow \LL{2}{\omega},
\quad \boldsymbol{\Pi}(\vvarphi) := G(\vvarphi\cdot\boldsymbol{p}) \, \vvarphi \times \boldsymbol{p},
\end{equation}
where $\boldsymbol{p} \in \R^3$ with $\lvert \boldsymbol{p} \rvert = 1$ is constant, while the function $G: [-1,1] \to \R$ belongs to $C^1([-1,1])$.
Using the chain rule and the product rule, we obtain that
\begin{equation*}
\partial_t \boldsymbol{\Pi}(\vvarphi) 
= G'(\vvarphi\cdot\boldsymbol{p}) (\partial_t \vvarphi \cdot \boldsymbol{p}) \, \vvarphi \times \boldsymbol{p}
+ G(\vvarphi\cdot\boldsymbol{p}) \, \partial_t \vvarphi \times \boldsymbol{p}.
\end{equation*}
Hence, we recover the framework of Section~\ref{section:llg:algorithm} if we consider the operator
\begin{equation}
\DPi{\vvarphi}{\ppsi} 
= G'(\vvarphi\cdot\boldsymbol{p}) (\ppsi \cdot \boldsymbol{p}) \, \vvarphi \times \boldsymbol{p}
+ G(\vvarphi\cdot\boldsymbol{p}) \, \ppsi \times \boldsymbol{p}
\end{equation}
as well as the ``discrete'' operators
\begin{equation}
\boldsymbol{\Pi}_h(\vvarphi)
= \boldsymbol{\Pi}(\vvarphi)
\quad \text{and} \quad
\DPih{\vvarphi}{\ppsi}
= \DPi{\vvarphi}{\ppsi};
\end{equation}
\end{subequations}
see Remark~\ref{remark:algorithm:llg}\textrm{(ii)}.
We show that the resulting approximate operators $\boldsymbol{\Pi}_h^i$, defined in the implicit case by~\eqref{eq:llg_choice1b} and in the explicit case by~\eqref{eq:llg_choice2b}, and the piecewise constant time reconstruction $\boldsymbol{\Pi}_{hk}^-$ satisfy the assumptions~\eqref{eq:assumptions:Pi} of Theorem~\ref{thm:maintheorem}.

Let $i=0,\dots,N-1$ and consider arbitrary $\ppsi_h, \widetilde{\ppsi}_h \in \Vh$ and $\vvarphi_{h},\widetilde{\vvarphi}_{h} \in \Mh$.
In the implicit case~\eqref{eq:llg_choice1b}, the Lipschitz-type continuity~\eqref{eq:Pihlipschitz} follows from the estimate
\begin{equation*}
\begin{split}
& \norm{\Pih{i}{\ppsi_h}{\vvarphi_h}{\widetilde{\vvarphi}_h} - \Pih{i}{\widetilde{\ppsi}_h}{\vvarphi_h}{\widetilde{\vvarphi}_h}}{\omega}
\stackrel{\eqref{eq:llg_choice1b}}{=} \frac{k}{2} \norm{\DPih{\vvarphi_h}{\ppsi_h} - \DPih{\vvarphi_h}{\widetilde{\ppsi}_h}}{\omega} \\
& \quad \stackrel{\eqref{eq:slonczewski:Pi+DPi}}{=} \frac{k}{2} \norm{
G'(\vvarphi_h\cdot\boldsymbol{p}) [(\ppsi_h-\widetilde{\ppsi}_h) \cdot \boldsymbol{p}] \, \vvarphi_h \times \boldsymbol{p}
+ G(\vvarphi_h\cdot\boldsymbol{p}) \, (\ppsi_h-\widetilde{\ppsi}_h) \times \boldsymbol{p}
}{\omega}
\lesssim k \norm{\ppsi_h-\widetilde{\ppsi}_h}{\omega},
\end{split}
\end{equation*}
where the hidden constant in the last estimate depends only on $\norm{G}{C^1([-1,1])}$.
In the explicit case~\eqref{eq:llg_choice2b}, \eqref{eq:Pihlipschitz} is trivially satisfied.

We prove the stability estimate~\eqref{eq:Pihstability}: In the implicit case~\eqref{eq:llg_choice1b}, it holds that
\begin{equation*}
\begin{split}
& \prod{\Pih{i}{\ppsi_h}{\vvarphi_{h}}{\widetilde{\vvarphi}_{h}}}{\ppsi_h}_{\omega} 
\stackrel{\eqref{eq:llg_choice1b}}{=}
\prod{\Ppi_h(\vvarphi_{h}) + \frac{k}{2} \DPih{\vvarphi_{h}}{\ppsi_h}}{\ppsi_h}_{\omega}
\\&\quad
\stackrel{\eqref{eq:slonczewski:Pi+DPi}}{=}
\prod{G(\vvarphi_{h}\cdot\boldsymbol{p}) \vvarphi_{h} \times \boldsymbol{p}}{\ppsi_h}_{\omega}
+ \frac{k}{2} \prod{G'(\vvarphi_{h}\cdot\boldsymbol{p})(\ppsi_h \cdot \boldsymbol{p}) \vvarphi_{h} \times \boldsymbol{p}}{\ppsi_h}_{\omega} \\
& \qquad + \frac{k}{2} \prod{G(\vvarphi_{h}\cdot\boldsymbol{p})(\ppsi_h \times \boldsymbol{p})}{\ppsi_h}_{\omega} 
\\&\quad
\stackrel{\phantom{\eqref{eq:slonczewski:Pi+DPi}}}{\lesssim}
\norm{\ppsi_h}{\omega} \left( \norm{\vvarphi_{h}}{\omega} + k \norm{\ppsi_h}{\omega} \right).
\end{split}
\end{equation*}
In the explicit case~\eqref{eq:llg_choice2b}, it holds that
\begin{equation*}
\begin{split}
& \prod{\Pih{i}{\ppsi_h}{\vvarphi_{h}}{\widetilde{\vvarphi}_{h}}}{\ppsi_h}_{\omega} 
\stackrel{\eqref{eq:llg_choice2b}}{=}
\prod{\Ppi_h(\vvarphi_{h}) + 
\frac{1}{2} \DPih{\vvarphi_{h}}{\vvarphi_{h}} - \frac{1}{2} \DPih{\vvarphi_{h}}{\widetilde{\vvarphi}_{h}}}{\ppsi_h}_{\omega}
\notag \\
& \stackrel{\eqref{eq:slonczewski:Pi+DPi}}{=} 
\prod{G(\vvarphi_{h}\cdot\boldsymbol{p}) \vvarphi_{h} \times \boldsymbol{p}}{\ppsi_h}_{\omega}
+ \frac{1}{2} \prod{G'(\vvarphi_{h}\cdot\boldsymbol{p})(\vvarphi_{h} \cdot \boldsymbol{p}) \vvarphi_{h} \times \boldsymbol{p}}{\ppsi_{h}}_{\omega} \\
& \quad + \frac{1}{2} \prod{G(\vvarphi_{h}\cdot\boldsymbol{p})(\vvarphi_{h} \times \boldsymbol{p})}{\ppsi_{h}}_{\omega}
- \frac{1}{2} \prod{G'(\vvarphi_{h}\cdot\boldsymbol{p})(\widetilde{\vvarphi}_{h} \cdot \boldsymbol{p}) \vvarphi_{h} \times \boldsymbol{p}}{\ppsi_{h}}_{\omega} \\
& \qquad - \frac{1}{2} \prod{G(\vvarphi_{h}\cdot\boldsymbol{p})(\widetilde{\vvarphi}_h \times \boldsymbol{p})}{\ppsi_{h}}_{\omega} \\
& \stackrel{\phantom{\eqref{eq:slonczewski:Pi+DPi}}}{\lesssim}
\norm{\ppsi_h}{\omega} \left( \norm{\vvarphi_{h}}{\omega}+ \norm{\widetilde{\vvarphi}_{h}}{\omega} \right).
\end{split}
\end{equation*}

Finally, we prove the consistency property~\eqref{eq:Pihkconvergence} with strong convergence: To that end, let $\ppsi_{hk}$ in $\LL{2}{0,T;\Vh}$ and $\vvarphi_{hk}$, $\widetilde{\vvarphi}_{hk}$ in $\LL{2}{0,T;\Mh}$ with $\ppsi_{hk} \rightharpoonup \ppsi$ and $\vvarphi_{hk}, \widetilde{\vvarphi}_{hk} \rightarrow \vvarphi$ in $\LL{2}{\omega_T}$. As in~\cite{bffgpprs2014}, the Lebesgue dominated convergence theorem proves that
\begin{align*}
\Ppi_h(\vvarphi_{hk}) \rightarrow \Ppi(\vvarphi) 
\quad 
\textrm{in } \LL{2}{\omega_T} 
\quad 
\textrm{as } h,k \rightarrow 0.
\end{align*}
Moreover, we infer from $G \in C^1([-1,1])$ and $\norm{\vvarphi_{hk}}{\LL{\infty}{\omega_T}} = 1$, that
\begin{align*}
& \norm{ \DPih{ \vvarphi_{hk} }{ \vvarphi_{hk} - \widetilde{\vvarphi}_{hk} } }{\LL{2}{\omega}} + 
\norm{ \DPih{ \vvarphi_{hk} }{ k \ppsi_{hk} } }{\LL{2}{\omega}} \notag \\
& \quad \stackrel{\eqref{eq:slonczewski:Pi+DPi}}{\lesssim} 
\norm{ \vvarphi_{hk} - \widetilde{\vvarphi}_{hk} }{\LL{2}{\omega}} + 
k \norm{ \ppsi_{hk} }{\LL{2}{\omega}} \rightarrow 0 
\quad 
\textrm{as } h,k \rightarrow 0.
\end{align*}
Thus, for both, the implicit case~\eqref{eq:llg_choice1b} and the explicit case~\eqref{eq:llg_choice2b}, we get consistency~\eqref{eq:Pihkconvergence} with strong convergence.
\begin{figure}[ht]
\centering
\begin{subfigure}{0.46\textwidth}
\centering
\includegraphics[width=0.9\textwidth]{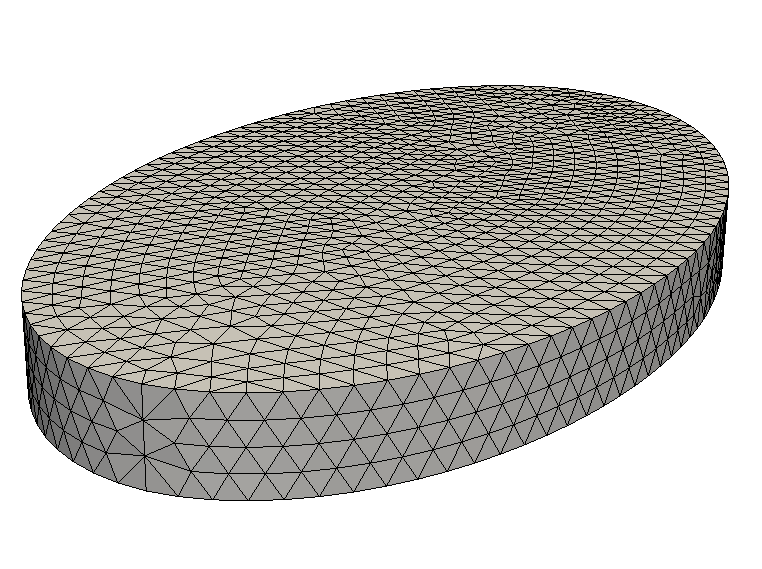}
\caption{Computational domain.}
\label{subfig:sublayer}
\end{subfigure}
\begin{subfigure}{0.49\textwidth}
\centering
\begin{tikzpicture}
\pgfplotstableread{plots/switching.dat}{\switching}
\pgfplotstableread{plots/energy.dat}{\energy}
\begin{axis}[
axis y line*=left,
xlabel={Time [\si{\nano\second}]},
ylabel={$\langle m_x \rangle$},
ylabel shift=-0.5cm,
width = 0.9\textwidth,
xmin=0,
xmax=3,
ymin=-1.1,
ymax=1.1,
]
\addplot[red,ultra thick] table[x=t, y=mx]{\switching};
\label{plot:mx}
\end{axis}
\begin{axis}[
axis y line*=right,
axis x line=none,
ylabel={Energy [\SI{e-18}{\joule}]},
ylabel shift=-0.2cm,
width=0.9\textwidth,
xmin=0,
xmax=3,
ymin=1,
ymax=11,
]
\addplot[blue,dashed,ultra thick] table[x=t, y=total]{\energy};
\end{axis} 
\end{tikzpicture}
\caption{Time evolutions of $\langle m_x \rangle$ (red solid line, left scale) and of the total energy (blue dashed line, right scale) during the switching process.}
\label{fig:switching}
\end{subfigure}
\caption{
Experiment of Section~\ref{subsubsec:slonczewski}:
Current-induced switching of a ferromagnetic film.}
\end{figure}
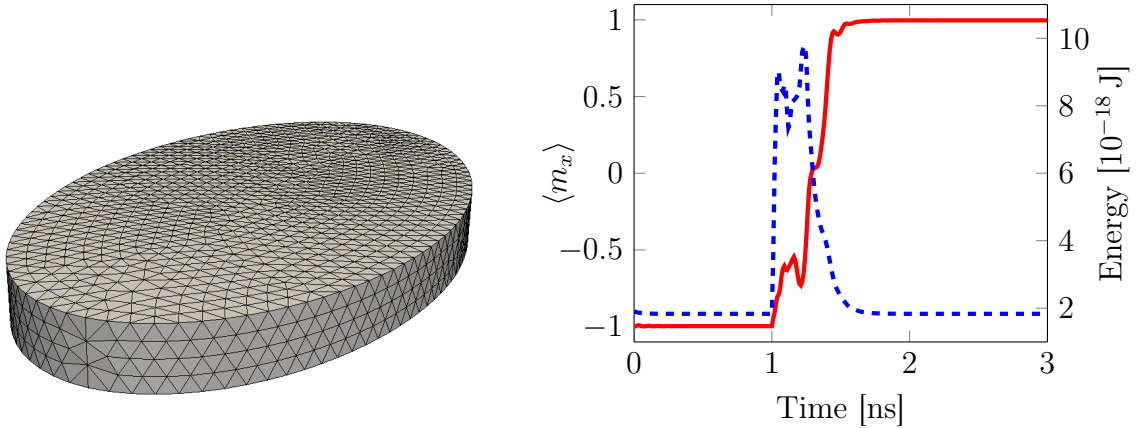

To test the effectivity of our algorithm, we simulate the writing process of an STT random access memory~\cite{hyybhyyshfnk2005} and reproduce the switching of a ferromagnetic film without external field.
The computational domain is an elliptic cylinder of height $d=$ \SI{10}{\nano\meter} (aligned with the $z$-direction) and elliptic cross section with semiaxis lengths $a=$ \SI{60}{\nano\meter} and $b=$ \SI{35}{\nano\meter} (parallel to the $xy$-plane); see Figure~\ref{subfig:sublayer}.
The film is supposed to be the free layer of a magnetic trilayer, which also includes a second ferromagnetic layer (the so-called fixed layer) with constant uniform magnetization $\boldsymbol{p} = (1,0,0)$ and a conducting nonmagnetic spacer.
For the material parameters, we choose the values of permalloy: damping parameter $\alpha=$~\num{0.1}, saturation magnetization $M_s=$ \SI{8.0e5}{\ampere\per\meter}, exchange stiffness constant $A=$ \SI{1.3e-11}{\joule\per\meter}, uniaxial anisotropy with constant $K=$ \SI{5.0e2}{\joule\per\meter\cubed} and easy exis $\boldsymbol{a} = (1,0,0)$.
With these choices, the effective field takes the form~\eqref{eq:heffdef} with
\begin{equation*}
\ell_{\mathrm{ex}} = \sqrt{\frac{2A}{\mu_0 M_s^2}} \approx \SI{5.7}{\milli\meter},
\quad \ppi(\mm)= \frac{2K}{\mu_0 M_s^2} (\boldsymbol{a}\cdot\mm)\boldsymbol{a} - \nabla u,
\quad \text{and} \quad \ff \equiv \0,
\end{equation*}
where $\mu_0 =$ \SI{4 \pi e-7}{\newton\per\ampere\squared} is the vacuum permeability and $u$ is the magnetostatic potential, solution of the transmission problem~\eqref{eq:magnetostatic}.

Starting from the uniform initial configuration $\mm^0 \equiv (-1,0,0)$, we solve~\eqref{eq:llg} with $\Ppi \equiv \0$ for \SI{1}{\nano\second} to reach a relaxed state.
Then, we inject a spin-polarized electric current with intensity $J_e =$ \SI{1e11}{\ampere\per\meter\squared} for \SI{1}{\nano\second}.
The resulting spin transfer torque is modeled by the operator~\eqref{eq:slonczewski:Pi}, where the function $G$ takes the phenomenological expression
\begin{equation*}
G(x) = \frac{\hbar J_e}{e \mu_0 M_s^2 d} \left[ \frac{(1+P)^3 (3+x)}{4 P^{3/2}} -4 \right]^{-1}
\quad \text{for all } x \in [-1,1];
\end{equation*}
see~\cite{slonczewski1996}.
Here, $\hbar =$ \SI{1.054571800e-34}{\joule\second} is the reduced Planck constant, $e =$ \SI{1.602176621e-19}{\coulomb} is the elementary charge, while $P = 0.8$ is the polarization parameter.
We solve~\eqref{eq:llg} with $\Ppi \equiv \0$ for \SI{1}{\nano\second} to relax the system to the new equilibrium.

For the spatial discretization, we consider a tetrahedral triangulation with mesh size \SI{3}{\nano\meter} (\num{15 885} elements, \num{3 821} nodes) generated by Netgen.
For the time discretization, we consider a constant time-step size of~\SI{0.1}{\pico\second} (\num{30 000} time-steps).

The time evolutions of the average $x$-component $\langle m_x \rangle$ of the magnetization in the free layer and of the total energy~\eqref{eq:energy_LLG} are depicted in Figure~\ref{fig:switching}.
Since the size of the film is well under the single-domain limit (see, e.g., \cite{adm2015}), the initial uniform magnetization configuration is preserved by the first relaxation process.
By applying a perpendicular spin-polarized current, the uniform magnetization of the free layer can be switched from $(-1,0,0)$ to $\boldsymbol{p}=(1,0,0)$.
The fundamental physics underlying this phenomenon is understood as the mutual transfer of spin angular momentum between the $\boldsymbol{p}$-polarized conduction electrons and the magnetization of the film.
During the switching process, the classical energy dissipation modulated by the damping parameter $\alpha$ is lost as an effect of the Slonczewski contribution; cf.\ the fourth term on the left-hand side of~\eqref{eq:strongerenergyestimate_LLG}.
The new state is also stable and is preserved by the final relaxation process.

\subsubsection{Zhang--Li model}\label{subsubsection:zhangli}
In~\cite{zl2004}, the authors derived an extended form of LLG to model the effect of an electric current flow on the magnetization dynamics in single-phase samples characterized by a current-in-plane injection geometry.
A similar equation was obtained in a phenomenological way in~\cite{tnms2005} for the description of the current-driven motion of domain walls in patterned nanowires.
Here, the operator takes the form
\begin{subequations}  \label{eq:zhangli:Pi+DPi}
\begin{equation} \label{eq:zhangli:Pi}
\boldsymbol{\Pi}: \HH{1}{\omega} \cap \LL{\infty}{\omega} \rightarrow \LL{2}{\omega},
\quad \boldsymbol{\Pi}(\vvarphi) := \vvarphi \times ( \uu \cdot \nabla ) \vvarphi + \beta ( \uu \cdot \nabla ) \vvarphi,
\end{equation}
where $\uu \in \LL{\infty}{\omega}$ and $\beta>0$ is constant.
The product rule yields that
\begin{equation*}
\partial_t \boldsymbol{\Pi}(\vvarphi) 
= \partial_t \vvarphi \times ( \uu \cdot \nabla ) \vvarphi
+ \vvarphi \times ( \uu \cdot \nabla ) \partial_t \vvarphi
+ \beta ( \uu \cdot \nabla ) \partial_t \vvarphi.
\end{equation*}
Hence, we recover the framework of Section~\ref{section:llg:algorithm} if we consider the operators
\begin{gather}
\DPi{\vvarphi}{\ppsi} 
= \ppsi \times ( \uu \cdot \nabla ) \vvarphi
+ \vvarphi \times ( \uu \cdot \nabla ) \ppsi
+ \beta ( \uu \cdot \nabla ) \ppsi, \\
\boldsymbol{\Pi}_h(\vvarphi)
= \boldsymbol{\Pi}(\vvarphi),
\quad \text{and} \quad
\DPih{\vvarphi}{\ppsi}
= \DPi{\vvarphi}{\ppsi}.
\end{gather}
\end{subequations}
Under the additional assumption
\begin{equation} \label{eq:extra_m0}
\sup_{h>0} \norm{\nabla\mm_h^0}{\LL{\infty}{\omega}} \leq C,
\end{equation}
we show that Theorem~\ref{thm:maintheorem}{\rm (a)}-{\rm (b)} still holds for the Adams--Bashforth approach from Remark~\ref{remark:algorithm:llg}{\rm (ii)}. To see this, recall Remark~\ref{remarks:llg}{\rm (iv)}. 
 
First, note that the Lipschitz-type continuity~\eqref{eq:Pihlipschitz} is trivially satisfied in the explicit case ($i=1,\dots,N-1$).
In the implicit case ($i=0$), it is not fulfilled for all $\ppsi_h, \widetilde{\ppsi}_h \in \Vh$ and $\vvarphi_{h},\widetilde{\vvarphi}_{h} \in \Mh$.
However, thanks to~\eqref{eq:extra_m0}, the desired inequality is satisfied in the specific situation of Remark~\ref{remarks:llg}{\rm (iv)}. Indeed, for arbitrary $\ppsi_h, \widetilde{\ppsi}_h \in \Vh$, it holds that
\begin{equation}\label{eq:lipschitz_zhangli}
\begin{split}
& \norm{\Pih{0}{\ppsi_h}{\mm_h^0}{\mm_h^{-1}} - \Pih{0}{\widetilde{\ppsi}_h}{\mm_h^0}{\mm_h^{-1}}}{\omega}
\stackrel{\eqref{eq:llg_choice1b}}{=} \frac{k}{2} \norm{\DPih{\mm_h^0}{\ppsi_h} - \DPih{\mm_h^0}{\widetilde{\ppsi}_h}}{\omega} \\
& \stackrel{\eqref{eq:zhangli:Pi+DPi}}{=} \frac{k}{2} \norm{(\ppsi_h-\widetilde{\ppsi}_h) \times ( \uu \cdot \nabla ) \mm_h^0
+ \mm_h^0 \times ( \uu \cdot \nabla ) (\ppsi_h-\widetilde{\ppsi}_h)
+ \beta ( \uu \cdot \nabla ) (\ppsi_h-\widetilde{\ppsi}_h)}{\omega} \\
& \stackrel{\phantom{\eqref{eq:zhangli:Pi+DPi}}}{\lesssim} k \, \norm{\ppsi_h-\widetilde{\ppsi}_h}{\HH{1}{\omega}},
\end{split}
\end{equation}
where the hidden constant depends on $C$ and $\norm{\uu}{\LL{\infty}{\omega}}$.

Next, we prove the stability estimate~\eqref{eq:Pihstability}: In the implicit case~\eqref{eq:llg_choice1b}, it holds that
\begin{equation*}
\begin{split}
& \prod{\Pih{i}{\ppsi_h}{\vvarphi_{h}}{\widetilde{\vvarphi}_{h}}}{\ppsi_h}_{\omega} 
\stackrel{\eqref{eq:llg_choice1b}}{=}
\prod{\Ppi_h(\vvarphi_{h}) + \frac{k}{2} \DPih{\vvarphi_{h}}{\ppsi_h}}{\ppsi_h}_{\omega} \\
& \quad
\stackrel{\eqref{eq:zhangli:Pi+DPi}}{=}
\prod{\vvarphi_{h} \times ( \uu \cdot \nabla ) \vvarphi_{h}}{\ppsi_h}_{\omega} + \beta \prod{( \uu \cdot \nabla ) \vvarphi_{h}}{\ppsi_h}_{\omega} 
\\
& \qquad  + \frac{k}{2} \prod{\vvarphi_{h} \times ( \uu \cdot \nabla ) \ppsi_h}{\ppsi_h}_{\omega}
+ \frac{\beta}{2} \, k \, \prod{( \uu \cdot \nabla ) \ppsi_h}{\ppsi_h}_{\omega} 
\\&\quad \stackrel{\phantom{\eqref{eq:zhangli:Pi+DPi}}}{\lesssim}
 \norm{\ppsi_h}{\omega} 
 \big( \norm{\nabla \vvarphi_{h}}{\omega} + k \norm{\nabla \ppsi_h}{\omega} \big).
\end{split}
\end{equation*}
In the explicit case~\eqref{eq:llg_choice2b}, it holds that
\begin{equation*}
\begin{split}
& \prod{\Pih{i}{\ppsi_h}{\vvarphi_{h}}{\widetilde{\vvarphi}_{h}}}{\ppsi_h}_{\omega} 
\stackrel{\eqref{eq:llg_choice2b}}{=}
\prod{\Ppi_h(\vvarphi_{h}) + 
\frac{1}{2} \DPih{\vvarphi_{h}}{\vvarphi_{h}} - \frac{1}{2} \DPih{\vvarphi_{h}}{\widetilde{\vvarphi}_{h}}}{\ppsi_h}_{\omega}
\notag \\
& \stackrel{\eqref{eq:zhangli:Pi+DPi}}{=} 
\prod{\vvarphi_{h} \times ( \uu \cdot \nabla ) \vvarphi_{h}}{\ppsi_h}_{\omega} 
+ \beta \prod{( \uu \cdot \nabla ) \vvarphi_{h}}{\ppsi_h}_{\omega} 
+ \prod{\vvarphi_{h} \times ( \uu \cdot \nabla ) \vvarphi_{h}}{\ppsi_h}_{\omega} \\
& \quad
+ \frac{\beta}{2} \, \prod{( \uu \cdot \nabla ) \vvarphi_{h}}{\ppsi_h}_{\omega} 
- \frac{1}{2} \, \prod{\widetilde{\vvarphi}_{h} \times ( \uu \cdot \nabla ) \vvarphi_{h}}{\ppsi_h}_{\omega}  \\
& \quad
- \frac{1}{2} \, \prod{\vvarphi_{h} \times ( \uu \cdot \nabla ) \widetilde{\vvarphi}_{h}}{\ppsi_h}_{\omega} 
-\frac{\beta}{2} \, \prod{( \uu \cdot \nabla ) \widetilde{\vvarphi}_{h}}{\ppsi_h}_{\omega} \\
& \stackrel{\phantom{\eqref{eq:zhangli:Pi+DPi}}}{\lesssim}
\norm{\ppsi_h}{\omega} \big( \norm{\nabla \vvarphi_{h}}{\omega}+ \norm{\nabla \widetilde{\vvarphi}_{h}}{\omega} \big).
\end{split}
\end{equation*}

To see~\eqref{eq:Pihkconvergence}, let $\zzeta \in \LL{2}{\omega_T}$. We show that
\begin{align}\label{eq:consistency_zhangli}
\Int{0}{T} \prod{ \Ppi_{hk}^-(\vv_{hk}^- ; \mm_{hk}^- , \mm_{hk}^= ) }{ \zzeta }_{\omega} \d{t} 
\rightarrow 
\Int{0}{T} \prod{ \Ppi(\mm) }{ \zzeta }_{\omega} \d{t} 
\quad \textrm{as } h,k \rightarrow 0.
\end{align}
Unwrapping the Adams--Bashforth approach from Remark~\ref{remark:algorithm:llg}{\rm (ii)}, we see that
\begin{align*}
& \Int{0}{T} \prod{ \Ppi_{hk}^-(\vv_{hk}^- ; \mm_{hk}^- , \mm_{hk}^= ) }{ \zzeta }_{\omega} \d{t}  \\
& 
= \Int{0}{k} \prod{ \Ppi_{hk}(\mm_{hk}^-) }{ \zzeta }_{\omega} \d{t}
+ \frac{k}{2} \Int{0}{k} \prod{ \DPih{\vv_{hk}^-}{\mm_{hk}^-} }{ \zzeta }_{\omega} \d{t} \\
& \qquad 
+ \Int{k}{T} \prod{ \Ppi_{hk}(\mm_{hk}^-) }{ \zzeta }_{\omega} \d{t} 
+ \frac12 \Int{k}{T} \prod{ \DPih{\mm_{hk}^-}{\mm_{hk}^-}}{ \zzeta }_{\omega} \d{t} 
- \frac12 \Int{k}{T} \prod{ \DPih{\mm_{hk}^-}{\mm_{hk}^=}}{ \zzeta }_{\omega} \d{t} \\
& =: T_{hk}^1 + \dots + T_{hk}^5.
\end{align*}
With the definitions from~\eqref{eq:zhangli:Pi+DPi}, the convergence properties from Lemma~\ref{lemma:extractsubsequences_ellg} prove that
\begin{align*}
T_{hk}^1 \rightarrow 0 
\quad \textrm{and} \quad 
T_{hk}^3 + T_{hk}^4 + T_{hk}^5 \rightarrow 
\Int{0}{T} \prod{\Ppi(\mm)}{\zzeta}_{\Omega} \d{t}
\quad \textrm{as } h,k \rightarrow 0.
\end{align*}
It remains to show that $T_{hk}^{2} \rightarrow 0$ as $h,k \rightarrow 0$. Note that
\begin{align*}
T_{hk}^2 
& \stackrel{\eqref{eq:zhangli:Pi+DPi}}{=} 
\frac{k}{2} \Int{0}{k} \prod{\vv_{hk}^- \times (\uu \cdot \nabla) \mm_{hk}^-}{\zzeta}_{\omega} \d{t} \notag \\
& \quad + 
\frac{k}{2} \Int{0}{k} \prod{\mm_{hk}^- \times (\uu \cdot \nabla) \vv_{hk}^-}{\zzeta}_{\omega} \d{t}
+
\frac{\beta k}{2} \Int{0}{k} \prod{(\uu \cdot \nabla) \vv_{hk}^-}{\zzeta}_{\omega} \d{t}.
\end{align*}
With the convergence properties from Lemma~\ref{lemma:extractsubsequences_ellg}, the last two terms of $T_{hk}^2$ go to $0$ as $h,k \rightarrow 0$. To deal with the first term, we derive from $\mm_{hk}^-(t)|_{[0,k)} = \mm_h^0$ that
\begin{align*}
&\frac{k}{2} \Int{0}{k} \prod{\vv_{hk}^- \times (\uu \cdot \nabla) \mm_{hk}^-}{\zzeta}_{\omega} \d{t}
 = 
\frac{k}{2} \Int{0}{k} \prod{\vv_{hk}^- \times (\uu \cdot \nabla) \mm_{h}^0}{\zzeta}_{\omega} \d{t} \\
& \qquad \lesssim k \, 
\norm{\uu}{\LL{\infty}{\omega}} 
\norm{\nabla \mm_{h}^0}{\LL{\infty}{\omega_T}} 
\norm{\vv_{hk}^-}{\omega_T}  
\norm{\zzeta}{\omega_T} 
\stackrel{\eqref{eq:extra_m0}}{\rightarrow} 0
\quad \textrm{as } h,k \rightarrow 0.
\end{align*}
This verifies~\eqref{eq:consistency_zhangli}.
\begin{figure}[ht]
\begin{subfigure}{0.48\textwidth}
\includegraphics[width=0.9\textwidth]{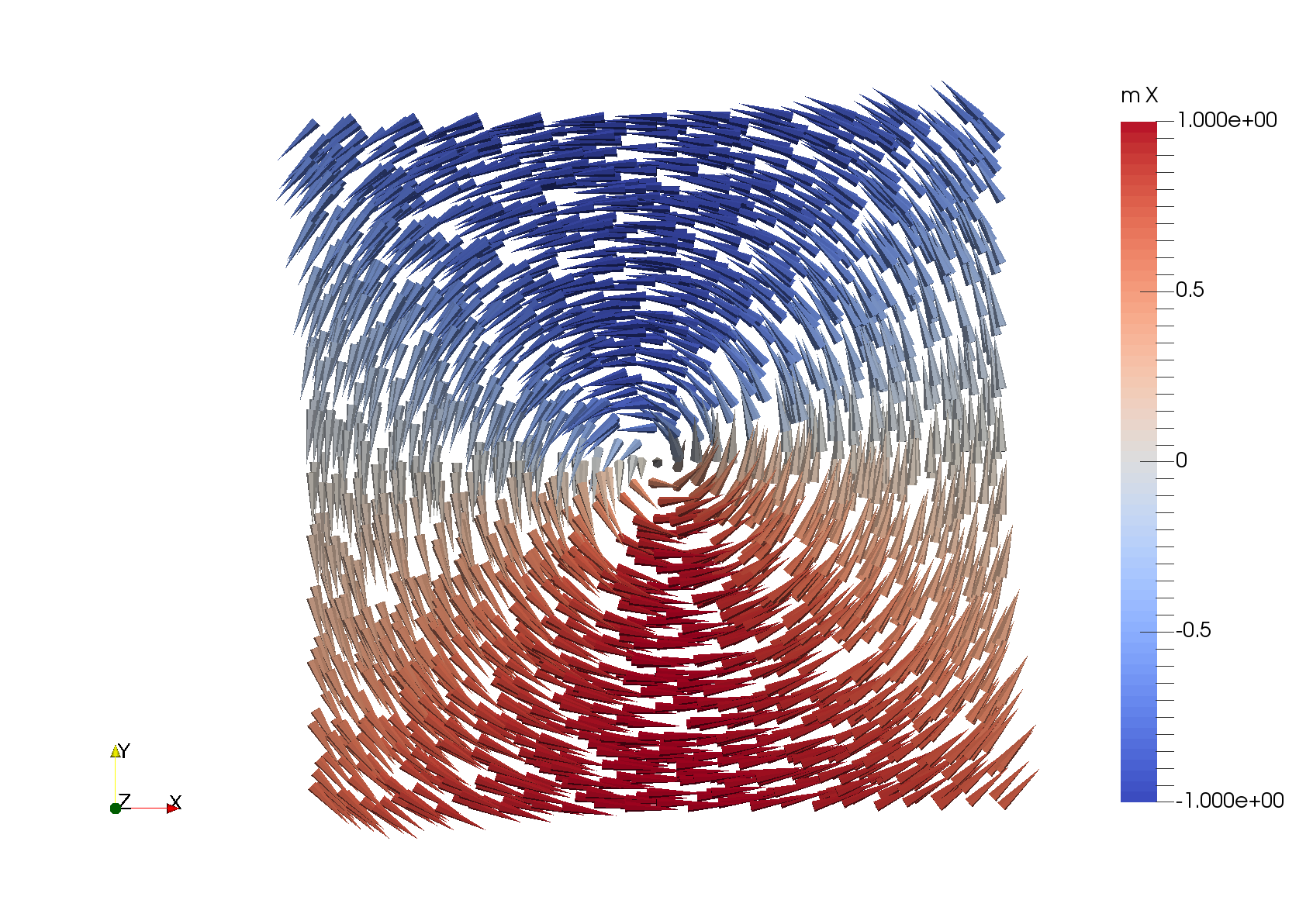}
\caption{Initial vortex state.}
\label{subfig:mumag5_initial}
\end{subfigure}\hspace*{\fill}
\begin{subfigure}{0.48\textwidth}
\includegraphics[width=0.9\textwidth]{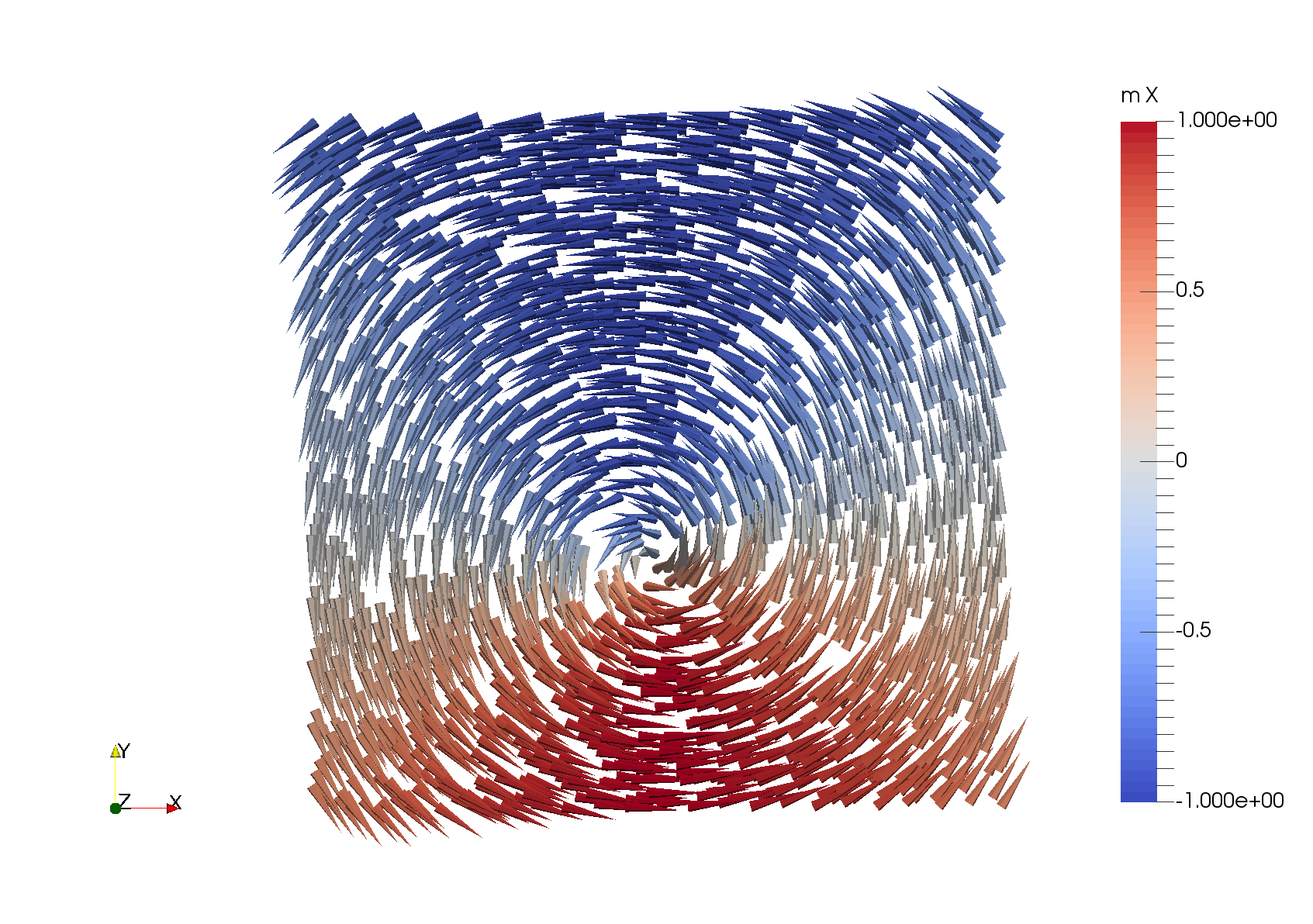}
\caption{Final vortex state.}
\label{subfig:mumag5_end}
\end{subfigure}
\caption{Experiment of Section~\ref{subsubsection:zhangli}: Vortex states in $\mu$MAG standard problem \#5.
The color scale refers to the $x$-component of the magnetization.}
\end{figure}
\begin{figure}[ht]
\centering
\begin{tikzpicture}
\pgfplotstableread{plots/sp5-oommf.dat}{\oommf}
\pgfplotstableread{plots/sp5-tps2ab.dat}{\tpsab}
\pgfplotstableread{plots/sp5-mpsab.dat}{\mpsab}
\begin{axis}[
width = 100mm,
xlabel={Time [\si{\nano\second}]},
xmin=0,
xmax=8,
ymin=-0.4,
ymax=0.4,
legend columns=2,
legend style={/tikz/column 2/.style={column sep=5pt}},
]
\addplot[blue,ultra thick] table[x=t, y=mx]{\tpsab};
\addplot[cyan,ultra thick] table[x=t, y=my]{\tpsab};
\addplot[purple,dashed,ultra thick] table[x=t, y=mx]{\oommf};
\addplot[red,dashed,ultra thick] table[x=t, y=my]{\oommf};
\addplot[brown,densely dotted,ultra thick] table[x=t, y=mx]{\mpsab};
\addplot[olive,densely dotted,ultra thick] table[x=t, y=my]{\mpsab};
\legend{
\texttt{TPS2+AB} $\langle m_x \rangle$,
\texttt{TPS2+AB} $\langle m_y \rangle$,
\texttt{OOMMF} $\langle m_x \rangle$,
\texttt{OOMMF} $\langle m_y \rangle$,
\texttt{MPS} $\langle m_x \rangle$,
\texttt{MPS} $\langle m_y \rangle$
}
\end{axis}
\end{tikzpicture}
\caption{Experiment of Section~\ref{subsubsection:zhangli}: Comparison of the results obtained for $\mu$MAG standard problem \#5 with \texttt{TPS2+AB}, \texttt{OOMMF}, and \texttt{MPS}.}
\label{example3:fig:average_xy}
\end{figure}
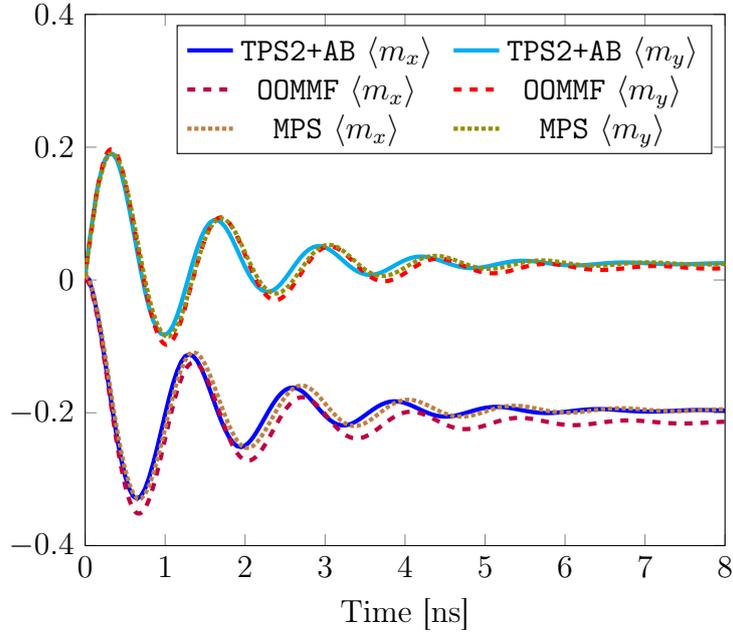

The extended LLG forms of~\cite{zl2004,tnms2005} are the subject of the $\mu$MAG standard problem \#5, proposed by the Micromagnetic Modeling Activity Group of the National Institute of Standards and Technology (NIST)~\cite{MUMAG}.
The sample under consideration is a permalloy film with dimensions \SI[scientific-notation=false]{100}{\nano\meter} $\times$ \SI[scientific-notation=false]{100}{\nano\meter} $\times$ \SI[scientific-notation=false]{10}{\nano\meter}, aligned with the $x$, $y$, and $z$ axes of a Cartesian coordinate system, with origin at the center of the film.
For the material parameters, we consider the same values as in Section~\ref{subsubsec:slonczewski}, except for the magnetocrystalline anisotropy, which is neglected, i.e., $K=0$.
The initial state is obtained by solving~\eqref{eq:llg} for the initial condition $\mm^0(x,y,z) = (-y,x,10) / \sqrt{x^2 + y^2 + 100}$ and $\Ppi \equiv \0$ for a sufficiently long time, until the system relaxes to equilibrium; see Figure~\ref{subfig:mumag5_initial}.
Given the spin velocity vector $\uu_T = (-72.17,0,0)$ \si{\meter\per\second} and the gyromagnetic ratio $\gamma_0 =$ \SI{2.21e5}{\meter\per\ampere\per\second}, we define $\Ppi$ by~\eqref{eq:zhangli:Pi} for $\uu=-\uu_T/(\gamma_0 M_s)$ and $\beta = 0.05$.
With the 
relaxed magnetization configuration as initial condition, we solve~\eqref{eq:llg} for \SI{8}{\nano\second}, which turns out to be a sufficiently long time to reach the new equilibrium; see Figure~\ref{subfig:mumag5_end}.

For the simulation, we 
consider a tetrahedral triangulation of the domain into \num{25 666} elements with maximal diameter \SI{3}{\nano\meter} (\num{5 915} nodes) generated by Netgen.
For the time discretization, we consider a constant time-step size of~\SI{0.1}{\pico\second} (\num{80 000} \emph{uniform} time-steps).

We compare our results with those obtained with the finite difference code \texttt{OOMMF}~\cite{dp1999} and with the implicit-explicit midpoint scheme of~\cite{prs2016} (\texttt{MPS}).
For~\texttt{OOMMF}, we consider the data downloadable from the $\mu$MAG homepage~\cite{MUMAG}, which refer to a uniform partition of the computational domain into \num{12 500} cubes with \SI{2}{\nano\meter} edge and \num{42350} \emph{adaptive} time-steps. In~\texttt{OOMMF}, the solution of~\eqref{eq:magnetostatic} for the stray field computation is based on a fast Fourier transform algorithm.
The results for~\texttt{MPS} refer to the same triangulation used for \texttt{TPS2+AB}, but are obtained with a \num{20} times smaller time-step size (\SI{5}{\femto\second}, i.e., \num{1600000} time-steps), which is necessary to ensure the well-posedness of the fixpoint iteration which solves the nonlinear system~\cite{prs2016}.

Figure~\ref{example3:fig:average_xy} shows the evolution of the averages $\langle m_x\rangle$ and $\langle m_y\rangle$ of the $x$- and $y$-component of $\mm$, respectively.
Despite the different nature of the methods, the results are in full qualitative agreement.


\subsection{Empirical convergence rates for ELLG}\label{example4}
We aim to illustrate the accuracy and the computational effort of Algorithm~\ref{alg:abtps_ellg}. We neglect $\mm$-dependent lower-order terms and dissipative effects, i.e., $\ppi_h^i(\vv_h^i;\mm_h^i,\mm_h^{i-1}) = \0 = \Ppi_h^i(\vv_h^i;\mm_h^i,\mm_h^{i-1})$, and compare the following four strategies:
\begin{itemize}
\item {\tt FC}: The fully-coupled approach $\hh_h^{i,\Ttheta} = \midh{\hh}{i}$ from~\eqref{eq:elle_ellg_choice1} for all $i = 0,\dots, N-1$;
\item {\tt DC-2}:~The decoupled second-order Adams--Bashforth approach from Remark~\ref{remark:adellg}{\rm (ii)};
\item {\tt DC-1}:~The decoupled 
explicit Euler approach $\hh_h^{i,\Ttheta} = \hh_h^i$ for all $i = 0,\dots, N-1$;
\item {\tt SF}: The simplified decoupled Adams--Bashforth approach from Remark~\ref{remark:adellg}{\rm (iv)}.
\end{itemize}
We always employ the canonical choices~\eqref{eq:rhoM} for $M$ and $\rho$. For the first time-step, we always use {\tt FC} to preserve a possible second-order convergence in time. For all implicit time-steps, the system~\eqref{eq:abtps_ellg} of Algorithm~\ref{alg:abtps_ellg} is solved by the fixpoint iteration from the proof of Theorem~\ref{thm:maintheorem_ellg}{\rm(i)} which is stopped if 
\begin{align*}
\norm{\eeta_h^\ell - \eeta_h^{\ell-1}}{\omega} +
\norm{\nnu_h^{\ell}-\nnu_h^{\ell-1}}{\Omega}
\le 10^{-10}.
\end{align*}
To test the schemes, we use a setting similar to that of Section~\ref{example1}: We choose $\omega = (-0.125,0.125)^3$, $\Omega = (0,1)^3$, and $T=7$. For the LLG-part~\eqref{eq:ELLG1}, we choose $\alpha = 1$, $\Cex = 1$, and set $\ff = (f_1,0,0)^T \in \boldsymbol{C}^1([0,T])$, where
\begin{align*}
f_1(t) :=
\begin{cases}
15 t^2 & \textrm{for } 0 \leq t \leq 1, \\
30 - 15(t - 2)^2 & \textrm{for } 1 < t \leq 2, \\
30 & \textrm{for } 2 < t \leq 4, \\
30 - 15(t - 4)^2 & \textrm{for } 4 < t \leq 5, \\
15 (t-6)^2 & \textrm{for } 5 < t \leq 6, \\
0 & \textrm{for } 6 < t \leq 7.
\end{cases}
\end{align*}
For the eddy current part~\eqref{eq:ELLG2}, we choose $\mu_0 = 1$, and $\sigma \in \L{\infty}{\Omega}$ with $\sigma|_{\omega} = 100$ and $\sigma|_{\Omega \setminus \overline{\omega}} = 1$.

We use a fixed triangulation $\Trian$ of $\Omega$ generated by Netgen, which resolves $\omega$ and satisfies
\begin{itemize}
\item $\max_{K \in \Trian|_{\omega}} h_K = 0.03$ on the sub-mesh $\Trian|_{\omega}$,
\item $\max_{K \in \Trian|_{\Omega \setminus \omega}} h_K = 0.125$ on the outer mesh $\Trian \setminus \Trian|_{\omega}$.
\end{itemize}
This yields \num{2 388} elements and \num{665} nodes for $\Trian|_{\omega}$ as well as \num{22 381} elements and \num{4 383} nodes for the overall mesh $\Trian$.

In all cases, the initial values $\mm_h^0$ and $\hh_h^0$ are obtained with setting $\ff = 0$ and relaxing with {\tt FC} the nodal interpolation of
\begin{align*}
\mm^0 = (-1,-1,-1) / \sqrt{3} 
\qquad 
\textrm{and} 
\qquad 
\hh^0(\xx) = 
\begin{cases} 
- \mm^0(\xx) & \textrm{for } \xx \in \overline{\omega}, \\
0 & \textrm{for } \xx \in  \overline{\Omega} \setminus \overline{\omega}.
\end{cases}
\end{align*}
The exact solution is unknown. To compute the empirical convergence rates, we consider a reference solution $\mm_{hk_{\rm ref}}$ computed with {\tt DC-2}, using the above mesh and the time-step size $k_{\textrm{ref}} = 2^{-14}$.

Figure~\ref{example1:fig:convergence_order_ellg} visualizes the experimental convergence orders of the four integrators. As expected, the fully-coupled approach {\tt FC} and the decoupled approach {\tt DC-2} lead to second-order convergence in time. As mentioned in Remark~\ref{remark:adellg}{\rm (iv)}, the simplified and fully linear approach {\tt SF} as well as {\tt DC-1} only lead to first-order convergence in time. Moreover, from the second time-step on the decoupled approach {\tt DC-2} is computationally as expensive as {\tt DC-1}. In contrast to that, the fully-coupled approach {\tt FC} employs a fixed-point iteration at each time-step and thus comes at the highest computational cost.

Overall, the decoupled approach {\tt DC-2} appears to be the method of choice with respect to both, computational time and empirical accuracy.

\begin{figure}[ht]
\centering
\begin{subfigure}{0.48\textwidth}
\scalebox{0.7}{
\begin{tikzpicture}
\pgfplotstableread{plots/errorsELLG_H1m_abRef.dat}{\data}
\begin{loglogaxis}[
xlabel={Time-step size ($k$)},
ylabel={Error},
width = 110mm,
legend style={
legend pos= south west},
x dir=reverse
]
\addplot[only marks, blue, mark=x, mark size=5,ultra thick] table[x=k, y=FC] {\data};
\addplot[only marks, red, mark=ball, ultra thick] table[x=k, y=DC-2] {\data};
\addplot[only marks, green, mark=triangle*, mark size=3,ultra thick] table[x=k, y=DC-1] {\data};
\addplot[only marks, orange, mark=square*, ultra thick] table[x=k, y=SF] {\data};
\addplot[purple,ultra thick] table[x=k, y expr={(\thisrow{k})/6}]{\data};
\addplot[black,ultra thick] table[x=k, y expr={(\thisrow{k}*\thisrow{k}/1.5)}]{\data};
\node at (axis cs:1e-4,1.5e-5) [anchor=north east] {$\mathcal{O}(k)$};
\node at (axis cs:1e-4,3e-8) [anchor=east] {$\mathcal{O}(k^2)$};
\legend{{\tt FC},{\tt DC-2},{\tt DC-1},{\tt SF}} %
\end{loglogaxis}
\end{tikzpicture}
}
\end{subfigure}
\begin{subfigure}{0.48\textwidth}
\scalebox{0.7}{
\begin{tikzpicture}
\pgfplotstableread{plots/errorsELLG_HCurlh_abRef.dat}{\data}
\begin{loglogaxis}[
xlabel={Time-step size ($k$)},
ylabel={Error},
width = 110mm,
legend style={
legend pos= south west},
x dir=reverse
]
\addplot[only marks, blue, mark=x, mark size=5,ultra thick] table[x=k, y=FC] {\data};
\addplot[only marks, red, mark=ball, ultra thick] table[x=k, y=DC-2] {\data};
\addplot[only marks, green, mark=triangle*, mark size=3,ultra thick] table[x=k, y=DC-1] {\data};
\addplot[only marks, orange, mark=square*, ultra thick] table[x=k, y=SF] {\data};
\addplot[purple,ultra thick] table[x=k, y expr={(\thisrow{k})*2.5}]{\data};
\addplot[black,ultra thick] table[x=k, y expr={(\thisrow{k}*\thisrow{k}*8)}]{\data};
\node at (axis cs:1e-4,2.6e-4) [anchor=north east] {$\mathcal{O}(k)$};
\node at (axis cs:1e-4,5e-7) [anchor=east] {$\mathcal{O}(k^2)$};
\legend{{\tt FC},{\tt DC-2},{\tt DC-1},{\tt SF}} %
\end{loglogaxis}
\end{tikzpicture}
}
\end{subfigure}
\caption{Experiment of Section~\ref{example4}: Reference error $\max_j ( \norm{\mm_{hk_{\rm ref}}(t_j) - \mm_{hk}(t_j)}{\HH{1}{\omega}})$ (left) and $\max_j ( \norm{\hh_{hk_{\rm ref}}(t_j) - \hh_{hk}(t_j)}{\Hcurl{\Omega}} )$ (right) for $k = 2^\ell \, k_{\rm ref}$ with $\ell  \in \{1,2,3,4,5\}$ and $k_{\rm ref} = 5\cdot10^{-5}$.}
\label{example1:fig:convergence_order_ellg}
\end{figure}
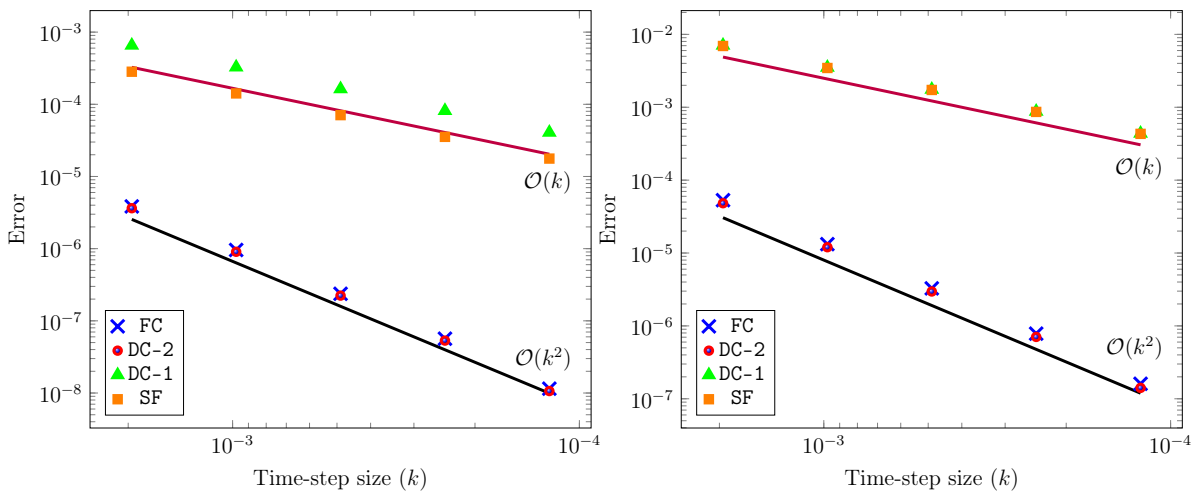


\bibliographystyle{alpha}
\bibliography{ref}




\end{document}